\documentclass[11pt,reqno]{amsart}
\usepackage{amsmath, color}
\usepackage{amssymb}
\usepackage{a4wide}
\usepackage{graphics}
\usepackage{epsfig}

\newtheorem{theorem}{Theorem}[section]
\newtheorem{lemma}[theorem]{Lemma}
\newtheorem{proposition}[theorem]{Proposition}

\newtheorem{remark}[theorem]{Remark}

\newtheorem{example}[theorem]{Example}

\newcommand{\tf}{\mathcal{F}}
\newcommand{\R}{\mathbb{R}}
\newcommand{\E}{\mathrm{E}}
\renewcommand{\P}{\mathrm{P}}
\newcommand{\ta}{\theta}

\makeatletter
 \@addtoreset{equation}{section}
 \makeatother
 
\allowdisplaybreaks
\usepackage{appendix}
\usepackage{hyperref}
\title{LAN property for an ergodic diffusion with jumps}
\date{\today}
\author[Arturo Kohatsu-Higa, Eulalia Nualart and Ngoc Khue Tran]{Arturo Kohatsu-Higa, Eulalia Nualart and Ngoc Khue Tran}

\address{Arturo Kohatsu-Higa, Department of Mathematical Sciences - Ritsumeikan University and Japan Science and Technology Agency, 1-1-1 Nojihigashi, Kusatsu, Shiga, 525-8577, Japan}
\email{khts00@fc.ritsumei.ac.jp}

\address{Eulalia Nualart, Dept. Economics and Business, Universitat Pompeu Fabra and Barcelona Graduate School of Economics, Ram\'on Trias Fargas 25-27, 08005 Barcelona, Spain}
\email{eulalia@nualart.es}

\address{Ngoc Khue Tran, Department of Mathematical Sciences - Ritsumeikan University and Japan Science and Technology Agency, 1-1-1 Nojihigashi, Kusatsu, Shiga, 525-8577, Japan}
\email{tnkhueprob@gmail.com}

\thanks{The first author was supported by KAKENHI grant 24340022 and a JST-CREST project Mathematical structure of complex financial products and infinite dimensional analysis. Second  author acknowledges support from the European Union programme FP7-PEOPLE-2012-CIG under grant agreement 333938. Third author acknowledges support from JST-CREST project and the program Vietnam Overseas Scholarship Program (Project 322) and wishes to thank Universit\'e Paris 13 for the hospitality where a part of this work was done}


\begin{document}
\maketitle
\begin{abstract} 
In this paper, we consider a multidimensional ergodic diffusion with jumps driven by a Brownian motion and a Poisson random measure associated with a compound Poisson process, whose drift coefficient depends on an unknown parameter. Considering the process discretely observed at high frequency, we derive the local asymptotic normality (LAN) property. 
\end{abstract}

\section{Introduction}

On a complete probability space $(\Omega, \tf, \P)$, we consider the $d$-dimensional process $X^{\theta}=(X_t^{\theta})_{t \geq 0}$ solution to the following stochastic differential equation (SDE) with jumps
\begin{equation}\label{c3eq1}\begin{split}
dX_t^{\theta}=b(\theta,X_t^{\theta})dt+\sigma(X_t^{\theta})dB_t+\int_{\mathbb{R}_0^d}c(X_{t-}^{\theta},z)\left(N(dt,dz)-\nu(dz)dt\right),
\end{split}
\end{equation}
where $X_0^{\theta}=x_0\in\R^d$, $\R_0^d:=\R^d\setminus\{0\}$, $B=(B_t)_{t \geq 0}$ is a $d$-dimensional Brownian motion, and $N(dt,dz)$ is a Poisson random measure in $(\R_{+}\times \R_0^d,\mathcal{B}(\R_{+}\times \R_0^d))$ independent of $B$, with intensity measure $\nu(dz)dt$ satisfying  $\lambda:=\int_{\R^d}\nu(dz)<\infty$. Let $\{\widehat{\mathcal{F}}_t\}_{t\geq 0}$ denote the natural filtration generated by $B$ and $N$.  The unknown parameter $\theta$ belongs to $\Theta$, a closed interval of $\R$. The coefficients $b: \Theta\times\R^d \to \R^d$, $\sigma: \R^d \to \R^d\otimes \R^d$ and $c: \R^d\times \R_0^d\to \R^d$ are measurable functions satisfying condition {\bf (A1)} below under which equation \eqref{c3eq1} has a unique $\{\widehat{\mathcal{F}}_t\}_{t \geq 0}$-adapted c\`adl\`ag solution $X^{\theta}$. We denote by $\P^{\theta}$ the probability law induced by $X^{\theta}$, and by $\E^{\theta}$ the expectation with respect to $\P^{\theta}$. 
For fixed $\theta_0\in \Theta$ and $n\geq 1$, we consider a discrete observation scheme at equidistant times $t_k=k \Delta_n$, $k \in \{0,...,n\}$ of the jump diffusion process $X^{\theta_0}$, which is denoted by $X^{n}=(X_{t_0}, X_{t_1},...,X_{t_n})$, where $\Delta_n\leq 1$ for all $n\geq 1$. We assume that the sequence of time-step sizes $\Delta_n$ satisfies the high-frequency and infinite horizon conditions: $\Delta_n\rightarrow 0$ and $n\Delta_n\rightarrow\infty$ as $n\rightarrow\infty$.

The aim of this paper is to prove the local asymptotic normality (LAN) property for estimators of $ \theta $ based on the observation $X^{n}$. As is well known, the LAN property is a fundamental concept in asymptotic theory of statistics, which was introduced by Le Cam \cite{LC60} and extended by Jeganathan \cite{JP82} to the local asymptotic mixed normality (LAMN) property. Initiated by Gobet \cite{G01}, some techniques of Malliavin calculus have recently been proved to be a powerful tool for the stochastic analysis of the log-likelihood ratio of diffusions. Concretely, Gobet \cite{G01} obtained the LAMN property from discrete observations at high frequency on the interval $[0,1]$ for multidimensional elliptic diffusion processes. For this purpose, the integration by parts formula of the Malliavin calculus is applied in order to obtain an explicit expression of the logarithm derivative of the transition density in terms of a conditional expectation involving the Skorohod integral. To treat the negligible terms, upper and lower Gaussian type bounds of the transition density are employed to show the convergence in probability to zero of sums of conditional expectations. In the same direction, the LAN property was established by Gobet \cite{G02} for multidimensional ergodic diffusions on the basis of discrete observations at high frequency on an increasing interval. Later on, Gobet and Gloter \cite{G3} obtained the LAMN property for integrated diffusions.

In the presence of jumps, several special cases have been studied. Precisely, the LAN property is established for some L\'evy processes whose transition density can be expressed in an explicit form. For instance, stable processes and normal inverse Gaussian L\'evy processes (see \cite{W, KM}). A\"it-Sahalia and Jacod \cite{AJ07} established the LAN property for a class of L\'evy processes involving a symmetric stable process, using a quasi-explicit representation of the density. 
{The LAN property for L\'evy processes observed discretely at low frequency can be found in \cite[Proposition 4.1, Lemma 2.12]{trabs}}. Recently, Kawai \cite{K13} deals with some cases where the  solution and transition density are semi-explicit. This implies that a Taylor expansion of the log-density with respect to the parameters can be obtained, which reduces the LAN property to a classical central limit theorem with independent increments and a residual term. This residual term depends strongly on estimates of the first and second derivatives of the logarithm of the density of the process, which can be treated using the integration by parts formula of Malliavin calculus. 

More recently, using a similar approach as in \cite{G01}, Cl\'ement et {\it al.} \cite{CDG14} establish the LAMN property for a stochastic process with jumps whose unknown parameters determine the jump structure. The number of jumps on the observation time interval is supposed to be deterministic and the corresponding jump times are given. As a consequence, upper and lower Gaussian type bounds for the transition density can be obtained and then used to treat the negligible terms. 

Later, Cl\'ement and Gloter \cite{CG15} prove the LAMN property for an SDE driven by a centered pure jump L\'evy process whose L\'evy measure is an $\alpha$-stable L\'evy measure near zero with $\alpha\in(1,2)$. For this, the authors verify the general sufficient conditions established by Jeganathan \cite{JP82}, which are essentially based on the $L^2$-regularity property of the transition density. Therefore, the upper and lower bounds for the density are not required for treating the negligible terms. The crucial point of the proof is the fact that using the time rescaling property of stable processes, the asymptotic behavior of the transition density and the derivative of its logarithm are completely determined by the density of a centered $\alpha$-stable L\'evy process and its derivative. In fact, as in \cite{G01} these quantities can be represented in terms of an expectation and a conditional expectation using the Malliavin calculus for jump processes developed by Bichteler, Gravereaux and Jacod \cite{BGJ87}. 

However, it seems that the validity of the LAN property for SDEs having a Brownian driver and a general jump structure has never been addressed in the literature. The first problem comes from the fact that the density function of the solution is not explicit in general. Secondly, the asymptotic behavior of the transition density and its logarithm derivative cannot be easily determined as in the cases of \cite{CG15} and \cite{K13}. As a consequence, the general sufficient conditions in \cite{JP82} cannot be used to show the LAN property for these general SDEs with jumps. Another problem is that the behavior of the transition density changes strongly due to the presence of jumps. In fact, one expects that the lower bound for the density of such SDEs with jumps will be controlled by the exponential behavior of the jump process, and that the upper bound will be controlled by the Gaussian behavior of the Wiener process. For instance, we consider the one-dimensional L\'evy process $(X_t^{x})_{t\geq 0}$ starting from $x\in \R$ defined by
\begin{equation*}
X_t^{x}=x+B_t+\sum_{i=1}^{N_t}Y_i,
\end{equation*}
where $B=(B_t)_{t\geq 0}$ is a standard Brownian motion, $N=(N_t)_{t\geq 0}$ is a Poisson process with intensity $\lambda>0$ independent of $B$, and $(Y_i)_{i\geq 0}$ are i.i.d. random variables independent of $B$ and $N$ with probability density $\frac{\varphi}{\lambda}$. Here, $\varphi(z)$ is the L\'evy density of the L\'evy process. It can be shown that when $\varphi$ is Gaussian, there exist constants $C_1, c_1, C, c>0$ such that for $0<t\leq 1$ and $\vert y-x\vert$ sufficiently large, the transition density $p(t,x,y)$ of $X_t^{x}$ satisfies
\begin{equation*}
C_1e^{-\lambda t}\exp\left(-c_1|y-x|\sqrt{\left|\ln\frac{|y-x|}{t}\right|}\right) \leq p(t,x,y)\leq \frac {C}{\sqrt{t}}\exp\left (
-c|y-x|\sqrt{\left|\ln\frac{|y-x|}{t}\right|}\right ),
\end{equation*}
and when $\varphi$ is exponential,
\begin{equation*}
C_1e^{-\lambda t}e^{-c_1|y-x|}\leq p(t,x,y)\leq \frac {C}{\sqrt{t}}e^{-c|y-x|}.
\end{equation*}
This shows that the upper and lower bounds of the density are of different characteristics making impossible to implement the argument in Gobet \cite{G01}, \cite{G02}. 

Our strategy is that in order to present the methodology used to prove the LAN property in the non-linear case \eqref{c3eq1}, it is essential to first well understand how the Malliavin calculus approach works in the presence of jumps, and how the Gaussian type estimates for the transition density conditioned on the jump structure can be derived and employed for a simple L\'evy process defined by
\begin{equation*}
X_t^{\theta,\sigma,\lambda}=x_0+\theta t +\sigma B_t + N_t -\lambda t,
\end{equation*}
where $N$ and $B$ are as above, and the parameters $\theta$, $\sigma$, and $\lambda$ are unknown. In \cite{KNT14}, we show the LAN property for this simple L\'evy process.

In this paper, our result uses the Malliavin calculus with respect to the Brownian motion initiated by Gobet \cite{G01}, in order to obtain an explicit expression of the logarithm derivative of the transition density. To deal with the expansion of the log-likelihood, one difficulty is the fact that the conditional expectations are computed under the probability measure $\P^{\theta(\ell)}$ coming from the Malliavin calculus, whereas the convergence is considered under the probability measure $\P^{\theta_0}\neq \P^{\theta(\ell)}$ where $ \theta(\ell) $ will be specified later on as a parameter value close to $ \theta_0 $. To solve this problem, we use Girsanov's theorem in order to change the measures (see Lemma \ref{c3Girsanov}). The technical Lemma \ref{c3lemma1} is given in order to measure the deviations of the Girsanov change of measure when the drift parameter changes. 

Let us mention that the goal of this paper is to define situations where the jump process will not ``deform" the Gaussian nature of the statistical experiment. As commented before, this cannot be achieved by simply obtaining upper and lower bounds for the transition density. Instead, we condition on the jump structure (number of jumps and amplitudes of jumps) and use large deviation type results which guarantee that the Gaussian nature of the statistical experiment will remain unchanged. Clearly, 
one can think of the reverse situation: That is, the case where the tails of the L\'evy process are heavy enough to perturb the Gaussian nature of the statistical experiment. 
Still, a central limit type theorem should be applicable and therefore one may believe that the LAN property should still hold if enough moment properties are assumed. More difficult to study are cases where the ellipticity condition is not satisfied. In general, it is challenging to ascertain validity of the LAN property. We leave as future research the study of these open problems. Here, to show the large deviation type estimates (see Lemma \ref{c3ordre}), lower and upper bounds for the transition density conditioned on the jump structure are strongly used.

This paper is organized as follows. In Section 2, we formulate the assumptions on equation \eqref{c3eq1} and state our main result in Theorem \ref{c3result}. Furthermore, some examples are given. Section 3 is devoted to give preliminary results needed for the proof of Theorem \ref{c3result}, such as an explicit expression for the logarithm derivative of the transition density using the Malliavin calculus. The proofs of these results are somewhat technical and are delayed to Appendices in order to provide the proof of our main result in a streamlined fashion. We prove our main result in Section 4.
 Finally, the proofs of some technical propositions and lemmas are presented in Section 5, where the upper bounds for the transition densities and the large deviation type estimates are obtained.

In this paper, we use $\overset{\P^{\theta}}{\longrightarrow}$ and $\overset{\mathcal{L}(\P^{\theta})}{\longrightarrow}$ to denote the convergence in probability and in law under $\P^{\theta}$, respectively. For $x\in\R^d$, $\vert x\vert$ denotes the Euclidean norm. $\vert A\vert$ denotes the Frobenius norm of the square matrix $A$, and tr($A$) denotes the trace. $^\ast$ denotes the transpose.
The compensated Poisson random measure is denoted by $\widetilde{N}(dt,dz):=N(dt,dz)-\nu(dz)dt$. Let $\widehat{Z}=(\widehat{Z}_t)_{t \geq 0}$ be a pure-jump L\'evy process associated with $N(dt,dz)$, i.e., $\widehat{Z}_t=\int_0^t\int_{\mathbb{R}_0^d}zN(ds,dz)$, for $t\geq 0$.

\section{Assumptions and main result}

We consider the following hypotheses on equation (\ref{c3eq1}).
\begin{list}{labelitemi}{\leftmargin=1cm}
\item[\bf(A1)] For any $\theta\in\Theta$, there exist a constant $L>0$ and a function $\zeta: \R_0^d\to \R_{+}$ of polynomial growth in $z$ with degree $m\geq 1$, i.e., $\zeta(z)\leq C(1+\vert z\vert^m)$ for some constant $C>0$, satisfying that $\int_{\R_0^d}\zeta^2(z)\nu(dz)<\infty$, such that for all $x,y \in \R^d$, $z\in \R_0^d$, 
\begin{align*}
&\vert b(\theta,x)-b(\theta,y)\vert+\vert\sigma(x)-\sigma(y)\vert\leq L\vert x-y\vert, \ \ \vert b(\theta,x)\vert \leq L\left(1+\vert x\vert\right),\\
&\vert c(x,z)-c(y,z)\vert\leq \zeta(z)\vert x-y\vert, \ \ \vert c(x,z)\vert\leq \zeta(z)(1+\vert x\vert).
\end{align*}
\vskip 12pt
\item[\bf(A2)] The diffusion matrix $\sigma$ satisfies an uniform ellipticity condition, that is, there exists a constant $c\geq 1$ such that for all $x, \xi\in\mathbb{R}^d$,
$$
\frac{1}{c}\vert\xi\vert^2\leq \vert\sigma(x)\xi\vert^2\leq c\vert\xi\vert^2.
$$
\vskip 12pt
\item[\bf(A3)] For all $(x,z)\in \R^d \times \R_{0}^d$ and $i\in\{1,\ldots,d\}$, $c_i(x,z)\neq 0$, and $c_i(x,0)=0$. Moreover, there exists a constant $C>0$ such that for all $z\in \R_0^d$,
$$
\inf_{x \in \R^d}\vert c(x,z)\vert\geq C \vert z\vert.
$$
\vskip 12pt
\item[\bf(A4)] The functions $b$, $\sigma$ and $c$ are of class $C^1$ w.r.t. $\theta$ and $x$. Each partial derivative $\partial_{\theta}b$, $\partial_{x_i}b$, $\partial_{x_i}\sigma$ and $\partial_{x_i}c$ is of class $C^1$ w.r.t. $x$.
Moreover, there exist positive constants $C, q, \epsilon$, independent of $\left(\theta, \theta_1, \theta_2, x, y, u,v, z\right)\in\Theta^3\times(\R^d)^4\times \R_0^d$ such that
\begin{itemize}
\item[\text{(a)}] \; $\vert \partial_{x_i}b(\theta,x) \vert + \vert \partial_{x_i}\sigma(x) \vert\leq C$, and $\vert \partial_{x_i}c(x,z) \vert\leq \zeta(z)$;
\item[\text{(b)}] \; $\vert h(\cdot,x)\vert\leq C\left(1+\vert x\vert^q\right)$ for $h(\cdot,x)=\partial_{\theta}b(\theta,x), \partial_{x_i,x_j}^2b(\theta,x), \partial_{x_i,\theta}^2b(\theta,x)$ or $\partial_{x_i,x_j}^2\sigma(x)$; 
\item[\text{(c)}] \; $\vert \partial_{x_i,x_j}^2c(x,z)\vert \leq C\zeta(z)\left(1+\vert x\vert\right)$; 
\item[\text{(d)}] \; $\vert \partial_{\theta}b(\theta_1,x)-\partial_{\theta}b(\theta_2,x)\vert \leq C\vert \theta_1 - \theta_2\vert^{\epsilon}\left(1+\vert x\vert^q\right)$; 
\item[\text{(e)}] \; $\vert \partial_{\theta}b(\theta,x)-\partial_{\theta}b(\theta,y)\vert \leq C\vert x-y\vert$; 
\end{itemize}
\vskip 12pt
\item[\bf(A5)] The process $X^{\theta_0}$ is ergodic, that is, there exists a unique invariant probability measure $\pi_{\theta_0}(dx)$ such that as $T\rightarrow\infty$,
\begin{equation*}
\dfrac{1}{T}\int_0^{T}g(X_t^{\theta_0})dt\overset{\P^{\ta_0}}{\longrightarrow}\int_{\mathbb{R}^d}g(x)\pi_{\theta_0}(dx),
\end{equation*}
for any $\pi_{\theta_0}$-integrable function $g:\R^d \to \R^{d'}$. Moreover, $\int_{\mathbb{R}^d}\vert x\vert^p\pi_{\theta_0}(dx)<\infty$, for any $p\geq 0$.
\vskip 12pt
\item[\bf(A6)] For any $p\geq 1$, $\int_{\R_0^d} \vert z\vert^p\nu(dz)<\infty$.
\vskip 12pt
\item[\bf(A7)] 
There exist constants $\rho_1>0$ and $\upsilon\in(0,\frac{1}{2})$ such that $\int_{\{\vert z\vert \leq\rho_1\Delta_n^{\upsilon}\}}\nu(dz)\rightarrow 0$ as $n\to\infty$.
\vskip 12pt
\item[\bf(A8)] 
\begin{itemize}
	\item[\text{(a)}] \; $\vert\det\nabla\psi(v)\vert\geq\eta(z)$, and  $\vert\nabla\psi^{-1}(v)u\vert\geq \frac{\vert u\vert}{\beta(z)}$,
	where $\psi(v)=f(\frac{v}{\sqrt{1-\vert v\vert^2}}+c(\frac{v}{\sqrt{1-\vert v\vert^2}},z))-f(\frac{x}{\sqrt{1-\vert x\vert^2}}+c(\frac{x}{\sqrt{1-\vert x\vert^2}},z))$, for $\vert v\vert<1$, $\vert x\vert<1$, and $\psi^{-1}$ is the inversion function of $\psi$. Here, $f(y)=\frac{y}{\sqrt{1+\vert y\vert^2}}$, for $y\in\R^d$, $\eta(z)=\frac{C}{\vert z\vert^3}{\bf 1}_{\{\vert z\vert>1\}}+C{\bf 1}_{\{\vert z\vert\leq 1\}}$, and $\beta(z)=C\vert z\vert^3{\bf 1}_{\{\vert z\vert>1\}}+C{\bf 1}_{\{\vert z\vert\leq 1\}}$.
	\item[\text{(b)}] \;
	The matrix $\nabla f\sigma$ satisfies an ellipticity assumption in $\R^d$. That is, for all $x\in\R^d$,
	$$
	\inf_{\xi\in\R^d: \vert \xi\vert=1}\vert\nabla f(x)\sigma(x)\xi\vert^2>0.
	$$ 
\end{itemize}
\end{list}

\noindent
{\bf Interpretation of the above hypotheses:}
Lipschitz continuity and linear growth conditions {\bf (A1)} on the coefficients $b$, $\sigma$ and $c$ ensure the existence of a unique c\`adl\`ag and adapted process $X^{\theta}=(X_t^{\theta})_{t \geq 0}$ solution to equation \eqref{c3eq1} on the filtered probability space $(\Omega,\mathcal{F},\{\widehat{\mathcal{F}}_t\}_{t\geq 0},\P)$ (see \cite[Theorem III.2.32]{JS03}). The drift coefficient is assumed to be unbounded, which will lead to a technical proof of upper bounds of the transition density conditioned on the jump structure by using a transformation of equation (\ref{c3eq1}) via the function $f$ defined in {\bf(A8)}\textnormal{(a)} (see Lemma \ref{c3lemma8}). The case of a bounded drift coefficient will be discussed in Subsection \ref{bounded}. Moreover, conditions on the jump coefficient {\bf (A1)} and {\bf (A3)} are needed in order to control the upper and lower bounds of the jump amplitudes of the L\'evy process (see the proof of Lemma \ref{c3ordre}). 

To be able to apply the Malliavin calculus, the uniform ellipticity condition {\bf(A2)} and regularity conditions {\bf(A4)}\textnormal{(a)}-\textnormal{(e)} on the coefficients are required. Condition {\bf(A6)} related to the finite moments of any order of the L\'evy measure is used to estimate the jump components.

Recall that ergodicity in the sense of {\bf(A5)} was shown by Masuda in \cite[Theorem 2.1]{M07} for a class of jump diffusion processes. Several examples of ergodic diffusion processes with jumps are given in \cite{M07, M08, Sh06}. Moreover, results on ergodicity and exponential ergodicity for diffusion processes with jumps have been established by Masuda \cite{M07, M08}. In addition, Kulik \cite{K09} provides a set of sufficient conditions for the exponential ergodicity of diffusion processes with jumps without Gaussian part and gives some examples. More recently, Qiao \cite{Q14} addresses the exponential ergodicity for SDEs with jumps with non-Lipschitz coefficients. However, ergodicity and exponentially ergodicity in these papers are understood in the sense of \cite{MT}, which are both stronger than in the sense of {\bf (A5)}.

Condition {\bf(A7)} controls the behavior of small jumps of the L\'evy process, which is determined by the mass of the L\'evy measure or jump size distribution around the origin. Indeed, this condition which is used in Lemma \ref{lemma6}, expresses the fact that the small jumps do not interfere with the Gaussian behavior of the transition density. This can be interpreted that the jump component is ``dominated" over by the Gaussian component in a small time interval. 
This is the main restriction which implies that the total L\'evy measure is finite and therefore we are dealing with the case where the jumps in \eqref{c3eq1} are given by a compound Poisson process. We have preferred this presentation in order to point out that the general problem for SDE driven by a general L\'evy process remains open.

Conditions {\bf(A1)}-{\bf(A2)} imply that the law of the discrete observation $(X_{t_0}^{\theta},X_{t_1}^{\theta},\ldots,X_{t_n}^{\theta})$ of the process $(X_t^{\theta})_{t \geq 0}$ has a density in $(\R^{d})^{n+1}$ that we denote by $p_n(\cdot;\theta)$. In particular, $p_n(\cdot;\theta_0)$ denotes the density of the random vector $X^n$.

In order to explain why we need the conditions {\bf{(A8)}}\textnormal{(a)} and {\bf{(A8)}}\textnormal{(b)}, note that in the classical Malliavin calculus one usually considers bounded smooth coefficients in order to prove that the density of $X_t^\theta$ has Gaussian upper bounds. In the present case, the drift coefficient has linear growth and therefore classical techniques do not apply. In \cite{G02}, a Girsanov's theorem approach is used but this argument does not work here due to the presence of jumps. Instead, we perform a change of variables $V_t^\theta=f(X_t^\theta)$ so that the random variable $V_t^\theta$ has bounded drift. Then Gaussian like estimates for $X_t^\theta$ can be obtained at the expense of these two conditions. Additionally, we do not require a squared exponential moment condition or that the coefficients have to be $ C^{1+\alpha} $ for some $ \alpha>0 $ like in \cite{G02}. We show in the next example that these conditions may be easily verified in the following four classes of L\'evy measures.

\begin{example}We assume $d=1$ and $c(x,z)=z$ in this example for simplicity.\\
\textnormal{1)}  Changing variables $u:=\frac{v}{\sqrt{1-v^2}}$, it is easy to check that
\begin{equation*}\begin{split}
\frac{8}{\left(\sqrt{z^2+4}+\vert z\vert\right)^3}\leq \psi'(v)\leq\frac{\left(\sqrt{z^2+4}+\vert z\vert\right)^3}{8},
\end{split}
\end{equation*}
for all $\vert v\vert<1$. Then, the inverse function theorem implies that $(\psi^{-1})'(v)=\frac{1}{\psi'(\psi^{-1}(v))}$. Therefore, condition {\bf(A8)}\textnormal{(a)} holds.
\vskip 5pt
\textnormal{2)} Then $f'(x)=(1+x^2)^{-\frac{3}{2}}>0$, for all $x\in\R$. Thus, condition {\bf{(A8)}}\textnormal{(b)} holds.
\vskip 5pt
\textnormal{3)} Class 1: Assume $\nu(dz)$ has a support on $\{\vert z \vert \geq C\}$ for some constant $C>0$. Then condition {\bf (A7)} holds for any $\rho_1, \upsilon>0$ since $\int_{\{\vert z\vert \leq\rho_1\Delta_n^{\upsilon}\}}\nu(dz)=0$, for $n$ sufficiently large.
\vskip 5pt
\textnormal{4)} Class 2: Assume that $\nu(dz)=\frac{dz}{\vert z\vert^{1+\alpha}}{\bf 1}_{\{\vert z\vert\leq 1\}}$, where $\alpha<0$. Condition {\bf(A6)} holds since for any $p\geq 1$, 
\begin{equation*}\begin{split}
\int_{\R_0} \vert z\vert^p\nu(dz)=\dfrac{2}{p-\alpha}<\infty.
\end{split}
\end{equation*}

Condition {\bf (A7)} holds for any $\rho_1>0$, $\upsilon>0$ since for $n$ sufficiently large,
\begin{equation*}\begin{split}
\int_{-\rho_1\Delta_n^{\upsilon}}^{\rho_1\Delta_n^{\upsilon}}\nu(dz)=-\dfrac{2}{\alpha}\rho_1^{-\alpha}\Delta_n^{-\alpha\upsilon},
\end{split}
\end{equation*}
which tends to zero as $n\to\infty$.
\vskip 5pt
\textnormal{5)} Class 3: Assume  that $\nu(dz)=C_1\varphi(z) {\bf 1}_{\{\vert z \vert > 1\}}dz+C_2\vert z\vert^{\kappa}{\bf 1}_{\{\vert z \vert \leq 1\}}dz$ for some constants $C_1, C_2>0$, where $\varphi$ is the standard Gaussian density and $\kappa>-1$. Condition {\bf(A6)} holds since for any $p\geq 1$, 
\begin{equation*}\begin{split}
\int_{\R_0}\vert z\vert^p\nu(dz)&=\frac{C_1}{\sqrt{2\pi}}\int_{\{\vert z \vert > 1\}}\vert z\vert^pe^{-\frac{\vert z\vert^2}{2}}dz+C_2\int_{\{\vert z \vert \leq 1\}}\vert z\vert^{p+\kappa}dz\\
&<C_1(p-1)!!+\dfrac{2C_2}{p+\kappa+1}<\infty.
\end{split}
\end{equation*}

Condition {\bf (A7)} holds for any $\rho_1>0$, $\upsilon>0$ since for $n$ sufficiently large,
\begin{equation*}\begin{split}
\int_{-\rho_1\Delta_n^{\upsilon}}^{\rho_1\Delta_n^{\upsilon}}\nu(dz)=\dfrac{2C_2}{\kappa+1}\rho_1^{\kappa+1}\Delta_n^{\upsilon\left(\kappa+1\right)},
\end{split}
\end{equation*}
which tends to zero as $n\to\infty$.
\vskip 5pt
\textnormal{6)} Class 4: Assume $\nu(dz)=C_1\varphi(z) {\bf 1}_{\{\vert z \vert > 1\}}dz+C_2\vert z\vert^{\kappa}{\bf 1}_{\{\vert z \vert \leq 1\}}dz$ for some constants $C_1, C_2>0$, where $\kappa>-1$ and $\varphi$ is the L\'evy measure of a symmetric gamma process, that is, $\varphi(z)=\alpha e^{-\beta \vert z \vert}\vert z \vert^{-1}$ for some $\alpha>0$ and $\beta>0$. Condition {\bf(A6)} holds since for any $p\geq 1$, 
\begin{equation*}\begin{split}
\int_{\R_0}\vert z\vert^p\nu(dz)&=C_1\alpha\int_{\{\vert z \vert > 1\}}\vert z\vert^{p-1}e^{-\beta \vert z \vert}dz+C_2\int_{\{\vert z \vert \leq 1\}}\vert z\vert^{p+\kappa}dz\\
&\leq\dfrac{C_1\alpha}{\beta^{p+1}}\int_{\{\vert z \vert > 1\}}\vert z\vert^{p-1}\dfrac{(p+1)!}{\vert z\vert^{p+1}}dz+C_2\int_{\{\vert z \vert \leq 1\}}\vert z\vert^{p+\kappa}dz\\
&<\dfrac{2C_1\alpha(p+1)!}{\beta^{p+1}}+\dfrac{2C_2}{p+\kappa+1}<\infty,
\end{split}
\end{equation*}
where we have used the inequality $e^{-x}<\frac{p!}{x^p}$, valid for any $x>0$ and $p\geq 1$.

Proceeding similarly as in example \textnormal{4)}, condition {\bf (A7)} holds for any $\rho_1>0$, $\upsilon>0$.
\end{example}

The main result of this paper is the following LAN property.
\begin{theorem}\label{c3result} Assume conditions {\bf(A1)}-{\bf(A8)}. Then, the LAN property holds for the likelihood at $\theta_0$ with rate of convergence $\sqrt{n\Delta_n}$ and asymptotic Fisher information $\Gamma(\theta_0)$. That is, for all $u\in\R$, as $n\to\infty$,
\begin{equation*}
\log\dfrac{p_n(X^{n};\theta_n)}{p_n(X^{n};\theta_0)}\overset{\mathcal{L}(\P^{\theta_0})}{\longrightarrow}u\mathcal{N}\left(0,\Gamma(\theta_0)\right)-\dfrac{u^2}{2}\Gamma\left(\theta_0\right),
\end{equation*}
where $\theta_n:=\theta_0+\frac{u}{\sqrt{n\Delta_n}}$, and $\mathcal{N}(0,\Gamma(\theta_0))$ is a centered Gaussian random variable with variance
$$
\Gamma\left(\theta_0\right)=\int_{\R^d}\left(\partial_{\theta}b(\theta_0,x)\right)^{\ast}
(\sigma\sigma^{\ast})^{-1}(x)\partial_{\theta}b(\theta_0,x)\pi_{\theta_0}(dx).
$$
\end{theorem}

\begin{remark}
To simplify the exposition, our result is established for the scalar parameter case. The multidimensional parameter case can be treated using a decomposition on the components of parameters and the same computations as the scalar parameter case. 
\end{remark}

\begin{remark}
In \cite{KNT14} we estimate the drift and diffusion parameters, and the jump intensity of a simple L\'evy process. Therefore, Theorem \ref{c3result} is a non-linear extension of the result in \cite{KNT14} when the unknown parameter appears only in the drift coefficient.
\end{remark}

\begin{remark}\label{afi} We recover the same formula for the asymptotic Fisher information $\Gamma(\theta_0)$ of ergodic diffusion processes without jumps obtained by Gobet in \cite[Theorem 4.1]{G02}. This comes from the fact that the jump component is dominated over by the Gaussian component in a small time interval, which will be seen in Section \ref{proof}.   
\end{remark}

\begin{remark}
When the LAN property holds at $\theta_0$, convolution and minimax theorems can be applied (see \cite{Haj72}, \cite{CY90}). On one hand, the asymptotically efficient estimators of the parameter $\theta_0$ are defined in terms of the optimal asymptotic variance $\Gamma(\theta_0)^{-1}$ and the optimal rate of convergence $\sqrt{n\Delta_n}$. On the other hand, one can derive the lower bound for the asymptotic variance of estimators given by $\Gamma(\theta_0)^{-1}$.\\
Let us mention that the question of asymptotic efficiency of the estimators based on discrete observations of ergodic diffusions with jumps was solved e.g. by Shimizu and Yoshida \cite{SY06} and Mai \cite{M14}. The estimators in \cite{SY06} are constructed from a contrast function which is based on a discretization of the likelihood function associated to the continuous observations of an ergodic diffusion with jumps whose drift and diffusion coefficients as well as its jump coefficient depend on unknown parameters. The drift parameter of the Ornstein-Uhlenbeck processes driven by a L\'evy process is dealt with in \cite{M14} where the estimators are constructed from a discretization of the time-continuous maximum likelihood estimators. These estimators are asymptotically efficient since their variance attains the lower bound given by $\Gamma(\theta_0)^{-1}$ with the optimal rate of convergence (see \cite[Theorem 2.1, Remark 2.2]{SY06}, \cite[Theorem 3.5, Remark 3.6]{M14} and \cite[Theorem 4.6]{M14}).
For other estimators for $ \theta $ based on quasi-likelihood estimators see Masuda \cite{M13} and Ogihara and Yoshida \cite{OY}.
\end{remark}

\begin{example}\label{c4example} \textnormal{1)} Consider the one-dimensional Ornstein-Uhlenbeck process with jumps defined as
$$
X_t^{\theta}=x_0-\theta\int_0^tX_s^{\theta}ds+\sigma B_t+\int_0^t\int_{\R_0}z\widetilde{N}(ds,dz),
$$
where $\theta>0$, $\sigma\in\R_0$. Assume that the L\'evy measure satisfies condition {\bf(A6)}. Then $X^{\theta}$ is ergodic in the sense of {\bf(A5)}. Furthermore, the invariant probability measure $\pi_{\theta}(dx)$ can be computed explicitly (see \cite[Theorem 17.5 and Corollary 17.9]{SK} and \cite[Theorem 2.6]{M07}), and satisfies $\int_{\mathbb{R}}\vert x\vert^p\pi_{\theta}(dx)<\infty$, for any $p\geq 0$. In particular,
$$
\Gamma(\theta)=\int_{\R}\frac{x^2}{\sigma^2}\pi_{\theta}(dx)=\dfrac{1}{2\theta}\left(1+\frac{1}{\sigma^2}\int_{\R_0}z^2\nu(dz)\right).
$$

Notice that conditions {\bf(A1)}-{\bf(A4)} and {\bf (A8)} hold. Assume further condition {\bf(A7)}. As a consequence of Theorem \ref{c3result}, the LAN property holds with rate of convergence $\sqrt{n\Delta_n}$ and asymptotic Fisher information $\Gamma(\theta_0)$.

\textnormal{2)} Consider the one-dimensional process
$$
X_t^{\theta}=x_0+\theta t+\sigma B_t+\int_0^t\int_{\R_0}z\widetilde{N}(ds,dz),
$$
where $\theta\in\R$ and $\sigma\in\R_0$. Assume that the L\'evy measure satisfies conditions {\bf(A6)} and {\bf(A7)}. Notice that conditions {\bf(A1)}-{\bf(A4)} and {\bf (A8)} hold. Then the LAN property holds with rate of convergence $\sqrt{n\Delta_n}$ and asymptotic Fisher information $\Gamma(\theta_0)=\frac{1}{\sigma^2}$. In this case, condition {\bf(A5)} is not needed since $\Gamma(\theta_0)$ can be obtained without using the ergodicity assumption, but thanks to the simple structure of the drift and diffusion coefficients (see \eqref{c3ergo} below).
\end{example}

As usual, constants will be denoted by $C$ or $c$ and they will always be independent of time and $\Delta_n$ but may depend on bounds for the set $\Theta$. They may change of value from one line to the next. 

\section{Preliminaries}
In this section, we introduce some preliminary results needed for the proof of Theorem \ref{c3result}. The proofs are given in the Appendix so that the reader can access the proof of the main result in the next section. 

In order to motivate the preliminary results to follow, recall that in order to deal with the log-likelihood ratio in Theorem \ref{c3result}, we may use the Markov property to rewrite the global likelihood function in terms of a product of transition densities and then apply a mean value theorem. We start as in Gobet \cite{G01} by applying the integration by parts formula of the Malliavin calculus on each interval $[t_k,t_{k+1}]$ to obtain an explicit expression for the logarithm derivative of the transition density. In order to avoid confusion with the observed process $X^{\theta}$, we introduce an extra probabilistic representation of $X^{\theta}$ for which the Malliavin calculus will be applied. That is, we consider on the same probability space $(\Omega, \tf, \P)$ the flow $Y^{\theta}(s,x)=(Y_t^{\theta}(s,x), t\geq s)$, $x\in\R^d$ on the time interval $[s,\infty)$ and with initial condition $Y_{s}^{\theta}(s,x)=x$ satisfying
\begin{equation}\label{c3flow}\begin{split}
Y_t^{\theta}(s,x)&=x+\int_{s}^tb(\theta,Y_u^{\theta}(s,x))du+\int_{s}^t\sigma(Y_u^{\theta}(s,x))dW_u\\
&\qquad+\int_{s}^t\int_{\R_0^d}c(Y_{u-}^{\theta}(s,x),z)\left(M(du,dz)-\nu(dz)du\right),
\end{split}
\end{equation}
where $W=(W_t)_{t\geq 0}$ is a Brownian motion, $M(dt,dz)$ is a Poisson random measure with intensity measure $\nu(dz)dt$ associated with a pure-jump L\'evy process $\widetilde{Z}=(\widetilde{Z}_t)_{t \geq 0}$, i.e., $\widetilde{Z}_t=\int_0^t\int_{\mathbb{R}_0^d}zM(ds,dz)$. Here, $\widetilde{Z}$ is an independent copy of $\widehat{Z}$, and $\widetilde{M}(dt,dz):=M(dt,dz)-\nu(dz)dt$ denotes the compensated Poisson random measure. The processes $(B, N, W, M)$ are mutually independent. In particular, we write $Y_t^{\theta}\equiv Y_t^{\theta}(0,x_0)$, for all $t\geq 0$. That is, 
\begin{equation}\label{c3eq1rajoute}\begin{split}
Y_t^{\theta}&=x_0+\int_{0}^tb(\theta,Y_u^{\theta})du+\int_{0}^t\sigma(Y_u^{\theta})dW_u+\int_{0}^t\int_{\R_0^d}c(Y_{u-}^{\theta},z)\widetilde{M}(du,dz).
\end{split}
\end{equation}

We will apply the Malliavin calculus on the Wiener space induced by $W$. Let $D$ and $\delta$ denote the Malliavin derivative and the Skorohod integral w.r.t. $W$ on each interval $[t_k, t_{k+1}]$, respectively. We denote by $\mathbb{D}^{1,2}$ the space of random variables differentiable in the sense of Malliavin, and by $\textnormal{Dom}\ \delta$ the domain of $\delta$. Notice that the Malliavin calculus adapted to our framework is introduced, for instance, in \cite{P08}. Recall that for a differentiable random variable $F\in\mathbb{D}^{1,2}$, its Malliavin derivative is denoted by $DF=(D^1F,\ldots,D^dF)$, where $D^i$ is the Malliavin derivative in the $i$th direction $W^i$ of the Brownian motion $W=(W^1,\ldots,W^d)$, for $i\in\{1,\ldots,d\}$. For a $\R^d$-valued process $U=(U^1,\ldots,U^d)\in \textnormal{Dom}\ \delta$, the Skorohod integral of $U$ is defined as $\delta(U)=\sum_{i=1}^{d}\delta^i(U^i)$, where $\delta^i$ denotes the Skorohod integral w.r.t. $W^i$.

For any $k \in \{0,...,n-1\}$, under conditions {\bf(A1)}, {\bf(A2)} and {\bf(A4)}\textnormal{(a)}-\textnormal{(c)}, the process $(Y_t^{\theta}(t_k,x), t\in [t_k,t_{k+1}])$ is differentiable w.r.t. $x$ and $\theta$, and we denote by $(\nabla_{x}Y_t^{\theta}(t_k,x), t\in [t_k,t_{k+1}])$ and $(\partial_{\theta}Y_t^{\theta}(t_k,x), t\in [t_k,t_{k+1}])$ the Jacobian matrix and vector, respectively (see Kunita \cite{K97}). These processes are the solutions to the linear equations
\begin{align}
&\nabla_xY_t^{\theta}(t_k,x)=\textup{I}_d+\int_{t_k}^t\nabla_xb(\theta,Y_s^{\theta}(t_k,x))
\nabla_xY_s^{\theta}(t_k,x)ds \label{partialx}\\
&+\sum_{i=1}^{d}\int_{t_k}^t\nabla_x\sigma_i(Y_s^{\theta}(t_k,x))\nabla_x
Y_s^{\theta}(t_k,x)dW_{s}^i+\int_{t_k}^t\int_{\mathbb{R}_0^d}
\nabla_xc(Y_{s-}^{\theta}(t_k,x),z)\nabla_x
Y_s^{\theta}(t_k,x)\widetilde{M}(ds,dz),\notag\\
&\partial_{\theta}Y_t^{\theta}(t_k,x)=\int_{t_k}^t\left(\partial_{\theta}b(\theta,Y_s^{\theta}(t_k,x))
+\nabla_xb(\theta,Y_s^{\theta}(t_k,x))\partial_{\theta}Y_s^{\theta}(t_k,x)\right)ds \label{partialtheta}\\
&+\sum_{i=1}^{d}\int_{t_k}^t
\nabla_x\sigma_i(Y_s^{\theta}(t_k,x))\partial_{\theta}Y_s^{\theta}(t_k,x)dW_{s}^i+\int_{t_k}^t\int_{\mathbb{R}_0^d}
\nabla_xc(Y_{s-}^{\theta}(t_k,x),z)\partial_{\theta}
Y_s^{\theta}(t_k,x)\widetilde{M}(ds,dz),\notag
\end{align}
where $\sigma_1,...,\sigma_d:\R^d \rightarrow \R^d$ denote the columns of the matrix $\sigma$.

Moreover, the random variables $Y_t^{\theta}(t_k,x)$, $\nabla_xY_t^{\theta}(t_k,x)$, $(\nabla_xY_t^{\theta}(t_k,x))^{-1}$ and $\partial_{\theta}Y_t^{\theta}(t_k,x)$ belong to $\mathbb{D}^{1,2}$ for any $t\in[t_k,t_{k+1}]$ (see \cite[Theorem 3]{P08}). On the other hand, the Malliavin derivative $D_sY_t^{\theta}(t_k,x)$ satisfies the following linear equation
\begin{align*}
&D_sY_t^{\theta}(t_k,x)=\sigma(Y_s^{\theta}(t_k,x))+\int_{s}^t\nabla_xb(\theta,Y_u^{\theta}(t_k,x))
D_sY_u^{\theta}(t_k,x)du\\
&+\sum_{i=1}^{d}\int_{s}^t\nabla_x\sigma_i(Y_u^{\theta}(t_k,x))D_sY_u^{\theta}(t_k,x)dW_{u}^i+\int_{s}^t\int_{\mathbb{R}_0^d}
\nabla_xc(Y_{u-}^{\theta}(t_k,x),z)D_sY_u^{\theta}(t_k,x)\widetilde{M}(du,dz),
\end{align*}
for $s\leq t$ a.e., and $D_sY_t^{\theta}(t_k,x)=0$ for $s>t$ a.e. By \cite[Proposition 7]{P08}, it holds that
$$
D_sY_t^{\theta}(t_k,x)=\nabla_xY_t^{\theta}(t_k,x)(\nabla_xY_s^{\theta}(t_k,x))^{-1}\sigma(Y_s^{\theta}(t_k,x)){\bf 1}_{[t_k,t]}(s).
$$

We consider the canonical filtered probability spaces $(\Omega^i, \mathcal{F}^i,\{\mathcal{F}_t^i\}_{t \geq 0}, \P^i)$, $i\in\{1,\ldots,4\}$, associated to each of the four processes $B, N(dt,dz), W$ and $M(dt,dz)$ . Then, $(\Omega,\mathcal{F},\{\mathcal{F}_t\}_{t \geq 0},\P)$ is the product filtered probability space of the four canonical spaces.

We set $\widehat{\Omega}=\Omega^1\times \Omega^2$, $\widehat{\mathcal{F}}=\mathcal{F}^1\otimes\mathcal{F}^2$, $\widehat{\P}=\P^1\otimes\P^2$, $\widehat{\mathcal{F}}_t=\mathcal{F}_t^1\otimes\mathcal{F}_t^2$, $\widetilde{\Omega}=\Omega^3\times \Omega^4$, $\widetilde{\mathcal{F}}=\mathcal{F}^3\otimes\mathcal{F}^4$, $\widetilde{\P}=\P^3\otimes\P^4$, and $\widetilde{\mathcal{F}}_t=\mathcal{F}_t^3\otimes\mathcal{F}_t^4$. Then, $\Omega=\widehat{\Omega}\times\widetilde{\Omega}$, $\mathcal{F}=\widehat{\mathcal{F}}\otimes\widetilde{\mathcal{F}}$, $\P=\widehat{\P}\otimes\widetilde{\P}$, $\mathcal{F}_t=\widehat{\mathcal{F}}_t\otimes\widetilde{\mathcal{F}}_t$, and $\E=\widehat{\E}\otimes\widetilde{\E}$, where $\E$, $\widehat{\E}$, $\widetilde{\E}$ denote the expectation w.r.t. $\P$, $\widehat{\P}$ and $\widetilde{\P}$, respectively. For all $A\in \widetilde{\mathcal{F}}$ and $x\in\R^d$, we set $\widetilde{\P}_x^{\theta}(A)=\widetilde{\E}[{\bf 1}_{A}\vert Y_{t_{k}}^{\theta}=x]
 $. We denote by $\widetilde{\E}_x^{\theta}$ the expectation w.r.t. $\widetilde{\P}_x^{\theta}$. That is, for all $\widetilde{\mathcal{F}}$-measurable random variables $V$, we have that $\widetilde{\E}_x^{\theta}[V]=\widetilde{\E}[V\vert Y_{t_{k}}^{\theta}=x]$.

Under conditions {\bf(A1)}, {\bf(A2)} and {\bf(A4)}\textnormal{(a)}, for any $t>s$ the law of $Y_t^{\theta}$ conditioned on $Y_s^{\theta}=x$ admits a positive transition density $p^{\theta}(t-s,x,y)$, which is differentiable w.r.t. $\theta$. As a consequence of \cite[Proposition 4.1]{G01}, we have the following explicit expression for the logarithm derivative of the transition density w.r.t. $\theta$ in terms of a conditional expectation. 

\begin{proposition}\label{c3Gobet} Under conditions {\bf(A1)}, {\bf(A2)} and {\bf(A4)}\textnormal{(a)}-\textnormal{(c)}, for all $k \in \{0,...,n-1\}$, $\theta\in\Theta$, and $x, y\in\R^d$,
\begin{equation*} \begin{split}
\dfrac{\partial_{\theta}p^{\theta}}{p^{\theta}}\left(\Delta_n,x,y\right)=\dfrac{1}{\Delta_n}\widetilde{\E}_{x}^{\theta}\left[\delta\left( U^{\theta}(t_k,x)  \partial_{\theta} Y_{t_{k+1}}^{\theta}(t_k,x) \right)\Big\vert Y_{t_{k+1}}^{\theta}=y\right],
\end{split}
\end{equation*}
where  $U_{t}^{\theta}(t_k,x)=(D_tY_{t_{k+1}}^{\theta}(t_k,x))^{-1}$, $t \in [t_k, t_{k+1}]$.
\end{proposition}

We have the following decomposition of the Skorohod integral appearing in the conditional expectation of Proposition \ref{c3Gobet}.
\begin{lemma}\label{c3delta} Under conditions {\bf(A1)}, {\bf(A2)} and {\bf(A4)}\textnormal{(a)}-\textnormal{(c)}, for all $k \in \{0,...,n-1\}$, $\theta\in \Theta$, and $x\in\R^d$,
\begin{equation*}\begin{split}
&\delta\left(U^{\theta}(t_k,x)\partial_{\theta}Y_{t_{k+1}}^{\theta}(t_k,x)\right)=\Delta_n(\partial_{\theta}
b(\theta,Y_{t_k}^{\theta}))^{\ast}(\sigma\sigma^{\ast})^{-1}(Y_{t_k}^{\theta})\left(Y_{t_{k+1}}^{\theta}-Y_{t_{k}}^{\theta}-b(\theta,Y_{t_k}^{\theta})\Delta_n\right)\\
&\qquad-R_{1}^{\theta,k}+R_{2}^{\theta,k}+R_{3}^{\theta,k}-R_4^{\theta,k}-R_5^{\theta,k}-R_6^{\theta,k},
\end{split}
\end{equation*}
where 
\begin{align*}
&R_{1}^{\theta,k}=\int_{t_k}^{t_{k+1}} \int_s^{t_{k+1}}\textnormal{tr} \left(D_s\left(((\nabla_xY_{u}^{\theta}(t_k,x))^{-1}\partial_{\theta}b(\theta,Y_{u}^{\theta}(t_k,x)))^{\ast} \right)  \sigma^{-1}(Y_{s}^{\theta}(t_k,x))  \nabla_x Y_{s}^{\theta}(t_k,x) \right) duds,\\
&R_{2}^{\theta,k}=\int_{t_k}^{t_{k+1}}((\nabla_xY_{s}^{\theta}(t_k,x))^{-1}\partial_{\theta}b(\theta,Y_{s}^{\theta}(t_k,x)))^{\ast} ds \\
&\qquad \qquad  \times \int_{t_k}^{t_{k+1}} \left( (\nabla_x Y_{s}^{\theta}(t_k,x))^{\ast} (\sigma^{-1}(Y_{s}^{\theta}(t_k,x)))^{\ast}- (\nabla_x Y_{t_k}^{\theta}(t_k,x))^{\ast} (\sigma^{-1}(Y_{t_k}^{\theta}(t_k,x)))^{\ast}\right)dW_s, \\
&R_{3}^{\theta,k}=\int_{t_k}^{t_{k+1}} \left(((\nabla_xY_{s}^{\theta}(t_k,x))^{-1}\partial_{\theta}b(\theta,Y_{s}^{\theta}(t_k,x)))^{\ast}-((\nabla_xY_{t_k}^{\theta}(t_k,x))^{-1}\partial_{\theta}b(\theta,Y_{t_k}^{\theta}(t_k,x)))^{\ast}\right) ds \\
&\qquad \qquad \times \int_{t_k}^{t_{k+1}}  (\nabla_x Y_{t_k}^{\theta}(t_k,x))^{\ast} (\sigma^{-1}(Y_{t_k}^{\theta}(t_k,x)))^{\ast}dW_s, \\
&R_4^{\theta,k}=\Delta_n(\partial_{\theta}
b(\theta,Y_{t_k}^{\theta}))^{\ast}(\sigma\sigma^{\ast})^{-1}(Y_{t_k}^{\theta})\int_{t_k}^{t_{k+1}}\left(b(\theta,Y_{s}^{\theta})-b(\theta,Y_{t_k}^{\theta})\right)ds,\\
&R_5^{\theta,k}=\Delta_n(\partial_{\theta}
b(\theta,Y_{t_k}^{\theta}))^{\ast}(\sigma\sigma^{\ast})^{-1}(Y_{t_k}^{\theta})\int_{t_k}^{t_{k+1}}\left(\sigma(Y_{s}^{\theta})-\sigma(Y_{t_k}^{\theta})\right)dW_s,\\
&R_6^{\theta,k}=\Delta_n(\partial_{\theta}
b(\theta,Y_{t_k}^{\theta}))^{\ast}(\sigma\sigma^{\ast})^{-1}(Y_{t_k}^{\theta})\int_{t_k}^{t_{k+1}}
\int_{\mathbb{R}_0^d}c(Y_{s-}^{\theta},z)\widetilde{M}(ds,dz).
\end{align*}
\end{lemma}

We will use the following estimates for the solution to \eqref{c3flow}.
\begin{lemma}\label{c3moment3} Assume conditions {\bf(A1)} and  {\bf(A6)}. 
\begin{itemize}
\item[\textnormal{(i)}] For any $p\geq 1$ and $\theta\in\Theta$, there exists a constant $C_p>0$ such that for all $k \in \{0,...,n-1\}$ and $t\in[t_k,t_{k+1}]$, 
\begin{equation*}
\E\left[\left\vert Y_t^{\theta}(t_k,x)-Y_{t_k}^{\theta}(t_k,x)\right\vert^p\big\vert Y_{t_{k}}^{\theta}(t_k,x)=x\right] \leq C_p\left\vert t-t_k\right\vert^{\frac{p}{2}\wedge 1}\left(1+\vert x\vert^p\right).
\end{equation*}

\item[\textnormal{(ii)}] For any function $g$ defined on $\Theta\times\R^d$ with polynomial growth in $x$ uniformly in $\theta\in\Theta$, there exist constants $C, q>0$ such that for all $k \in \{0,...,n-1\}$ and $t\in[t_k,t_{k+1}]$, 
$$
\E\left[\left\vert g(\theta,Y_t^{\theta}(t_k,x))\right\vert\big\vert Y_{t_{k}}^{\theta}(t_k,x)=x\right]\leq C\left(1+\vert x\vert^q\right).
$$
Moreover, all these statements remain valid for $X^{\theta}$.
\end{itemize}
\end{lemma}

Assuming conditions {\bf(A1)}, {\bf(A2)}, {\bf(A4)}\textnormal{(a)}-\textnormal{(c)} and {\bf(A6)}, and using Gronwall's inequality, one can easily check that for any $\theta\in\Theta$ and $p\geq 2$, there exist constants $C_p, q>0$ such that for all $k \in \{0,...,n-1\}$ and $t\in[t_k,t_{k+1}]$, 
\begin{align} \nonumber
&\E\left[\left\vert \nabla_xY_t^{\theta}(t_k,x)\right\vert^p+\left\vert (\nabla_xY_t^{\theta}(t_k,x))^{-1}\right\vert^p\Big\vert Y_{t_{k}}^{\theta}(t_k,x)=x\right]\\ \nonumber
&\qquad+\sup_{s\in [t_k,t_{k+1}]}\E\left[\left\vert D_sY_t^{\theta}(t_k,x) \right\vert^p\Big\vert Y_{t_{k}}^{\theta}(t_k,x)=x\right]\leq C_p, \quad\text{and}\\ \label{ytheta}
&\E\left[\left\vert \partial_{\theta}Y_t^{\theta}(t_k,x)\right\vert^p\Big\vert Y_{t_{k}}^{\theta}(t_k,x)=x\right]\\ \nonumber
&\qquad+\sup_{s\in [t_k,t_{k+1}]}\E\left[\left\vert D_s\left(\nabla_xY_t^{\theta}(t_k,x)\right)\right \vert^p\Big\vert Y_{t_{k}}^{\theta}(t_k,x)=x\right]\leq C_p\left(1+\vert x\right\vert^q ),
\end{align}
where the constant $C_p$ is uniform in $\theta$.
As a consequence, we have the following estimates, which follow easily from \eqref{c3e0}, Lemma \ref{c3moment3} and properties of the moments of the Brownian motion.
\begin{lemma} \label{c3estimate}
Under conditions {\bf(A1)}, {\bf(A2)}, {\bf(A4)}\textnormal{(a)}-\textnormal{(e)} and {\bf(A6)}, for any $\theta\in \Theta$ and $p\geq 2$, there exist constants $C_p, q>0$ such that for all $k \in \{0,...,n-1\}$, 
\begin{align}
&\E\left[-R_{1}^{\theta,k}+R_{2}^{\theta,k}+R_{3}^{\theta,k}\big\vert Y_{t_{k}}^{\theta}(t_k,x)=x\right]=0,\label{c3es1}\\
&\E\left[\left\vert -R_{1}^{\theta,k}+R_{2}^{\theta,k}+R_{3}^{\theta,k}\right\vert^p\big\vert Y_{t_{k}}^{\theta}(t_k,x)=x\right]\leq C_p\Delta_n^{\frac{3p+1}{2}}\left(1+\vert x\vert^q\right).\label{c3es2}
\end{align}
\end{lemma}

We next recall Girsanov's theorem on each interval $[t_k, t_{k+1}]$. 
\begin{lemma}\label{c3Girsanov} Assume conditions {\bf(A1)}-{\bf(A2)}. For all $\theta, \theta_1\in\Theta$, and $k \in \{0,...,n-1\}$, define the measure
\begin{equation*} \begin{split}
\widehat{Q}_k^{\theta_1,\theta}(A)=\widehat{\E}\left[{\bf 1}_{A}e^{-\int_{t_k}^{t_{k+1}}\left(b(\theta,X_t)-b(\theta_1,X_t)\right)^{\ast}\sigma^{-1}(X_t)dB_t+\frac{1}{2}\int_{t_k}^{t_{k+1}}\left\vert\left(b(\theta,X_t)-b(\theta_1,X_t)\right)^{\ast}\sigma^{-1}(X_t)\right\vert^2dt}\right],
\end{split}
\end{equation*}
for all $A\in \widehat{\mathcal{F}}$. Then $\widehat{Q}_k^{\theta_1,\theta}$ is a probability measure and under $\widehat{Q}_k^{\theta_1,\theta}$, the process $(B_t^{\widehat{Q}_k^{\theta_1,\theta}}=B_t+\int_{t_k}^{t}\sigma^{-1}(X_s)(b(\theta,X_s)-b(\theta_1,X_s))ds, t \in [t_k, t_{k+1}])$ is a Brownian motion.
\end{lemma}

\begin{lemma}\label{c3lemma1} Assume conditions {\bf(A1)}, {\bf(A2)}, {\bf(A4)}\textnormal{(b)} and {\bf(A6)}. Let $\theta, \theta_1\in\Theta$ such that $\vert \theta-\theta_1\vert\leq \frac{C}{\sqrt{n\Delta_n}}$, for some constant $C>0$. Then there exist constants $C, q_0>0$ such that for any random variable $V$, $p>1$, and $k \in \{0,...,n-1\}$,
\begin{equation*}\begin{split}
\left\vert\E_{\widehat{Q}_k^{\theta_1,\theta}}\left[V\left(\dfrac{d\widehat{\P}}{d \widehat{Q}_k^{\theta_1,\theta}}-1\right)\Big\vert X_{t_k}^{\theta}\right]\right\vert\leq \dfrac{C}{\sqrt{n}}\left(1+\vert X_{t_k}^{\theta}\vert^{q_0}\right)\int_0^1\left(\E_{\widehat{\P}^{\alpha}}\left[\vert V\vert^p\Big\vert X_{t_k}^{\theta}\right]\right)^{\frac{1}{p}}d\alpha,
\end{split}
\end{equation*}
where $\E_{\widehat{\P}^{\alpha}}$, $\alpha\in[0,1]$, denotes the expectation under the probability measure $\widehat{\P}^{\alpha}$ defined as
\begin{equation*} \begin{split}
\dfrac{d\widehat{\P}^{\alpha}}{d\widehat{Q}_k^{\theta_1,\theta}}:=e^{\alpha\int_{t_k}^{t_{k+1}}\left(b(\theta,X_t)-b(\theta_1,X_t)\right)^{\ast}\sigma^{-1}(X_t)dB_t-\frac{\alpha^2}{2}\int_{t_k}^{t_{k+1}}\left\vert\left(b(\theta,X_t)-b(\theta_1,X_t)\right)^{\ast}\sigma^{-1}(X_t)\right\vert^2dt}.
\end{split}
\end{equation*}
\end{lemma}

Next, we recall a discrete ergodic theorem.
\begin{lemma}\textnormal{\cite[Lemma 8]{K}}\label{c3ergodic} Assume conditions {\bf(A1)} and {\bf(A5)}. Consider a differentiable function $g: \R^d \to \R^{d'}$, whose derivatives have polynomial growth in $x$. Then, as $n\to\infty$,
\begin{equation*}
\dfrac{1}{n}\sum_{k=0}^{n-1}g(X_{t_k})\overset{\P^{\ta_0}}{\longrightarrow}\int_{\mathbb{R}^d}g(x)\pi_{\theta_0}(dx).
\end{equation*}
\end{lemma}

We finally recall two convergence in probability results and a central limit theorem for triangular arrays of random variables. For each $n\in\mathbb{N}$, let $(Z_{k,n})_{k\geq 1}$ and $(\zeta_{k,n})_{k\geq 1}$ be two sequences of random variables defined on the filtered probability space $(\Omega, \mathcal{F}, \{\mathcal{F}_t\}_{t\geq 0}, \P)$, and assume that they are $\mathcal{F}_{t_{k+1}}$-measurable, for all $k$.
\begin{lemma}\label{zero} \textnormal{\cite[Lemma 9]{GJ93}} Assume that as $n  \rightarrow \infty$,  
\begin{equation*} 
\textnormal{(i)}\;  \sum_{k=0}^{n-1}\E\left[Z_{k,n}\vert \mathcal{F}_{t_k}\right] \overset{\P}{\longrightarrow} 0, \quad \text{ and } \quad \textnormal{(ii)} \,  \sum_{k=0}^{n-1}\E\left[Z_{k,n}^2\vert \mathcal{F}_{t_k} \right]\overset{\P}{\longrightarrow} 0.
\end{equation*}
Then as $n\rightarrow\infty$, $\sum_{k=0}^{n-1}Z_{k,n}\overset{\P}{\longrightarrow}0$.
\end{lemma}

\begin{lemma}\label{zero2} \textnormal{\cite[Lemma 4.1]{J11}}  Assume that as $n  \rightarrow \infty$,  
\begin{equation*}
\sum_{k=0}^{n-1}\E\left[\vert Z_{k,n}\vert\vert \mathcal{F}_{t_k}\right] \overset{\P}{\longrightarrow} 0.
\end{equation*}
Then as $n  \rightarrow \infty$, 
$
\sum_{k=0}^{n-1}Z_{k,n}\overset{\P}{\longrightarrow} 0.
$
\end{lemma}

\begin{lemma}\label{clt} \textnormal{\cite[Lemma 4.3]{J11}} 
Assume that there exist real numbers $M$ and $V>0$ such that 
\begin{equation*}\begin{split}
&\sum_{k=0}^{n-1}\E\left[\zeta_{k,n}\vert \mathcal{F}_{t_k}\right] \overset{\P}{\longrightarrow} M, \qquad \sum_{k=0}^{n-1}\left(\E\left[\zeta_{k,n}^2\vert \mathcal{F}_{t_k} \right]-\left(\E\left[\zeta_{k,n}\vert \mathcal{F}_{t_k}\right]\right)^2\right)\overset{\P}{\longrightarrow} V, \text{ and }\\
&\sum_{k=0}^{n-1}\E\left[\zeta_{k,n}^4\vert \mathcal{F}_{t_k}\right] \overset{\P}{\longrightarrow} 0,
\end{split}
\end{equation*}
as $n \rightarrow \infty$. Then as $n  \rightarrow \infty$,  $\sum_{k=0}^{n-1}\zeta_{k,n}\overset{\mathcal{L}(\P)}{\longrightarrow} \mathcal{N}+M$, where $\mathcal{N}$ is a centered Gaussian random variable with variance $V$.
\end{lemma}

\section{Proof of Theorem \ref{c3result}}
\label{proof}

In this section, the proof of Theorem \ref{c3result} will be divided into three steps. We begin deriving a stochastic expansion of the log-likelihood ratio using Proposition \ref{c3Gobet} and Lemma \ref{c3delta}. The second step deals with the main contributions by applying the central limit theorem for triangular arrays to show the LAN property. Finally, the last step is devoted to treat the negligible contributions of the expansion.

\subsection{Expansion of the log-likelihood ratio}
\label{c3sec:expan}
\begin{lemma} Assume conditions {\bf(A1)}, {\bf(A2)} and {\bf(A4)}\textnormal{(a)}-\textnormal{(c)}. Then
\begin{equation}\label{expansion}\begin{split}
&\log\dfrac{p_n(X^{n};\theta_n)}{p_n(X^{n};\theta_0)}
=\sum_{k=0}^{n-1}\xi_{k,n}+\dfrac{u}{\sqrt{n\Delta_n^3}} \sum_{k=0}^{n-1}\int_0^1 \bigg\{Z_{k,n}^{4,\ell}+Z_{k,n}^{5,\ell}+Z_{k,n}^{6,\ell}\\
&\quad+\widetilde{\E}_{X_{t_{k}}}^{\theta(\ell)}\left[-R_{1}^{\theta(\ell),k}+R_{2}^{\theta(\ell),k}+R_{3}^{\theta(\ell),k}-R_4^{\theta(\ell),k}-R_5^{\theta(\ell),k}-R_6^{\theta(\ell),k}\Big\vert Y_{t_{k+1}}^{\theta(\ell)}=X_{t_{k+1}}\right]\bigg\} d \ell,
\end{split}
\end{equation}
where $\theta(\ell):=\theta_n(\ell,u):=\theta_0+\frac{\ell u}{\sqrt{n\Delta_n}}$, and
\begin{align*}
&\xi_{k,n}=\dfrac{u}{\sqrt{n\Delta_n}}\int_0^1\left(\partial_{\theta}
b(\theta(\ell),X_{t_k})\right)^{\ast}(\sigma\sigma^{\ast})^{-1}(X_{t_k})\\
&\qquad \qquad \times \left(\sigma(X_{t_k})\left(B_{t_{k+1}}-B_{t_{k}}\right)+\left(b(\theta_0,X_{t_k})-b(\theta(\ell),X_{t_k})\right)\Delta_n\right)d\ell,\\
&Z_{k,n}^{4,\ell}=\Delta_n\left(\partial_{\theta}
b(\theta(\ell),X_{t_k})\right)^{\ast}(\sigma\sigma^{\ast})^{-1}(X_{t_k})\int_{t_k}^{t_{k+1}}\left(b(\theta_0,X_{s}^{\theta_0})-b(\theta_0,X_{t_k})\right)ds,\\
&Z_{k,n}^{5,\ell}=\Delta_n\left(\partial_{\theta}
b(\theta(\ell),X_{t_k})\right)^{\ast}(\sigma\sigma^{\ast})^{-1}(X_{t_k})\int_{t_k}^{t_{k+1}}\left(\sigma(X_{s}^{\theta_0})-\sigma(X_{t_k})\right)dB_s,\\
&Z_{k,n}^{6,\ell}=\Delta_n\left(\partial_{\theta}
b(\theta(\ell),X_{t_k})\right)^{\ast}(\sigma\sigma^{\ast})^{-1}(X_{t_k})\int_{t_k}^{t_{k+1}}
\int_{\R_0^d}c(X_{s-}^{\theta_0},z)\widetilde{N}(ds,dz).
\end{align*}
\end{lemma}
\begin{proof}
By the Markov property and Proposition \ref{c3Gobet},
\begin{equation*} \begin{split}
&\log\dfrac{p_n(X^{n};\theta_n)}{p_n(X^{n};\theta_0)}=\sum_{k=0}^{n-1}\log\dfrac{p^{\theta_n}}{p^{\theta_0}}(\Delta_n,X_{t_k},X_{t_{k+1}})\\
&\qquad=\sum_{k=0}^{n-1}\dfrac{u}{\sqrt{n\Delta_n}}\int_0^1\dfrac{\partial_{\theta}p^{\theta(\ell)}}{p^{\theta(\ell)}}(\Delta_n,X_{t_k},X_{t_{k+1}})d\ell\\
&\qquad=\sum_{k=0}^{n-1}\dfrac{u}{\sqrt{n\Delta_n^3}}\int_0^1\widetilde{\E}_{X_{t_k}}^{\theta(\ell)}\left[\delta\left(U^{\theta(\ell)}(t_k,X_{t_k})\partial_{\theta}Y_{t_{k+1}}^{\theta(\ell)}(t_k,X_{t_k})\right)\Big\vert Y_{t_{k+1}}^{\theta(\ell)}=X_{t_{k+1}}\right]d\ell.
\end{split}
\end{equation*}

We next apply Lemma \ref{c3delta}, and use equation \eqref{c3eq1} for the term $X_{t_{k+1}}-X_{t_{k}}$ coming from the term $Y_{t_{k+1}}^{\theta(\ell)}-Y_{t_{k}}^{\theta(\ell)}$ in Lemma \ref{c3delta}, to obtain the expansion of the log-likelihood ratio \eqref{expansion}. Thus, the result follows.
\end{proof}

In the next two subsections, we will show that $\xi_{k,n}$ is the only term that contributes to the limit and all the others terms are negligible. Therefore, the main term in the asymptotic behavior is  given by the Gaussian and drift components of equation \eqref{c3eq1}.

In all what follows, hypothesis {\bf(A5)} and Lemma \ref{c3ergodic} will be used repeatedly without being quoted.

\subsection{Main contributions: LAN property}
\begin{lemma}\label{main} Assume conditions {\bf(A1)}, {\bf(A2)}, {\bf(A4)}\textnormal{(a)}-\textnormal{(d)}, {\bf(A5)} and {\bf(A6)}. Then as $n\to\infty$,
\begin{align*}
\sum_{k=0}^{n-1}\xi_{k,n}\overset{\mathcal{L}(\P^{\theta_0})}{\longrightarrow}u\mathcal{N}\left(0,\Gamma(\theta_0)\right)-\dfrac{u^2}{2}\Gamma\left(\theta_0\right),
\end{align*}
where 
$\Gamma(\theta_0)$ is defined in Theorem \ref{c3result}. 
\end{lemma}
\begin{proof}
Applying Lemma \ref{clt} to $\xi_{k,n}$, we need to consider $\E^{\theta_0}[\xi_{k,n}^r\vert \widehat{\mathcal{F}}_{t_k}]$ for $r\in\{1,2,4\}$. This conditional expectation equals $\E[\xi_{k,n}^r\vert \widehat{\mathcal{F}}_{t_k}]$. Therefore, it suffices to show that as $n\to\infty$,
\begin{align}
&\sum_{k=0}^{n-1}\E\left[\xi_{k,n}\vert \widehat{\mathcal{F}}_{t_k}\right]\overset{\P^{\theta_0}}{\longrightarrow}-\dfrac{u^2}{2}\Gamma(\theta_0),\label{c3l1} \\
&\sum_{k=0}^{n-1}\left(\E\left[\xi_{k,n}^2\vert \widehat{\mathcal{F}}_{t_k}\right]-\left(\E\left[\xi_{k,n}\vert \widehat{\mathcal{F}}_{t_k}\right]\right)^2\right)\overset{\P^{\theta_0}}{\longrightarrow}u^2\Gamma(\theta_0), \label{c3l2}\\
&\sum_{k=0}^{n-1}\E\left[\xi_{k,n}^4\vert \widehat{\mathcal{F}}_{t_k}\right]\overset{\P^{\theta_0}}{\longrightarrow}0. \label{c3l3}
\end{align}

{\it Proof of \eqref{c3l1}.} Using the fact that $\E[B_{t_{k+1}}-B_{t_{k}}\vert \widehat{\mathcal{F}}_{t_k}]=0$ and the mean value theorem for vector-valued functions, we get that
\begin{align*}
\sum_{k=0}^{n-1}\E\left[\xi_{k,n}\vert \widehat{\mathcal{F}}_{t_k}\right]&=-\dfrac{u^2}{n}\sum_{k=0}^{n-1}\int_0^1\ell\left(\partial_{\theta}
b(\theta(\ell),X_{t_k})\right)^{\ast}(\sigma\sigma^{\ast})^{-1}(X_{t_k})\int_0^1\partial_{\theta}
b(\theta(\ell,\alpha),X_{t_k})d\alpha d\ell\\
&=-\dfrac{u^2}{2n}\sum_{k=0}^{n-1}\left(\partial_{\theta}b(\theta_0,X_{t_k})\right)^{\ast}(\sigma\sigma^{\ast})^{-1}(X_{t_k})\partial_{\theta}b(\theta_0,X_{t_k})-H_1-H_2,
\end{align*}
where $\theta(\ell,\alpha):=\theta_0+\frac{\alpha\ell u}{\sqrt{n\Delta_n}}$, $H_1=\sum_{k=0}^{n-1}H_{k,n}$, and
\begin{align*}
&H_{k,n}=\dfrac{u^2}{n}\int_0^1\ell\left(\partial_{\theta}
b(\theta(\ell),X_{t_k})\right)^{\ast}(\sigma\sigma^{\ast})^{-1}(X_{t_k})\int_0^1\left(\partial_{\theta}
b(\theta(\ell,\alpha),X_{t_k})-\partial_{\theta}b(\theta_0,X_{t_k})\right)d\alpha d\ell,\\
&H_2=\dfrac{u^2}{n}\sum_{k=0}^{n-1}\int_0^1\ell\left(\partial_{\theta}
b(\theta(\ell),X_{t_k})-\partial_{\theta}
b(\theta_0,X_{t_k})\right)^{\ast}(\sigma\sigma^{\ast})^{-1}(X_{t_k})\partial_{\theta}
b(\theta_0,X_{t_k})d\ell.
\end{align*}

Using hypotheses {\bf(A2)} and {\bf(A4)}\textnormal{(b)}, \textnormal{(d)}, we get that for some constants $C, q>0$,
\begin{align*}
\sum_{k=0}^{n-1}\E\left[\vert H_{k,n}\vert \vert\widehat{\mathcal{F}}_{t_k}\right]\leq \frac{C\vert u\vert^{\epsilon+2}}{(\sqrt{n\Delta_n})^{\epsilon}}\dfrac{1}{n}\sum_{k=0}^{n-1}\left(1+\vert X_{t_k}\vert^{q}\right),
\end{align*}
which, by Lemma \ref{zero2}, implies that $H_1\overset{\P^{\theta_0}}{\longrightarrow}0$ as $n\to\infty$. Thus, so does $H_2$ by using the same argument. On the other hand, applying Lemma \ref{c3ergodic}, we obtain that as $n\to\infty$,
\begin{equation}\label{c3ergo}
\dfrac{1}{n}\sum_{k=0}^{n-1}\left(\partial_{\theta}b(\theta_0,X_{t_k})\right)^{\ast}(\sigma\sigma^{\ast})^{-1}(X_{t_k})\partial_{\theta}b(\theta_0,X_{t_k})\overset{\P^{\theta_0}}{\longrightarrow}\Gamma(\theta_0),
\end{equation}
which gives \eqref{c3l1}.

\vskip 12pt
{\it Proof of \eqref{c3l2}.} First, from the previous computations, we have that
\begin{align*}
\sum_{k=0}^{n-1}\left(\E\left[\xi_{k,n}\vert \widehat{\mathcal{F}}_{t_k}\right]\right)^2&=\dfrac{u^4}{n^2}\sum_{k=0}^{n-1}\left(\int_0^1\ell\left(\partial_{\theta}
b(\theta(\ell),X_{t_k})\right)^{\ast}(\sigma\sigma^{\ast})^{-1}(X_{t_k})\int_0^1\partial_{\theta}
b(\theta(\ell,\alpha),X_{t_k})d\alpha d\ell\right)^2\\
&\leq \dfrac{Cu^4}{n^2}\sum_{k=0}^{n-1}\left(1+\vert X_{t_k}\vert^{q}\right),
\end{align*}
for some constants $C, q>0$, which converges to zero in $\P^{\theta_0}$-probability as $n\to\infty$.

Next, using properties of the moments of the Brownian motion, we can write
\begin{align*}
\sum_{k=0}^{n-1}\E\left[\xi_{k,n}^2\vert \widehat{\mathcal{F}}_{t_k}\right]=\dfrac{u^2}{n}\sum_{k=0}^{n-1}\left(\partial_{\theta}b(\theta_0,X_{t_k})\right)^{\ast}(\sigma\sigma^{\ast})^{-1}(X_{t_k})\partial_{\theta}b(\theta_0,X_{t_k})+H_3+H_4+H_5,
\end{align*}
where, setting $\theta(\ell^{'}):=\theta_n(\ell^{'},u):=\theta_0+\frac{\ell^{'} u}{\sqrt{n\Delta_n}}$, 
\begin{align*}
&H_3=\dfrac{u^2}{n}\sum_{k=0}^{n-1}\int_0^1\left(\partial_{\theta}
b(\theta(\ell),X_{t_k})-\partial_{\theta}
b(\theta_0,X_{t_k})\right)^{\ast}(\sigma\sigma^{\ast})^{-1}(X_{t_k})\partial_{\theta}b(\theta_0,X_{t_k})d\ell,\\
&H_4=\dfrac{u^2}{n}\sum_{k=0}^{n-1}\int_0^1\int_0^1(\partial_{\theta}
b(\theta(\ell),X_{t_k}))^{\ast}(\sigma\sigma^{\ast})^{-1}(X_{t_k})\left(\partial_{\theta}
b(\theta(\ell^{'}),X_{t_k})-\partial_{\theta}
b(\theta_0,X_{t_k})\right)d\ell d\ell^{'},\\
&H_5=\dfrac{u^2\Delta_n}{n}\sum_{k=0}^{n-1}\int_0^1\int_0^1(\partial_{\theta}
b(\theta(\ell),X_{t_k}))^{\ast}(\sigma\sigma^{\ast})^{-1}(X_{t_k})\left(b(\theta_0,X_{t_k})-b(\theta(\ell),X_{t_k})\right)\\
&\qquad\qquad\times\left(b(\theta_0,X_{t_k})-b(\theta(\ell^{'}),X_{t_k})\right)^{\ast}(\sigma\sigma^{\ast})^{-1}(X_{t_k})\partial_{\theta}b(\theta(\ell^{'}),X_{t_k})d\ell d\ell^{'}.
\end{align*}
As for the term $H_1$, using hypotheses {\bf(A2)} and {\bf(A4)}\textnormal{(b)}, \textnormal{(d)}, we get that $H_3, H_4, H_5$ converge to zero in $\P^{\theta_0}$-probability as $n\to\infty$. Moreover, using again \eqref{c3ergo}, we conclude \eqref{c3l2}.

\vskip 12pt
{\it Proof of \eqref{c3l3}.} Basic computations yield
\begin{align*}
\sum_{k=0}^{n-1}\E\left[\xi_{k,n}^4\vert \widehat{\mathcal{F}}_{t_k}\right]\leq \dfrac{Cu^4}{n^2}\sum_{k=0}^{n-1}\left(1+\vert X_{t_k}\vert^q\right),
\end{align*}
for some constants $C, q>0$. The proof of Lemma \ref{main} is completed.
\end{proof}

\subsection{Negligible contributions}
\begin{lemma}\label{negligible} Under conditions {\bf(A1)}-{\bf(A8)}, as $n\to\infty$,
\begin{equation*}\begin{split}
&\dfrac{u}{\sqrt{n\Delta_n^3}}\sum_{k=0}^{n-1}\int_0^1\bigg\{Z_{k,n}^{4,\ell}+Z_{k,n}^{5,\ell}+Z_{k,n}^{6,\ell}\\
&\quad+\widetilde{\E}_{X_{t_{k}}}^{\theta(\ell)}\left[-R_{1}^{\theta(\ell),k}+R_{2}^{\theta(\ell),k}+R_{3}^{\theta(\ell),k}-R_4^{\theta(\ell),k}-R_5^{\theta(\ell),k}-R_6^{\theta(\ell),k}\Big\vert Y_{t_{k+1}}^{\theta(\ell)}=X_{t_{k+1}}\right]\bigg\} d \ell\overset{\P^{\theta_0}}{\longrightarrow} 0.
\end{split}
\end{equation*}
\end{lemma}
\begin{proof} The proof is completed by combining the five Lemmas \ref{lemma2}-\ref{lemma6} below.
\end{proof}

Consequently, from Lemmas \ref{expansion}, \ref{main} and \ref{negligible}, the proof of Theorem \ref{c3result} is now completed.

\begin{lemma}\label{lemma2} Under conditions {\bf(A1)}, {\bf(A2)}, {\bf(A4)}\textnormal{(a)}-\textnormal{(e)}, {\bf(A5)} and {\bf(A6)}, as $n\to\infty$,
\begin{equation*}
\sum_{k=0}^{n-1}\dfrac{u}{\sqrt{n\Delta_n^3}}\int_0^1\widetilde{\E}_{X_{t_k}}^{\theta(\ell)}\left[-R_{1}^{\theta(\ell),k}+R_{2}^{\theta(\ell),k}+R_{3}^{\theta(\ell),k}\Big\vert Y_{t_{k+1}}^{\theta(\ell)}=X_{t_{k+1}}\right]d\ell\overset{\P^{\theta_0}}{\longrightarrow} 0.
\end{equation*}
\end{lemma}
\begin{proof} 
It suffices to show that conditions (i) and (ii) of Lemma \ref{zero} hold under the measure $\P^{\theta_0}$. We start showing (i). Applying Girsanov's theorem, Lemma \ref{c3lemma1}, \eqref{c3es1}, and \eqref{c3es2} with $p=2$, we get that
\begin{equation*}\begin{split} 
&\left\vert\sum_{k=0}^{n-1}\dfrac{u}{\sqrt{n\Delta_n^3}}\int_0^1\E\left[\widetilde{\E}_{X_{t_k}}^{\theta(\ell)}
\left[-R_{1}^{\theta(\ell),k}+R_{2}^{\theta(\ell),k}+R_{3}^{\theta(\ell),k}\Big\vert Y_{t_{k+1}}^{\theta(\ell)}=
X_{t_{k+1}}\right]\Big\vert \widehat{\mathcal{F}}_{t_k}\right]d\ell\right\vert\\
&\leq\sum_{k=0}^{n-1}\dfrac{\vert u\vert}{\sqrt{n\Delta_n^3}}\int_0^1\left\vert\E_{\widehat{Q}_k^{\theta(\ell),\theta_0}}\left[(-R_{1}^{\theta(\ell),k}+R_{2}^{\theta(\ell),k}+R_{3}^{\theta(\ell),k})\left(\dfrac{d\widehat{\P}}{d \widehat{Q}_k^{\theta(\ell),\theta_0}}-1\right)\Big\vert X_{t_k}\right]\right\vert d\ell\\
&\leq \dfrac{C\vert u\vert\Delta_n^{\frac{1}{4}}}{n}\sum_{k=0}^{n-1}\left(1+\vert X_{t_k}\vert^{q}\right),
\end{split}
\end{equation*}
for some constants $C,q>0$.
Observe that \eqref{c3es2} remains valid under the measure $\widehat{\P}^{\alpha}$ defined in Lemma \ref{c3lemma1}. This shows Lemma \ref{zero}(i). Similarly, applying Jensen's inequality, Girsanov's theorem, Lemma \ref{c3lemma1}, and \eqref{c3es2} with $p \in \{2,4\}$, we obtain that
\begin{align*}
&\sum_{k=0}^{n-1}\dfrac{u^2}{n\Delta_n^3}\E\left[\left(\int_0^1\widetilde{\E}_{X_{t_k}}^{\theta(\ell)}
\left[-R_{1}^{\theta(\ell),k}+R_{2}^{\theta(\ell),k}+R_{3}^{\theta(\ell),k}\Big\vert Y_{t_{k+1}}^{\theta(\ell)}=X_{t_{k+1}}\right]d\ell\right)^2\Big\vert \widehat{\mathcal{F}}_{t_k}
\right]\\
& \leq\sum_{k=0}^{n-1}\dfrac{u^2}{n\Delta_n^3}\int_0^1\bigg\{\E_{\widehat{Q}_k^{\theta(\ell),\theta_0}}
\left[\left(-R_{1}^{\theta(\ell),k}+R_{2}^{\theta(\ell),k}+R_{3}^{\theta(\ell),k}\right)^2\Big\vert X_{t_k}\right]\\
&\qquad+\left\vert\E_{\widehat{Q}_k^{\theta(\ell),\theta_0}}\left[\left(
-R_{1}^{\theta(\ell),k}+R_{2}^{\theta(\ell),k}+R_{3}^{\theta(\ell),k}\right)^2\left(\dfrac{d\widehat{\P}}{d \widehat{Q}_k^{\theta(\ell),\theta_0}}-1\right)\Big\vert X_{t_k}\right]\right\vert\bigg\}d\ell\\
&\leq \dfrac{C u^2\Delta_n^{\frac{1}{4}}}{n}\sum_{k=0}^{n-1}\left(1+\vert X_{t_k}\vert^{q}\right),
\end{align*}
which concludes the desired result.
\end{proof}

\begin{lemma}\label{lemma3} Under conditions {\bf(A1)}, {\bf(A2)}, {\bf(A4)}\textnormal{(b)}, {\bf(A5)} and {\bf(A6)}, as $n\to\infty$,
\begin{equation*}\begin{split}
\sum_{k=0}^{n-1}\dfrac{u}{\sqrt{n\Delta_n^3}}\int_0^1\widetilde{\E}_{X_{t_{k}}}^{\theta(\ell)}\left[R_5^{\theta(\ell),k}\Big\vert Y_{t_{k+1}}^{\theta(\ell)}=X_{t_{k+1}}\right]d\ell\overset{\P^{\theta_0}}{\longrightarrow} 0.
\end{split}
\end{equation*}
\end{lemma}
\begin{proof}
We proceed similarly as in the proof of Lemma \ref{lemma2}.
\end{proof}

\begin{lemma}\label{lemma4} Under conditions {\bf(A1)}, {\bf(A2)}, {\bf(A4)}\textnormal{(b)}, {\bf(A5)} and {\bf(A6)}, as $n\to\infty$, 
\begin{equation*}
\sum_{k=0}^{n-1}\dfrac{u}{\sqrt{n\Delta_n^3}}\int_0^1Z^{5,\ell}_{k,n}d\ell\overset{\P^{\theta_0}}{\longrightarrow} 0.
\end{equation*}
\end{lemma}
\begin{proof} 
Clearly, for all $n\geq 1$,
$$
\sum_{k=0}^{n-1}\dfrac{u}{\sqrt{n\Delta_n^3}}\int_0^1\E\left[Z^{5,\ell}_{k,n}\big\vert \widehat{\mathcal{F}}_{t_k}\right] d\ell=0,
$$
and by Lemma \ref{c3moment3}(i),
$$
\sum_{k=0}^{n-1}\dfrac{u^2}{n\Delta_n^3}\E\left[\left(\int_0^1Z^{5,\ell}_{k,n}d\ell\right)^2\Big\vert \widehat{\mathcal{F}}_{t_k}
\right]\leq \frac{Cu^2\Delta_n}{n}\sum_{k=0}^{n-1}\left(1+\vert X_{t_k}\vert^{q}\right),
$$
for some constants $C, q>0$. Thus, Lemma \ref{zero} concludes the desired result.
\end{proof}

\begin{lemma}\label{lemma5} Assume conditions {\bf(A1)}, {\bf(A2)}, {\bf(A4)}\textnormal{(b)}, \textnormal{(e)}, {\bf(A5)} and {\bf(A6)}. Then as $n\to\infty$, 
\begin{equation*}\begin{split}
&\sum_{k=0}^{n-1}\dfrac{u}{\sqrt{n\Delta_n^3}}\int_0^1\left(Z_{k,n}^{4,\ell}-\widetilde{\E}_{X_{t_k}}^{\theta(\ell)}\left[R_4^{\theta(\ell),k}\Big\vert Y_{t_{k+1}}^{\theta(\ell)}=X_{t_{k+1}}\right]\right)d\ell\overset{\P^{\theta_0}}{\longrightarrow} 0.
\end{split}
\end{equation*}
\end{lemma}
\begin{proof}
By the mean value theorem for vector-valued functions,
\begin{align*}
Z_{k,n}^{4,\ell}-\widetilde{\E}_{X_{t_k}}^{\theta(\ell)}\left[R_4^{\theta(\ell),k}\Big\vert Y_{t_{k+1}}^{\theta(\ell)}=X_{t_{k+1}}\right]=\Delta_n\left(\partial_{\theta}
b(\theta(\ell),X_{t_k})\right)^{\ast}(\sigma\sigma^{\ast})^{-1}(X_{t_k})\left(M_{k,n,1}+M_{k,n,2}\right),
\end{align*}
where
\begin{align*}
M_{k,n,1}&=-\dfrac{\ell u}{\sqrt{n\Delta_n}}\int_{t_k}^{t_{k+1}}\int_{0}^{1}\left(\partial_{\theta}b(\theta_0+\frac{\alpha\ell u}{\sqrt{n\Delta_n}},X_{s}^{\theta_0})-\partial_{\theta}b(\theta_0+\frac{\alpha\ell u}{\sqrt{n\Delta_n}},X_{t_k})\right)d\alpha ds,\\
M_{k,n,2}&=\int_{t_k}^{t_{k+1}}\left(b(\theta(\ell),X_{s}^{\theta_0})-b(\theta(\ell),X_{t_k})\right)ds\\
&\qquad-\widetilde{\E}_{X_{t_k}}^{\theta(\ell)}\left[\int_{t_k}^{t_{k+1}}\left(b(\theta(\ell),Y_{s}^{\theta(\ell)})-b(\theta(\ell),Y_{t_k}^{\theta(\ell)})\right)ds\Big\vert Y_{t_{k+1}}^{\theta(\ell)}=X_{t_{k+1}}\right].
\end{align*}

Using {\bf(A2)}, {\bf(A4)}\textnormal{(b)}, \textnormal{(e)} and Lemma \ref{c3moment3}(i), we get
that
\begin{align*}
\sum_{k=0}^{n-1}\dfrac{\vert u\vert}{\sqrt{n\Delta_n}}\E\left[\left\vert\int_0^1 \left(\partial_{\theta}
b(\theta(\ell),X_{t_k})\right)^{\ast}(\sigma\sigma^{\ast})^{-1}(X_{t_k})M_{k,n,1}d\ell\right\vert\bigg \vert  \widehat{\mathcal{F}}_{t_k}\right]\leq \dfrac{C u^2\sqrt{\Delta_n}}{n}\sum_{k=0}^{n-1}\left(1+\vert X_{t_k}\vert^{q}\right),
\end{align*}
for some constants $C, q>0$. Therefore, by Lemma \ref{zero2}, we conclude that as $n\to\infty$, 
\begin{align*}
\sum_{k=0}^{n-1}\dfrac{u}{\sqrt{n\Delta_n}}\int_0^1\left(\partial_{\theta}
b(\theta(\ell),X_{t_k})\right)^{\ast}(\sigma\sigma^{\ast})^{-1}(X_{t_k})M_{k,n,1}d\ell\overset{\P^{\theta_0}}{\longrightarrow} 0.
\end{align*}

We next show that as $n\to\infty$,
\begin{equation*}\begin{split}
\sum_{k=0}^{n-1}\dfrac{u}{\sqrt{n\Delta_n}}\int_0^1\left(\partial_{\theta}
b(\theta(\ell),X_{t_k})\right)^{\ast}(\sigma\sigma^{\ast})^{-1}(X_{t_k})M_{k,n,2}d\ell\overset{\P^{\theta_0}}{\longrightarrow} 0.
\end{split}
\end{equation*}

Using Girsanov's theorem, the Lipschitz condition on $b$, {\bf(A2)}, {\bf(A4)}\textnormal{(b)}, and Lemmas \ref{c3lemma1} and \ref{c3moment3}(i), we obtain that 
\begin{align*}
&\left\vert\sum_{k=0}^{n-1}\dfrac{u}{\sqrt{n\Delta_n}}\int_0^1\left(\partial_{\theta}
b(\theta(\ell),X_{t_k})\right)^{\ast}(\sigma\sigma^{\ast})^{-1}(X_{t_k})\E\left[M_{k,n,2}\vert \widehat{\mathcal{F}}_{t_k}\right]d\ell\right\vert\\
&=\bigg\vert\sum_{k=0}^{n-1}\dfrac{u}{\sqrt{n\Delta_n}}\int_0^1\left(\partial_{\theta}
b(\theta(\ell),X_{t_k})\right)^{\ast}(\sigma\sigma^{\ast})^{-1}(X_{t_k})\\
&\qquad\times\bigg\{\int_{t_k}^{t_{k+1}}\E_{\widehat{Q}_k^{\theta(\ell),\theta_0}}\left[\left(b(\theta(\ell),X_{s}^{\theta(\ell)})-b(\theta(\ell),X_{t_k}^{\theta(\ell)})\right)\left(\dfrac{d\widehat{\P}}{d \widehat{Q}_k^{\theta(\ell),\theta_0}}-1\right)\Big\vert X_{t_k}\right]ds\\
&\qquad-\E_{\widehat{Q}_k^{\theta(\ell),\theta_0}}\left[\int_{t_k}^{t_{k+1}}\left(b(\theta(\ell),Y_{s}^{\theta(\ell)})-b(\theta(\ell),Y_{t_k}^{\theta(\ell)})\right)ds \left(\dfrac{d\widehat{\P}}{d \widehat{Q}_k^{\theta(\ell),\theta_0}}-1\right)\Big\vert X_{t_k}\right]\bigg\}d\ell\bigg\vert\\
&\leq \dfrac{C\vert u\vert\Delta_n}{n}\sum_{k=0}^{n-1}\left(1+\vert X_{t_k}\vert^{q}\right),
\end{align*}
for some constants $C, q>0$, which shows Lemma \ref{zero}(i).

Finally, we proceed as in the proof of Lemma \ref{lemma2} to show that condition (ii) of Lemma \ref{zero} holds true. Thus, the result follows.
\end{proof}

\begin{lemma}\label{lemma6} Assume conditions {\bf(A1)}-{\bf(A8)}. Then as $n\to\infty$, 
\begin{equation*}\begin{split}
&\sum_{k=0}^{n-1}\dfrac{u}{\sqrt{n\Delta_n^3}}\int_0^1\left(Z_{k,n}^{6,\ell}-\widetilde{\E}_{X_{t_k}}^{\theta(\ell)}\left[R_6^{\theta(\ell),k}\Big\vert Y_{t_{k+1}}^{\theta(\ell)}=X_{t_{k+1}}\right]\right)d\ell\overset{\P^{\theta_0}}{\longrightarrow} 0.
\end{split}
\end{equation*}
\end{lemma}
\begin{proof}
First, by Girsanov's theorem,
\begin{align*} 
\E\left[Z_{k,n}^{6,\ell}-\widetilde{\E}_{X_{t_k}}^{\theta(\ell)}\left[R_6^{\theta(\ell),k}\Big\vert Y_{t_{k+1}}^{\theta(\ell)}=X_{t_{k+1}}\right]\Big\vert  \widehat{\mathcal{F}}_{t_k}\right]=-\E_{\widehat{Q}_k^{\theta(\ell),\theta_0}}\left[R_6^{\theta(\ell),k}\dfrac{d\widehat{\P}}{d \widehat{Q}_k^{\theta(\ell),\theta_0}}\Big\vert X_{t_k}\right]=0,
\end{align*}
where we have used the independence between $R_6^{\theta(\ell),k}$ and $\frac{d\widehat{\P}}{d \widehat{Q}_k^{\theta(\ell),\theta_0}}$ together with the fact that $\E_{\widehat{Q}_k^{\theta(\ell),\theta_0}}[R_6^{\theta(\ell),k}]=0$ and $\E_{\widehat{Q}_k^{\theta(\ell),\theta_0}}[\frac{d\widehat{\P}}{d \widehat{Q}_k^{\theta(\ell),\theta_0}}\vert X_{t_k}]=1$. This shows that the term (i) of Lemma \ref{zero} is actually equal to 0 for all $n \geq 1$.

We next show that condition (ii) of Lemma \ref{zero} holds. Cauchy-Schwarz inequality gives
\begin{equation*} \begin{split}
&\sum_{k=0}^{n-1}\dfrac{u^2}{n\Delta_n^3}\E\left[\left(\int_0^1\left(Z_{k,n}^{6,\ell}-\widetilde{\E}_{X_{t_k}}^{\theta(\ell)}\left[R_6^{\theta(\ell),k}\Big\vert Y_{t_{k+1}}^{\theta(\ell)}=X_{t_{k+1}}\right]\right)d\ell\right)^2\bigg\vert \widehat{\mathcal{F}}_{t_k}\right]\\
&\leq \dfrac{3u^2}{n\Delta_n}\sum_{k=0}^{n-1}\int_0^1\left(D_{1}+D_2+D_3\right)d\ell,
\end{split}
\end{equation*}
where, setting $e_k(\theta(\ell)):=\left(\partial_{\theta}
b(\theta(\ell),X_{t_k})\right)^{\ast}(\sigma\sigma^{\ast})^{-1}(X_{t_k})$,
\begin{align*} 
&D_1=\E\left[\left(e_k(\theta(\ell))\int_{t_k}^{t_{k+1}}\int_{\R_0^d}\left(c(X_{s-}^{\theta_0},z)-c(X_{t_k},z)\right)\widetilde{N}(ds,dz)\right)^2\Big\vert X_{t_k}\right],\\
&D_2=\E\left[\widetilde{\E}_{X_{t_k}}^{\theta(\ell)}\left[\left(e_k(\theta(\ell))\int_{t_k}^{t_{k+1}}
\int_{\R_0^d}\left(c(Y_{s-}^{\theta(\ell)},z)-c(Y_{t_k}^{\theta(\ell)},z)\right)\widetilde{M}(ds,dz)\right)^2\bigg\vert Y_{t_{k+1}}^{\theta(\ell)}=X_{t_{k+1}}\right]\Big\vert X_{t_k}\right],\\
&D_3=\E\bigg[\bigg(e_k(\theta(\ell))\bigg(\int_{t_k}^{t_{k+1}}
\int_{\R_0^d}c(X_{t_k},z)\widetilde{N}(ds,dz)\\
&\qquad\qquad-\widetilde{\E}_{X_{t_k}}^{\theta(\ell)}\left[\int_{t_k}^{t_{k+1}}
\int_{\R_0^d}c(Y_{t_k}^{\theta(\ell)},z)\widetilde{M}(ds,dz)\bigg\vert Y_{t_{k+1}}^{\theta(\ell)}=X_{t_{k+1}}\right]\bigg)\bigg)^2\Big\vert X_{t_k}\bigg].
\end{align*}

Using Burkh\"{o}lder's inequality, the Lipschitz property of $c$ and Lemma \ref{c3moment3}(i), together with hypotheses {\bf(A2)}, {\bf(A4)}\textnormal{(b)} and {\bf(A6)}, we get that for some constants $C, q>0$,
$$
\dfrac{3u^2}{n\Delta_n}\sum_{k=0}^{n-1}\int_0^1D_{1}d\ell\leq \dfrac{Cu^2\Delta_n}{n}\sum_{k=0}^{n-1}\left(1+\vert X_{t_k}\vert^{q}\right).
$$

Using Girsanov's theorem, Burkh\"{o}lder's inequality, Lemmas \ref{c3lemma1} and \ref{c3moment3}(i), together with hypotheses {\bf(A1)}, {\bf(A2)}, {\bf(A4)}\textnormal{(b)} and {\bf(A6)}, we obtain that for some constants $C, q>0$,
\begin{align*} 
&D_2\leq\E_{\widehat{Q}_k^{\theta(\ell),\theta_0}}\left[\left(e_k(\theta(\ell))\int_{t_k}^{t_{k+1}}
\int_{\R_0^d}\left(c(Y_{s-}^{\theta(\ell)},z)-c(Y_{t_k}^{\theta(\ell)},z)\right)\widetilde{M}(ds,dz)\right)^2\Big\vert X_{t_k}\right]\\
&+\left\vert \E_{\widehat{Q}_k^{\theta(\ell),\theta_0}}\left[\left(e_k(\theta(\ell))\int_{t_k}^{t_{k+1}}
\int_{\R_0^d}\left(c(Y_{s-}^{\theta(\ell)},z)-c(Y_{t_k}^{\theta(\ell)},z)\right)\widetilde{M}(ds,dz)\right)^2\left(\dfrac{d\widehat{\P}}{d \widehat{Q}_k^{\theta(\ell),\theta_0}}-1\right)\Big\vert X_{t_k}\right]\right\vert\\
&\leq C\left(\Delta_n^2+\dfrac{\Delta_n}{\sqrt{n}}\right)\left(1+\vert X_{t_k}\vert^{q}\right).
\end {align*} 
This implies that
\begin{align*} 
\dfrac{3u^2}{n\Delta_n}\sum_{k=0}^{n-1}\int_0^1D_{2}d\ell\leq C\left(\Delta_n+\dfrac{1}{\sqrt{n}}\right)\dfrac{u^2}{n}\sum_{k=0}^{n-1}\left(1+\vert X_{t_k}\vert^{q}\right).
\end {align*} 

Again, Girsanov's theorem yields $D_3=D_{3,1}+D_{3,2}$, where
\begin{align*} 
D_{3,1}&=\E_{\widehat{Q}_k^{\theta(\ell),\theta_0}}\bigg[\bigg(e_k(\theta(\ell))\bigg(\int_{t_k}^{t_{k+1}}
\int_{\R_0^d}c(X_{t_k},z)\widetilde{N}(ds,dz)\\
&\qquad-\widetilde{\E}_{X_{t_k}}^{\theta(\ell)}\left[\int_{t_k}^{t_{k+1}}
\int_{\R_0^d}c(Y_{t_k}^{\theta(\ell)},z)\widetilde{M}(ds,dz)\bigg\vert Y_{t_{k+1}}^{\theta(\ell)}=X_{t_{k+1}}\right]\bigg)\bigg)^2\left(\dfrac{d\widehat{\P}}{d \widehat{Q}_k^{\theta(\ell),\theta_0}}-1\right)\Big\vert X_{t_k}\bigg],\\
D_{3,2}&=\E_{\widehat{Q}_k^{\theta(\ell),\theta_0}}\bigg[\bigg(e_k(\theta(\ell))\bigg(\int_{t_k}^{t_{k+1}}
\int_{\R_0^d}c(X_{t_k},z)N(ds,dz)\\
&\qquad-\widetilde{\E}_{X_{t_k}}^{\theta(\ell)}\left[\int_{t_k}^{t_{k+1}}
\int_{\R_0^d}c(Y_{t_k}^{\theta(\ell)},z)M(ds,dz)\bigg\vert Y_{t_{k+1}}^{\theta(\ell)}=X_{t_{k+1}}\right]\bigg)\bigg)^2\Big\vert X_{t_k}\bigg].
\end{align*}

Observe that $\vert D_{3,1}\vert\leq 2(D_{3,1,1} +  D_{3,1,2})$, where
\begin{align*} 
&D_{3,1,1}=\left\vert\E_{\widehat{Q}_k^{\theta(\ell),\theta_0}}\left[\left(e_k(\theta(\ell))\int_{t_k}^{t_{k+1}}
\int_{\R_0^d}c(X_{t_k},z)\widetilde{N}(ds,dz)\right)^2\left(\dfrac{d\widehat{\P}}{d \widehat{Q}_k^{\theta(\ell),\theta_0}}-1\right)\Big\vert X_{t_k}\right]\right\vert,\\
&D_{3,1,2}=\bigg\vert\E_{\widehat{Q}_k^{\theta(\ell),\theta_0}}\bigg[\left(\widetilde{\E}_{X_{t_k}}^{\theta(\ell)}\left[e_k(\theta(\ell))\int_{t_k}^{t_{k+1}}
\int_{\R_0^d}c(Y_{t_k}^{\theta(\ell)},z)\widetilde{M}(ds,dz)\bigg\vert Y_{t_{k+1}}^{\theta(\ell)}=X_{t_{k+1}}\right]\right)^2\\
&\qquad\qquad\qquad\times\left(\dfrac{d\widehat{\P}}{d \widehat{Q}_k^{\theta(\ell),\theta_0}}-1\right)\Big\vert X_{t_k}\bigg]\bigg\vert.
\end{align*}

Using the same arguments as for the term $D_2$, we get that for some constants $C, q>0$,
$$
\dfrac{3u^2}{n\Delta_n}\sum_{k=0}^{n-1}\int_0^1D_{3,1,1}d\ell\leq \dfrac{C}{\sqrt{n\Delta_n}}\dfrac{u^2}{n}\sum_{k=0}^{n-1}\left(1+\vert X_{t_k}\vert^{q}\right).
$$

Applying Lemma \ref{c3lemma1}, Jensen's inequality and {\bf(A1)}, {\bf(A2)}, {\bf(A4)}\textnormal{(b)} and {\bf(A6)}, we obtain that for some constants $C, q>0$,
\begin{align*} 
&\dfrac{3u^2}{n\Delta_n}\sum_{k=0}^{n-1}\int_0^1D_{3,1,2}d\ell\leq\dfrac{Cu^2}{n\Delta_n\sqrt{n}}\sum_{k=0}^{n-1}\left(1+\vert X_{t_k}\vert^{q}\right)\\
&\times\int_0^1\left(\E_{\widehat{\P}^{\alpha}}\left[\widetilde{\E}_{X_{t_k}}^{\theta(\ell)}\left[\left(e_k(\theta(\ell))\int_{t_k}^{t_{k+1}}
\int_{\R_0^d}c(Y_{t_k}^{\theta(\ell)},z)\widetilde{M}(ds,dz)\right)^4\bigg\vert Y_{t_{k+1}}^{\theta(\ell)}=X_{t_{k+1}}\right]\Big\vert X_{t_k}\right]\right)^{1/2}d\ell\\
&\leq \dfrac{C}{\sqrt{n\Delta_n}}\dfrac{u^2}{n}\sum_{k=0}^{n-1}\left(1+\vert X_{t_k}\vert^{q}\right).
\end{align*} 

Finally, it remains to treat $D_{3,2}$. Multiplying the random variable inside the expectation by ${\bf 1}_{\widehat{J}_{0,k}} +{\bf 1}_{\widehat{J}_{1,k}}+ {\bf 1}_{\widehat{J}_{2,k}}$, applying Lemma \ref{c3ordre}, and using the inequality $e^{-x}<\frac{p!}{x^p}$, valid for any $x>0$ and $p\geq 1$, we get that for $n$ large enough, for any $\alpha\in(\upsilon+3m\gamma+3\gamma,\frac{1}{2})$, $\alpha_0\in(\frac 14,\frac 12-3\gamma)$, $\varepsilon\in(0,\alpha_0-3m\gamma)$, $q>1$, and $p\geq 1$,
\begin{align*}
&\dfrac{3u^2}{n\Delta_n}\sum_{k=0}^{n-1}\int_0^1D_{3,2}d\ell=\dfrac{3u^2}{n\Delta_n}\sum_{k=0}^{n-1}\int_0^1\left(M_0^{\theta(\ell)}+M_1^{\theta(\ell)}+M_2^{\theta(\ell)}\right) d\ell\\
&\leq C
\left(\lambda^{\frac{2}{q}}\Delta_n^{\frac{1}{q}}+\Delta_n^{2\varepsilon}+\left(\int_{\{\vert z\vert \leq\rho_1\Delta_n^{\upsilon}\}}\nu(dz)\right)^{\frac{1}{q}}\right)\dfrac{u^2}{n}\sum_{k=0}^{n-1}\left(1+\vert X_{t_k}\vert^{q_1}\right)\\
&\quad+C\dfrac{u^2}{n}\sum_{k=0}^{n-1}\left(1+\vert X_{t_k}\vert^{q_1}\right)\left(\Delta_n^{-2m\gamma-1}e^{-C_0\frac{\Delta_n^{2\alpha-1}}{(1+\vert X_{t_k}\vert^2)^3}}+\Delta_n^{-2m\gamma-\frac{d}{2}-1}e^{-C_1\frac{\Delta_n^{2\alpha_0+6\gamma-1}}{(1+\vert X_{t_k}\vert^2)^3}}\right)\\
&\leq C
\left(\lambda^{\frac{2}{q}}\Delta_n^{\frac{1}{q}}+\Delta_n^{2\varepsilon}+\left(\int_{\{\vert z\vert \leq\rho_1\Delta_n^{\upsilon}\}}\nu(dz)\right)^{\frac{1}{q}}\right)\dfrac{u^2}{n}\sum_{k=0}^{n-1}\left(1+\vert X_{t_k}\vert^{q_1}\right)\\
&\quad+Cp!\left(\Delta_n^{-2m\gamma-1+\left(1-2\alpha\right)p}+\Delta_n^{-2m\gamma-\frac{d}{2}-1+(1-2\alpha_0-6\gamma)p}\right)\dfrac{u^2}{n}\sum_{k=0}^{n-1}\left(1+\vert X_{t_k}\vert^{q_1}\right)\left(1+\vert X_{t_k}\vert^2\right)^{3p},
\end{align*}
for some constants $C, C_0, C_1>0$, $q_1>1$. Note that the events $\widehat{J}_{0,k}$, $\widehat{J}_{1,k}$, $\widehat{J}_{2,k}$, the constants $\upsilon$, $\gamma$ and $M_0^{\theta(\ell)}, M_1^{\theta(\ell)}, M_2^{\theta(\ell)}$ are defined in Subsections \ref{expression} and \ref{large}. 

Therefore, using hypothesis {\bf(A7)} and choosing $p\geq 1$ such that $-2m\gamma-1+(1-2\alpha)p>0$ and $-2m\gamma-\frac{d}{2}-1+(1-2\alpha_0-6\gamma)p>0$, we conclude that as $n\to\infty$, 
\begin{align*}
\dfrac{3u^2}{n\Delta_n}\sum_{k=0}^{n-1}\int_0^1D_{3,2}d\ell\overset{\P^{\theta_0}}{\longrightarrow} 0.
\end{align*}
Thus, the desired proof is now completed.
\end{proof}

\subsection{Bounded drift}
\label{bounded}

The aim of this Subsection is to prove that the LAN property also holds true for equation \eqref{c3eq1} when the drift is assumed to be bounded. For this, the assumptions on the drift and jump coefficients are reformulated as follows.

\begin{list}{labelitemi}{\leftmargin=1cm}
	\item[\bf(A1')] Same condition as {\bf(A1)} except that $\vert b(\theta,x)\vert \leq L(1+\vert x\vert)$ and $\vert c(x,z)\vert\leq \zeta(z)(1+\vert x\vert)$ are replaced by 
	\begin{align*}\vert b(\theta,x)\vert \leq L, \qquad \vert c(x,z)\vert\leq \zeta(z).
	\end{align*}
	
	\item[\bf(A8')] $\vert\det(\textup{I}_d+\nabla_{x}c(x,z))\vert\geq \eta$, and $\vert\nabla\psi^{-1}(v)u\vert\geq \frac{\vert u\vert}{\beta}$, for some constants $\eta, \beta>0$, where $\psi(v)=v+c(v,z)-x-c(x,z)$.
\end{list}

In this case, the LAN property also holds.
\begin{theorem} Assume conditions {\bf(A1')}, {\bf(A2)}-{\bf(A7)} and {\bf(A8')}. Then, the statement of Theorem \ref{c3result} remains valid.
\end{theorem}
\begin{proof}
	The proof follows along the same lines as that of Theorem \ref{c3result} except that the estimates \eqref{q0} and \eqref{q1} are now replaced by \eqref{bq0} and \eqref{bq1}.
\end{proof}

\section{Appendix}
\subsection{Proof of Proposition \ref{c3Gobet}}
\begin{proof}
Let $f: \R^d\rightarrow \R$ be a continuously differentiable function with compact support. Fix $t \in [t_k,t_{k+1}]$. The chain rule of the Malliavin calculus gives $(D_t(f(Y_{t_{k+1}}^{\theta}(t_k,x))))^{\ast}=(\nabla f(Y_{t_{k+1}}^{\theta}(t_k,x)))^{\ast} \, D_tY_{t_{k+1}}^{\theta}(t_k,x)$. Since the matrix $D_tY_{t_{k+1}}^{\theta}(t_k,x)$ is invertible a.s., we have $(\nabla f(Y_{t_{k+1}}^{\theta}(t_k,x)))^{\ast}=(D_t(f(Y_{t_{k+1}}^{\theta}(t_k,x))))^{\ast} \, U_{t}^{\theta}(t_k,x)$, where $U_t^{\theta}(t_k,x)=(D_tY_{t_{k+1}}^{\theta}(t_k,x))^{-1}$.

Then, using the integration by parts formula of the Malliavin calculus on the interval $[t_k,t_{k+1}]$, we get that
\begin{equation*}\begin{split}
\partial_{\theta}\widetilde{\E}\left[f(Y_{t_{k+1}}^{\theta}(t_k,x))\right]&=\widetilde{\E}\left[(\nabla f(Y_{t_{k+1}}^{\theta}(t_k,x)))^{\ast} \, \partial_{\theta} Y_{t_{k+1}}^{\theta}(t_k,x)\right]\\
&=\dfrac{1}{\Delta_n}\widetilde{\E}\left[\int_{t_k}^{t_{k+1}}(\nabla f(Y_{t_{k+1}}^{\theta}(t_k,x)))^{\ast} \,  \partial_{\theta} Y_{t_{k+1}}^{\theta}(t_k,x)dt\right]\\
&=\dfrac{1}{\Delta_n}\widetilde{\E}\left[\int_{t_k}^{t_{k+1}}(D_t(f(Y_{t_{k+1}}^{\theta}(t_k,x))))^{\ast} \, U_{t}^{\theta}(t_k,x) \, \partial_{\theta} Y_{t_{k+1}}^{\theta}(t_k,x)dt\right]\\
&=\dfrac{1}{\Delta_n}\widetilde{\E}\left[f(Y_{t_{k+1}}^{\theta}(t_k,x))\delta\left(U^{\theta}(t_k,x) \partial_{\theta}Y_{t_{k+1}}^{\theta}(t_k,x)\right)\right].
\end{split}
\end{equation*}
Observe that by (\ref{ytheta}), the family $((\nabla f(Y_{t_{k+1}}^{\theta}(t_k,x)))^{\ast} \,\partial_{\theta}Y_{t_{k+1}}^{\theta}(t_k,x), \theta \in \Theta)$ is uniformly integrable. This justifies that we can interchange $\partial_{\theta}$
and $\widetilde{\E}$.
Note that here $\delta(V)\equiv\delta(V{\bf 1}_{[t_k, t_{k+1}]}(\cdot))$ for any $V\in \textnormal{Dom}\ \delta$. On the other hand, using the stochastic flow property, we have that
\begin{equation*}\begin{split}
\partial_{\theta}\widetilde{\E}\left[f(Y_{t_{k+1}}^{\theta}(t_k,x))\right]=\int_{\R^d}f(y)\partial_{\theta}p^{\theta}(\Delta_n,x,y)dy,
\end{split}
\end{equation*}
and
\begin{align*}
&\widetilde{\E}\left[f(Y_{t_{k+1}}^{\theta}(t_k,x))\delta\left(U^{\theta}(t_k,x) \partial_{\theta}Y_{t_{k+1}}^{\theta}(t_k,x)\right)\right]\\
&=\widetilde{\E}\left[f(Y_{t_{k+1}}^{\theta})\delta\left(U^{\theta}(t_k,x) \partial_{\theta}Y_{t_{k+1}}^{\theta}(t_k,x)\right)\Big\vert Y_{t_{k}}^{\theta}=x\right]\\
&=\int_{\R^d}f(y)\widetilde{\E}\left[\delta\left(U^{\theta}(t_k,x) \partial_{\theta}Y_{t_{k+1}}^{\theta}(t_k,x)\right)\Big\vert Y_{t_{k}}^{\theta}=x, Y_{t_{k+1}}^{\theta}=y\right]p^{\theta}(\Delta_n,x,y)dy,
\end{align*}
which finishes the desired proof.
\end{proof}

\subsection{Proof of Lemma \ref{c3delta}}
\begin{proof}
From \eqref{partialx} and It\^o's formula, 
\begin{align*} 
&(\nabla_xY_{t}^{\theta}(t_k,x))^{-1}=\textup{I}_d-\int_{t_k}^t(\nabla_xY_{s}^{\theta}(t_k,x))^{-1}\left(\nabla_xb(\theta,Y_s^{\theta}(t_k,x))-\sum_{i=1}^{d}(\nabla_x\sigma_i(Y_s^{\theta}(t_k,x)))^2\right)ds\\
&\qquad-\sum_{i=1}^{d}\int_{t_k}^t(\nabla_xY_{s}^{\theta}(t_k,x))^{-1}\nabla_x\sigma_i(Y_s^{\theta}(t_k,x))dW_s^i\\
&\qquad+\int_{t_k}^t\int_{\mathbb{R}_0^d}(\nabla_xY_{s}^{\theta}(t_k,x))^{-1}\left(\textup{I}_d+\nabla_xc(Y_{s-}^{\theta}(t_k,x),z)\right)^{-1}(\nabla_xc(Y_{s-}^{\theta}(t_k,x),z))^2\nu(dz)ds\\
&\qquad-\int_{t_k}^t\int_{\mathbb{R}_0^d}(\nabla_xY_{s}^{\theta}(t_k,x))^{-1}\left(\textup{I}_d+\nabla_xc(Y_{s-}^{\theta}(t_k,x),z)\right)^{-1}\nabla_xc(Y_{s-}^{\theta}(t_k,x),z)\widetilde{M}(ds,dz),
\end{align*}
which, together with \eqref{partialtheta} and It\^o's formula again, implies that
\begin{equation}\label{identity}\begin{split}
(\nabla_xY_{t_{k+1}}^{\theta}(t_k,x))^{-1}\partial_{\theta}Y_{t_{k+1}}^{\theta}(t_k,x)=\int_{t_k}^{t_{k+1}}(\nabla_xY_{s}^{\theta}(t_k,x))^{-1}\partial_{\theta}b(\theta,Y_{s}^{\theta}(t_k,x))ds.
\end{split}
\end{equation}

Then, using the product rule \cite[(1.48)]{N} and \eqref{identity}, we obtain that
\begin{align*} 
&\delta\left(U^{\theta}(t_k,x) \partial_{\theta}Y_{t_{k+1}}^{\theta}(t_k,x)\right)\\
&=\delta\left(\sigma^{-1}(Y_{\cdot}^{\theta}(t_k,x))\nabla_x Y_{\cdot}^{\theta}(t_k,x)(\nabla_x Y_{t_{k+1}}^{\theta}(t_k,x))^{-1}\partial_{\theta}Y_{t_{k+1}}^{\theta}(t_k,x)\right)\\
&=(\partial_{\theta}Y_{t_{k+1}}^{\theta}(t_k,x))^{\ast} ((\nabla_x Y_{t_{k+1}}^{\theta}(t_k,x))^{-1})^{\ast} \int_{t_k}^{t_{k+1}} (\nabla_x Y_{s}^{\theta}(t_k,x))^{\ast} (\sigma^{-1}(Y_{s}^{\theta}(t_k,x)))^{\ast} dW_s\\
&\qquad-\int_{t_k}^{t_{k+1}} \text{tr} \left(D_s\left((\partial_{\theta}Y_{t_{k+1}}^{\theta}(t_k,x))^{\ast} ((\nabla_x Y_{t_{k+1}}^{\theta}(t_k,x))^{-1})^{\ast}\right)   \sigma^{-1}(Y_{s}^{\theta}(t_k,x)) \nabla_x Y_{s}^{\theta}(t_k,x)  \right)ds\\
&=\int_{t_k}^{t_{k+1}}((\nabla_xY_{s}^{\theta}(t_k,x))^{-1}\partial_{\theta}b(\theta,Y_{s}^{\theta}(t_k,x)))^{\ast} ds \int_{t_k}^{t_{k+1}} (\nabla_x Y_{s}^{\theta}(t_k,x))^{\ast} (\sigma^{-1}(Y_{s}^{\theta}(t_k,x)))^{\ast} dW_s\\
&-\int_{t_k}^{t_{k+1}} \int_s^{t_{k+1}}\text{tr} \left(D_s\left(((\nabla_xY_{u}^{\theta}(t_k,x))^{-1}\partial_{\theta}b(\theta,Y_{u}^{\theta}(t_k,x)))^{\ast}\right)  \sigma^{-1}(Y_{s}^{\theta}(t_k,x))  \nabla_x Y_{s}^{\theta}(t_k,x)  \right) duds.
\end{align*}

We next add and subtract the matrix  $((\nabla_xY_{t_k}^{\theta}(t_k,x))^{-1}\partial_{\theta}b(\theta,Y_{t_k}^{\theta}(t_k,x)))^{\ast}$ in the first integral and the matrix $(\nabla_x Y_{t_k}^{\theta}(t_k,x))^{\ast} (\sigma^{-1}(Y_{t_k}^{\theta}(t_k,x)))^{\ast}$ in the second integral. This, together with the fact that $Y_{t_k}^{\theta}(t_k,x)=Y_{t_k}^{\theta}=x$, yields 
\begin{equation}
\delta\left(U^{\theta}(t_k,x) \partial_{\theta}Y_{t_{k+1}}^{\theta}(t_k,x)\right)=\Delta_n(\sigma^{-1}(Y_{t_k}^{\theta})\partial_{\theta}b(\theta,Y_{t_k}^{\theta}))^{\ast} (W_{t_{k+1}}-W_{t_{k}})-R_{1}^{\theta,k}+R_{2}^{\theta,k}+R_{3}^{\theta,k}.\label{c3e0}
\end{equation}

On the other hand, by equation \eqref{c3eq1rajoute} we have that
\begin{equation*} \begin{split}
W_{t_{k+1}}-W_{t_{k}}&=\sigma^{-1}(Y_{t_k}^{\theta})\bigg(Y_{t_{k+1}}^{\theta}-Y_{t_{k}}^{\theta}-b(\theta,Y_{t_k}^{\theta})\Delta_n-\int_{t_k}^{t_{k+1}}\left(b(\theta,Y_{s}^{\theta})-b(\theta,Y_{t_k}^{\theta})\right)ds\\
&\qquad-\int_{t_k}^{t_{k+1}}\left(\sigma(Y_{s}^{\theta})-\sigma(Y_{t_k}^{\theta})\right)dW_s-\int_{t_k}^{t_{k+1}}
\int_{\mathbb{R}_0^d}c(Y_{s-}^{\theta},z)\widetilde{M}(ds,dz)\bigg).
\end{split}
\end{equation*}
This, together with \eqref{c3e0}, conclude the desired result.
\end{proof}

\subsection{Proof of Lemma \ref{c3lemma1}}
\begin{proof} Observe that
\begin{equation*}\begin{split} 
&\dfrac{d\widehat{\P}}{d \widehat{Q}_k^{\theta_1,\theta}}-1=\int_0^1\int_{t_k}^{t_{k+1}}\left(b(\theta,X_t)-b(\theta_1,X_t)\right)^{\ast}\sigma^{-1}(X_t)\\
&\qquad \left(dB_t-\alpha\sigma^{-1}(X_t)\left(b(\theta,X_t)-b(\theta_1,X_t)\right)dt\right)\dfrac{d\widehat{\P}^{\alpha}}{d \widehat{Q}_k^{\theta_1,\theta}}d\alpha.
\end{split}
\end{equation*}

Consider the process $W=(W_t, t\in[t_k,t_{k+1}])$ defined by
\begin{equation*}
W_t:=B_t-\alpha \int_{t_k}^{t}\sigma^{-1}(X_s)\left(b(\theta,X_s)-b(\theta_1,X_s)\right)ds.
\end{equation*}
By Girsanov's theorem, $W$ is a Brownian  motion under $\widehat{\P}^{\alpha}$. 

Using Girsanov's theorem, H\"older's and Burkholder-David-Gundy's inequalities, the mean value theorem, and hypotheses {\bf(A2)}, {\bf(A4)}\textnormal{(b)}, together with Lemma \ref{c3moment3} \textnormal{(ii)}, we get that
\begin{align*}
&\left\vert\E_{\widehat{Q}_k^{\theta_1,\theta}}\left[V\left(\dfrac{d\widehat{\P}}{d \widehat{Q}_k^{\theta_1,\theta}}-1\right)\Big\vert X_{t_k}^{\theta}\right]\right\vert\\
&=\left\vert\int_0^1 \E_{\widehat{\P}^{\alpha}}\left[V\int_{t_k}^{t_{k+1}}\left(b(\theta,X_t)-b(\theta_1,X_t)\right)^{\ast}\sigma^{-1}(X_t)dW_t\Big\vert X_{t_k}^{\theta}\right] d\alpha\right\vert\\
&\leq\int_0^1\left( \E_{\widehat{\P}^{\alpha}}\left[\vert V\vert^p\Big\vert X_{t_k}^{\theta}\right]\right)^{\frac{1}{p}}\left( \E_{\widehat{\P}^{\alpha}}\left[\left\vert\int_{t_k}^{t_{k+1}}\left(b(\theta,X_t)-b(\theta_1,X_t)\right)^{\ast}\sigma^{-1}(X_t)dW_t\right\vert^q\Big\vert X_{t_k}^{\theta}\right]\right)^{\frac{1}{q}}d\alpha\\
&\leq\dfrac{C}{\sqrt{n}}\left(1+\vert X_{t_k}^{\theta}\vert^{q_0}\right)\int_0^1\left(\E_{\widehat{\P}^{\alpha}}\left[\vert V\vert^p\Big\vert X_{t_k}^{\theta}\right]\right)^{\frac{1}{p}}d\alpha,
\end{align*}
for some constants $C, q_0>0$, where $p,q>1$ and $\frac{1}{p}+\frac{1}{q}=1$. Thus, the result follows.
\end{proof}

\subsection{Transition density estimates}
For any $t>s$ and $i\geq 0$, we denote by $q_{(i)}^{\theta}(t-s,x,y)$ the transition density of $X_t^{\theta}$ conditioned on $X_s^{\theta}=x$ and $N_t-N_s=i$, where $N_t=N([0,t]\times\R^d)$, $t\geq 0$ is a Poisson process with intensity $\lambda=\int_{\R^d}\nu(dz)$. That is, 
\begin{equation}\label{c3density}\begin{split}
p^{\theta}(t-s,x,y)=\sum_{i=0}^{\infty}q_{(i)}^{\theta}(t-s,x,y)e^{-\lambda(t-s)}\frac{(\lambda(t-s))^i}{i!}.
\end{split}
\end{equation}

For any $t>s$ and $i\geq 1$, we denote by $q_{(i)}^{\theta}(t-s,x,y;z_1,\ldots,z_i)$ the transition density of $X_t^{\theta}$ conditioned on $X_s^{\theta}=x, N_t-N_s=i$ and $\widehat{\Lambda}_{[s,t]}=\{z_1,\ldots, z_i\}$, where $\widehat{\Lambda}_{[s,t]}$ are the jump amplitudes of $\widehat{Z}$ on the interval $[s,t]$, i.e., $\widehat{\Lambda}_{[s,t]}:=\{\Delta \widehat{Z}_u; s\leq u\leq t\}$.

We next show the upper bound estimates for the transition density $q_{(0)}^{\theta}(t-s,x,y)$ and $q_{(1)}^{\theta}(\Delta_n,x,y;z)$. For this, we transform equation \eqref{c3eq1} by introducing a new $\R^d$-valued process $V^{\theta}=(V_t^{\theta})_{t\geq 0}$ defined by $V_t^{\theta}:=f(X_t^{\theta})=\frac{X_t^{\theta}}{\sqrt{1+\vert X_t^{\theta}\vert^2}}$. Notice that $\vert V_t^{\theta}\vert <1$, for all $t\geq 0$. On the other hand,
$$
\nabla f(x)=\frac{1}{\left(1+\vert x\vert^2\right)^{\frac{3}{2}}}\left(\textup{I}_d+A(x)\right),
$$
where the matrix $A(x)=\vert x\vert^2\textup{I}_d-xx^{\ast}$ is symmetric and non-negative definite, for all $x\in\R^d$. Moreover, it is easy to check that $\det\nabla f(x)=(1+\vert x\vert^2)^{-\frac{d}{2}-1}>0$, for any $x\in\R^d$.

By It\^o's formula, $V^{\theta}$ satisfies the following SDE with jumps
\begin{align}
dV_t^{\theta}&=\left(\nabla f(X_t^{\theta})b(\theta,X_t^{\theta})+\dfrac{1}{2}\textup{tr}\left\{\nabla^2 f(X_t^{\theta})\sigma^2(X_t^{\theta})\right\}-\int_{\R_0^d}\nabla f(X_t^{\theta})c(X_{t-}^{\theta},z)\nu(dz)\right)dt\notag\\
&\qquad+\nabla f(X_t^{\theta})\sigma(X_t^{\theta})dB_t+\int_{\R_0^d}\left(f(X_{t-}^{\theta}+c(X_{t-}^{\theta},z))-f(X_{t-}^{\theta})\right)N(dt,dz),\label{V}
\end{align}
where $X_t^{\theta}=\frac{V_t^{\theta}}{\sqrt{1-\vert V_t^{\theta}\vert^2}}$ and $\textup{tr}\{\nabla^2 f(X_t^{\theta})\sigma^2(X_t^{\theta})\}=(\textup{tr}\{\nabla^2 f_i(X_t^{\theta})\sigma^2(X_t^{\theta})\})_{i\in\{1,\ldots,d\}}$. Observe that the drift and diffusion coefficients of equation \eqref{V} are uniformly bounded and continuously differentiable with bounded partial derivatives. On the other hand, the diffusion coefficient does not satisfy an uniform ellipticity condition but satisfies, by hypothesis {\bf(A8)}(b), an ellipticity assumption in all $\R^d$. The new jump coefficient is given by $\widetilde{c}(v,z):=f(\frac{v}{\sqrt{1-\vert v\vert^2}}+c(\frac{v}{\sqrt{1-\vert v\vert^2}},z))-v$, where $\vert v\vert <1$. Note that when $X^{\theta}$ has a jump at time $\tau$ with jump size $c(X_{\tau-}^{\theta},\Delta \widehat{Z}_{\tau})$, then $V^{\theta}$ has the jump size $\widetilde{c}(V_{\tau-}^{\theta},\Delta \widehat{Z}_{\tau})$ at the same time.

For any $t>s$, we denote by $q_{(0)}^{\theta,V}(t-s,x,y)$ the transition density of $V_t^{\theta}$ conditioned on $V_s^{\theta}=x$ and $N_t-N_s=0$. Under conditions {\bf(A1)}, {\bf(A2)}, {\bf(A4)}{(a)} and {\bf(A8)}(b), by \cite[Corollary 3.25]{KS85} and \cite[Theorem 9 iii)]{B03}, for any $\theta\in\Theta$, there exist constants $c, C>1$ such that for all $0<t\leq 1$, and $x, y \in\R^d$,
\begin{equation}\label{c3eq3}
q_{(0)}^{\theta,V}(t,x,y)\leq \dfrac{C}{t^{d/2}}e^{-\frac{\vert y-x\vert^2}{ct}}.
\end{equation}

For any $t>s$, we denote by $q_{(1)}^{\theta,V}(t-s,x,y;z)$ the transition density of $V_t^{\theta}$ conditioned on $V_s^{\theta}=x, N_t-N_s=1$ and $\widehat{\Lambda}_{[s,t]}=\{z\}$. Then, $q_{(1)}^{\theta,V}(\Delta_n,x,y;z)$ satisfies the following estimate.
\begin{lemma}\label{q1v} Under conditions {\bf(A1)}, {\bf(A2)}, {\bf(A4)}\textnormal{(a)} and {\bf(A8)}, for all $\theta\in\Theta$, there exist constants $C_1, C_2>0$ such that for all $x, y\in \R^d$, and $z \in \R_0^d$,
\begin{equation*}
q_{(1)}^{\theta,V}(\Delta_n,x,y;z)\leq\dfrac{C_1\beta^d(z)}{\eta(z)\Delta_n^{d/2}}e^{-\frac{\left\vert y-x-\widetilde{c}(x,z)\right\vert^2}{C_2\beta^2(z)\Delta_n}},
\end{equation*}
where $\eta(z)$ and $\beta(z)$ are defined in {\bf(A8)}\textnormal{(a)}.
\end{lemma}
\begin{proof}
Using the Chapman-Kolmogorov equation and the fact that the distribution of the jump time conditioned on $N_{t_{k+1}}-N_{t_{k}}=1$ is uniform on $[t_k,t_{k+1}]$, together with (\ref{c3eq3}), we get that for some constants $c, C>1$,
\begin{equation*} \begin{split}
q_{(1)}^{\theta,V}(\Delta_n,x,y;z)&=\dfrac{1}{\Delta_n}\int_{t_k}^{t_{k+1}}\int_{\mathbb{R}^d}q_{(0)}^{\theta,V}(t-t_k,x,v)q_{(0)}^{\theta,V}(t_{k+1}-t,v+\widetilde{c}(v,z),y)dvdt\\
&\leq\dfrac{C}{\Delta_n}\int_{t_k}^{t_{k+1}}\int_{\mathbb{R}^d}\dfrac{1}{(t-t_k)^{d/2}}e^{-\frac{\vert v-x\vert^2}{c(t-t_k)}}\dfrac{1}{(t_{k+1}-t)^{d/2}}e^{-\frac{\vert y-v-\widetilde{c}(v,z)\vert^2}{c(t_{k+1}-t)}}dvdt.
\end{split}
\end{equation*}

We next use the change of variables $u:=\psi(v):=v+\widetilde{c}(v,z)-x-\widetilde{c}(x,z)$. Observe that $\psi(x)=0$ and the gradient satisfies
$\det\nabla\psi(v)=\det\nabla f(\frac{v}{\sqrt{1-\vert v\vert^2}}+c(\frac{v}{\sqrt{1-\vert v\vert^2}},z))>0$, for all $\vert v\vert<1$ and $z\in\R_0^d$. Therefore, the mapping $v\mapsto\psi(v)$ admits an inverse function $\psi^{-1}$. On the other hand, $\det\nabla\psi(v)\geq\eta(z)$, and using the mean value theorem there exists $\alpha\in(0,1)$ such that
$$
\left\vert\psi^{-1}(u)-\psi^{-1}(0)\right\vert^2=\left\vert\nabla\psi^{-1}(\alpha u) u\right\vert^2\geq \dfrac{\vert u\vert^2}{\beta^2(z)},
$$
where we have used hypothesis {\bf(A8)}\textnormal{(a)}. Therefore, 
\begin{align*}
&q_{(1)}^{\theta,V}(\Delta_n,x,y;z)\\
&\leq\dfrac{C}{\Delta_n}\int_{t_k}^{t_{k+1}}\int_{\mathbb{R}^d}\dfrac{1}{(t-t_k)^{\frac{d}{2}}}e^{-\frac{\left\vert\psi^{-1}(u)-\psi^{-1}(0)\right\vert^2}{c(t-t_k)}} \dfrac{1}{(t_{k+1}-t)^{d/2}}e^{-\frac{\left\vert y-u-x-\widetilde{c}(x,z)\right\vert^2}{c(t_{k+1}-t)}}\dfrac{1}{\vert\det\nabla\psi(v)\vert}dudt\\
&\leq\dfrac{C}{\eta(z)\Delta_n}\int_{t_k}^{t_{k+1}}\int_{\mathbb{R}^d}\dfrac{1}{(t-t_k)^{d/2}}e^{-\frac{\vert u\vert^2}{c\beta^2(z)(t-t_k)}}\dfrac{1}{(t_{k+1}-t)^{\frac{d}{2}}}e^{-\frac{\left\vert y-u-x-\widetilde{c}(x,z)\right\vert^2}{c(t_{k+1}-t)}}dudt\\
&=\dfrac{C\beta^d(z)}{\eta(z)\Delta_n}\int_{t_k}^{t_{k+1}}\dfrac{1}{(c\left(\beta^2(z)(t-t_k)+t_{k+1}-t\right))^{d/2}} e^{-\frac{\left\vert y-x-\widetilde{c}(x,z)\right\vert^2}{c\left(\beta^2(z)(t-t_k)+t_{k+1}-t\right)}}dt.
\end{align*}

Next, observe that
\begin{align*}
c\left(\beta^2(z)\wedge 1\right)\Delta_n \leq c\left(\beta^2(z)(t-t_k)+t_{k+1}-t\right)\leq c\left(\beta^2(z)\vee 1\right)\Delta_n,
\end{align*}
from where we deduce that
\begin{equation*} \begin{split}
q_{(1)}^{\theta,V}(\Delta_n,x,y;z)\leq\dfrac{C\beta^d(z)}{\eta(z)(\left(\beta^2(z)\wedge 1\right)\Delta_n)^{d/2}}e^{-\frac{\left\vert y-x-\widetilde{c}(x,z)\right\vert^2}{c\left(\beta^2(z)\vee 1\right)\Delta_n}},
\end{split}
\end{equation*}
for some constant $C>0$. Therefore, the desired result follows.
\end{proof}

\begin{lemma}\label{c3lemma8} Under conditions {\bf(A1)}, {\bf(A2)}, {\bf(A4)}\textnormal{(a)} and {\bf(A8)}, for all $\theta\in\Theta$, there exist constants $c, C>1$ and $C_1, C_2>0$ such that for all $t>s$, $x, y\in \R^d$, and $z \in \R_0^d$,
\begin{align}
&q_{(0)}^{\theta}(t-s,x,y)\leq\dfrac{C}{(t-s)^{d/2}}e^{-\frac{\left\vert f(y)-f(x)\right\vert^2}{c(t-s)}}\det\nabla f(y),\label{q0}\\
&q_{(1)}^{\theta}(\Delta_n,x,y;z)\leq \dfrac{C_1\beta^d(z)}{\eta(z)\Delta_n^{d/2}}e^{-\frac{\left\vert f(y)-f(x)-\widetilde{c}(f(x),z)\right\vert^2}{C_2\beta^2(z)\Delta_n}}\det\nabla f(y).\label{q1}
\end{align}
\end{lemma}
\begin{proof}
Since $\det\nabla f(y)>0$, for all $x, y\in\R^d$ and $t>s$, we have that 
\begin{equation*}\begin{split}
\P\left(X_t^{\theta}\in dy\Big\vert X_s^{\theta}=x, N_t-N_s=0\right)=\P\left(V_t^{\theta}\in df(y)\Big\vert V_s^{\theta}=f(x), N_t-N_s=0\right),
\end{split}
\end{equation*}
which, together with (\ref{c3eq3}), implies that
\begin{equation*}\begin{split}
q_{(0)}^{\theta}(t-s,x,y)=q_{(0)}^{\theta,V}\left(t-s,f(x),f(y)\right)\, \det\nabla f(y) \leq\dfrac{C}{(t-s)^{d/2}}e^{-\frac{\left\vert f(y)-f(x)\right\vert^2}{c(t-s)}}\det\nabla f(y).
\end{split}
\end{equation*}
This concludes \eqref{q0}. Similarly,
\begin{equation*}\begin{split}
&\P\left(X_{t_{k+1}}^{\theta}\in dy\Big\vert X_{t_{k}}^{\theta}=x,N_{t_{k+1}}-N_{t_{k}}=1, \widehat{\Lambda}_{[t_k,t_{k+1}]}=\{z\}\right)\\
&=\P\left(V_{t_{k+1}}^{\theta}\in df(y)\Big\vert V_{t_{k}}^{\theta}=f(x), N_{t_{k+1}}-N_{t_{k}}=1, \widehat{\Lambda}_{[t_k,t_{k+1}]}=\{z\}\right),
\end{split}
\end{equation*}
which, together with Lemma \ref{q1v}, implies that
\begin{equation*}\begin{split}
q_{(1)}^{\theta}(\Delta_n,x,y;z)&=q_{(1)}^{\theta,V}(\Delta_n,f(x),f(y);z)\,\det\nabla f(y)\\
&\leq\dfrac{C_1\beta^d(z)}{\eta(z)\Delta_n^{d/2}}e^{-\frac{\left\vert f(y)-f(x)-\widetilde{c}(f(x),z)\right\vert^2}{C_2\beta^2(z)\Delta_n}}\det\nabla f(y).
\end{split}
\end{equation*}
Therefore, the desired result follows.
\end{proof}

Under the conditions in Subsection \ref{bounded}, the upper bound estimates of the transition densities are as follows.
\begin{lemma} Under conditions {\bf(A1')}, {\bf(A2)}, {\bf(A4)}\textnormal{(a)} and {\bf(A8')}, for all $\theta\in\Theta$, there exist constants $C, c>1$,  $C_1, C_2>0$ such that for all $0\leq t\leq 1$, $x, y\in \R^d$, and $z \in \R_0^d$,
	\begin{align}
	&q_{(0)}^{\theta}(t,x,y)\leq\dfrac{C}{t^{d/2}}e^{-\frac{\vert y-x\vert^2}{ct}},\label{bq0}\\
	&q_{(1)}^{\theta}(\Delta_n,x,y;z)\leq \dfrac{C_1}{\Delta_n^{d/2}}e^{-\frac{\left\vert y-x-c(x,z)\right\vert^2}{C_2\Delta_n}}.\label{bq1}
	\end{align}
\end{lemma}
\begin{proof} 
	The estimate \eqref{bq0} follows from Azencott \cite[page 478]{A84}. Using the change of variables $u:=\psi(v):=v+c(v,z)-x-c(x,z)$, and condition {\bf(A8')}, the proof of \eqref{bq1} follows along the same lines as that of Lemma \eqref{q1v} and is therefore omitted.
\end{proof}

\subsection{Expression of conditional expectations via transition densities}
\label{expression}

Consider the events
$
\widehat{J}_{i,k}=\{ N_{t_{k+1}}-N_{t_{k}}=i\}$ and $
\widetilde{J}_{i,k}=\{ M_{t_{k+1}}-M_{t_{k}}=i\}$, for $k \in \{0,...,n-1\}$ and $i\in\{0,1\}$, where $M_t=M([0,t]\times \R^d)$, $t\geq 0$ is a Poisson process with intensity $\lambda$. We denote by $\widetilde{\Lambda}_{[s,t]}$ the jump amplitudes of $\widetilde{Z}$ on the interval $[s,t]$, i.e, $\widetilde{\Lambda}_{[s,t]}:=\{\Delta\widetilde{Z}_u; s\leq u\leq t\}$, and by $\mu(dz)=\frac{\nu(dz)}{\lambda}$ the jump size distribution of $\widetilde{Z}$.

For each $k \in \{0,...,n-1\}$, we consider the events
\begin{equation*}\begin{split}
&\widehat{A}_{k,n}=\left\{\left\vert\int_{t_k}^{t_{k+1}}\int_{\R_0^d}zN(ds,dz)\right\vert\in[\rho_1\Delta_n^{\upsilon},\rho_2\Delta_n^{-\gamma}]\right\},\\
&\widetilde{A}_{k,n}=\left\{\left\vert\int_{t_k}^{t_{k+1}}\int_{\R_0^d}zM(ds,dz)\right\vert\in[\rho_1\Delta_n^{\upsilon},\rho_2\Delta_n^{-\gamma}]\right\},
\end{split}
\end{equation*}
where $\rho_1,\upsilon$ are constants from hypothesis {\bf (A7)}, and $\rho_2$, $\gamma$ are some positive constants with $\gamma\in(0,\frac{1}{2})$. We then denote by $\widehat{A}_{k,n}^c$ and $\widetilde{A}_{k,n}^c$ their corresponding complementary events.

Set $I={\{z\in\R_0^d: \rho_1\Delta_n^{\upsilon}\leq\vert z \vert\leq\rho_2\Delta_n^{-\gamma}\}}$ or $I={\{a\in\R_0^d: \rho_1\Delta_n^{\upsilon}\leq\vert a \vert\leq\rho_2\Delta_n^{-\gamma}\}}$, and recall that $e_k(\theta)=\left(\partial_{\theta}
b(\theta,X_{t_k})\right)^{\ast}(\sigma\sigma^{\ast})^{-1}(X_{t_k})$. As in \cite[Lemma 2.2]{KNT14}, we have the following expressions for the conditional expectations in terms of the transition densities. 
\begin{lemma}\label{c3lemma7} Under conditions {\bf(A1)}, {\bf(A2)} and {\bf(A4)}\textnormal{(a)}, for all $k \in \{0,...,n-1\}$ and $\theta\in\Theta$,
\begin{equation} \label{c31}\begin{split}
&\E_{\widehat{Q}_k^{\theta,\theta_0}}\bigg[{\bf 1}_{\widehat{J}_{0,k}}\bigg(e_k(\theta)\widetilde{\E}_{X_{t_k}}^{\theta}\left[{\bf 1}_{\widetilde{J}_{1,k}}{\bf{1}}_{\widetilde{A}_{k,n}}\int_{t_k}^{t_{k+1}}
\int_{\R_0^d}c(Y_{t_k}^{\theta},z)M(ds,dz)\bigg\vert Y_{t_{k+1}}^{\theta}=X_{t_{k+1}}\right]\bigg)^2\Big\vert X_{t_{k}}\bigg]\\
&=\int_{\R^d}\left(\dfrac{e_k(\theta)\int_{I}q_{(1)}^{\theta}(\Delta_n,X_{t_k},y;a)c\left(X_{t_k},a\right)\mu(da)e^{-\lambda\Delta_n}\lambda\Delta_n}{p^{\theta}(\Delta_n,X_{t_k},y)}\right)^2q_{(0)}^{\theta}(\Delta_n,X_{t_k},y)e^{-\lambda\Delta_n}dy,
\end{split}
\end{equation}
\begin{equation} \label{c32} \begin{split}
&\E_{\widehat{Q}_k^{\theta,\theta_0}}\left[{\bf 1}_{\widehat{J}_{1,k}}\left(e_k(\theta){\bf{1}}_{\widehat{A}_{k,n}}\int_{t_k}^{t_{k+1}}\int_{\R_0^d}c(X_{t_k},z)N(ds,dz)\widetilde{\E}_{X_{t_k}}^{\theta}\left[{\bf 1}_{\widetilde{J}_{0,k}}\bigg\vert Y^{\theta}_{t_{k+1}}=X_{t_{k+1}}\right]\right)^2\Big\vert X_{t_{k}}\right]\\
&=\int_{I}\int_{\R^d}\left(\dfrac{q_{(0)}^{\theta}(\Delta_n,X_{t_k},y)e^{-\lambda\Delta_n}}{p^{\theta}(\Delta_n,X_{t_k},y)}\right)^2q_{(1)}^{\theta}(\Delta_n,X_{t_k},y;a)e^{-\lambda \Delta_n}\lambda\Delta_n\left(e_k(\theta)c(X_{t_k},a)\right)^2dy \mu(da),
\end{split}
\end{equation}
and
\begin{equation} \label{c33}
\begin{split}
&\E_{\widehat{Q}_k^{\theta,\theta_0}}\bigg[{\bf 1}_{\widehat{J}_{1,k}}\bigg(e_k(\theta)\bigg({\bf{1}}_{\widehat{A}_{k,n}}\int_{t_k}^{t_{k+1}}\int_{\R_0^d}c(X_{t_k},z)N(ds,dz)\widetilde{\E}_{X_{t_k}}^{\theta}\left[{\bf 1}_{\widetilde{J}_{1,k}}\bigg\vert Y^{\theta}_{t_{k+1}}=X_{t_{k+1}}\right]\\
&\qquad \qquad -\widetilde{\E}_{X_{t_k}}^{\theta}\left[{\bf 1}_{\widetilde{J}_{1,k}}{\bf{1}}_{\widetilde{A}_{k,n}}\int_{t_k}^{t_{k+1}}\int_{\R_0^d}c(Y_{t_k}^{\theta},z)M(ds,dz)\bigg\vert Y^{\theta}_{t_{k+1}}=X_{t_{k+1}}\right]\bigg)\bigg)^2\Big\vert X_{t_{k}}\bigg]\\
&=\int_{I}\int_{\R^d}\left(\dfrac{e_k(\theta)\int_{I}\left(c(X_{t_k},z)-c(X_{t_k},a)\right)q_{(1)}^{\theta}(\Delta_n,X_{t_k},y;a)\mu(da)e^{-\lambda\Delta_n}\lambda \Delta_n}{p^{\theta}(\Delta_n,X_{t_k},y)}\right)^2\\
&\qquad \qquad \times q_{(1)}^{\theta}(\Delta_n,X_{t_k},y;z)e^{-\lambda\Delta_n}\lambda \Delta_n dy\mu(dz).
\end{split}
\end{equation}
\end{lemma}
\begin{proof}
Using Bayes' formula, we get that
\begin{equation*} \begin{split}
&\widetilde{\E}_{X_{t_k}}^{\theta}\left[{\bf 1}_{\widetilde{J}_{1,k}}{\bf{1}}_{\widetilde{A}_{k,n}}\int_{t_k}^{t_{k+1}}
\int_{\R_0^d}c(Y_{t_k}^{\theta},z)M(ds,dz)\bigg\vert Y_{t_{k+1}}^{\theta}=X_{t_{k+1}}\right]\\
&=\dfrac{\widetilde{\E}_{X_{t_k}}^{\theta}\left[c(Y_{t_k}^{\theta},\widetilde{\Lambda}_{[t_k,t_{k+1}]}){\bf 1}_{\{\vert\widetilde{\Lambda}_{[t_k,t_{k+1}]}\vert\in I\}}{\bf 1}_{\{Y_{t_{k+1}}^{\theta}=X_{t_{k+1}}\}}\Big\vert \widetilde{J}_{1,k}\right]\widetilde{\P}_{X_{t_k}}^{\theta}\left(\widetilde{J}_{1,k}\right)}{p^{\theta}(\Delta_n,X_{t_k},X_{t_{k+1}})}\\
&=\dfrac{\int_{I}q_{(1)}^{\theta}(\Delta_n,X_{t_k},X_{t_{k+1}};a)c\left(X_{t_k},a\right)\mu(da)e^{-\lambda \Delta_n}\lambda \Delta_n}{p^{\theta}(\Delta_n,X_{t_k},X_{t_{k+1}})}.
\end{split}
\end{equation*}
This, together with Bayes' formula again, implies that
\begin{align*} 
&\E_{\widehat{Q}_k^{\theta,\theta_0}}\bigg[{\bf 1}_{\widehat{J}_{0,k}}\bigg(e_k(\theta)\widetilde{\E}_{X_{t_k}}^{\theta}\left[{\bf 1}_{\widetilde{J}_{1,k}}{\bf{1}}_{\widetilde{A}_{k,n}}\int_{t_k}^{t_{k+1}}
\int_{\R_0^d}c(Y_{t_k}^{\theta},z)M(ds,dz)\bigg\vert Y_{t_{k+1}}^{\theta}=X_{t_{k+1}}\right]\bigg)^2\Big\vert X_{t_{k}}\bigg]\\
&=\widehat{Q}_k^{\theta,\theta_0}\left(\widehat{J}_{0,k}\big\vert X_{t_{k}}\right)\\
&\qquad\times\E_{\widehat{Q}_k^{\theta,\theta_0}}\left[\left(\dfrac{e_k(\theta)\int_{I}q_{(1)}^{\theta}(\Delta_n,X_{t_k},X_{t_{k+1}};a)c\left(X_{t_k},a\right)\mu(da)e^{-\lambda \Delta_n}\lambda \Delta_n}{p^{\theta}(\Delta_n,X_{t_k},X_{t_{k+1}})}\right)^2\Big\vert \widehat{J}_{0,k}, X_{t_{k}}\right],
\end{align*}
which implies (\ref{c31}). Similarly,
\begin{align*}
&\E_{\widehat{Q}_k^{\theta,\theta_0}}\left[{\bf 1}_{\widehat{J}_{1,k}}\left(e_k(\theta){\bf{1}}_{\widehat{A}_{k,n}}\int_{t_k}^{t_{k+1}}\int_{\R_0^d}c(X_{t_k},z)N(ds,dz)\widetilde{\E}_{X_{t_k}}^{\theta}\left[{\bf 1}_{\widetilde{J}_{0,k}}\bigg\vert Y^{\theta}_{t_{k+1}}=X_{t_{k+1}}\right]\right)^2\Big\vert X_{t_{k}}\right]\\
&=\E_{\widehat{Q}_k^{\theta,\theta_0}}\left[{\bf 1}_{\widehat{J}_{1,k}}\left(e_k(\theta){\bf 1}_{\{\vert\widehat{\Lambda}_{[t_k,t_{k+1}]}\vert\in I\}}c\left(X_{t_k},\widehat{\Lambda}_{[t_k,t_{k+1}]}\right)\right)^2
\left(\dfrac{q_{(0)}^{\theta}(\Delta_n,X_{t_k},X_{t_{k+1}})e^{-\lambda\Delta_n}}{p^{\theta}(\Delta_n,X_{t_k},X_{t_{k+1}})}\right)^2\Big\vert X_{t_{k}}\right]\\
&=\int_{I}\E_{\widehat{Q}_k^{\theta,\theta_0}}\left[\left(\dfrac{q_{(0)}^{\theta}(\Delta_n,X_{t_k},X_{t_{k+1}})e^{-\lambda\Delta_n}}{p^{\theta}(\Delta_n,X_{t_k},X_{t_{k+1}})}\right)^2\Big\vert \widehat{J}_{1,k}, \widehat{\Lambda}_{[t_k,t_{k+1}]}=\{a\}, X_{t_{k}}\right]\left(e_k(\theta)c(X_{t_k},a)\right)^2\\
&\qquad\times\widehat{Q}_k^{\theta,\theta_0}\left(\widehat{\Lambda}_{[t_k,t_{k+1}]}\in da,\widehat{J}_{1,k}\Big\vert X_{t_{k}}\right) \\
&=\int_{I}\int_{\R^d}\left(\dfrac{q_{(0)}^{\theta}(\Delta_n,X_{t_k},y)e^{-\lambda\Delta_n}}{p^{\theta}(\Delta_n,X_{t_k},y)}\right)^2q_{(1)}^{\theta}(\Delta_n,X_{t_k},y;a)e^{-\lambda \Delta_n}\lambda \Delta_n\left(e_k(\theta)c(X_{t_k},a)\right)^2dy \mu(da),
\end{align*}
which shows (\ref{c32}). The proof of (\ref{c33}) follows along the same lines and is therefore omitted.
\end{proof}

\subsection{Large deviation type estimates}
\label{large}

By abuse of notation, we consider the events
$\widehat{J}_{2,k}=\{ N_{t_{k+1}}-N_{t_{k}}\geq 2\}$ and $\widetilde{J}_{2,k}=
\{M_{t_{k+1}}-M_{t_{k}}\geq 2\}$. For $i\in\{0,1,2\}$, set
\begin{align*}
M_i^{\theta}&=\E_{\widehat{Q}_k^{\theta,\theta_0}}\bigg[{\bf 1}_{\widehat{J}_{i,k}}\bigg(e_k(\theta)\bigg(\int_{t_k}^{t_{k+1}}\int_{\R_0^d}c(X_{t_k},z)N(ds,dz)\\
&\qquad-\widetilde{\E}_{X_{t_k}}^{\theta}\left[\int_{t_k}^{t_{k+1}}\int_{\R_0^d}c(Y_{t_k}^{\theta},z)M(ds,dz)\bigg\vert Y^{\theta}_{t_{k+1}}=X_{t_{k+1}}\right]\bigg)\bigg)^2\Big\vert X_{t_k}\bigg].
\end{align*}

In all what follows, to deal with the estimation of $\vert e_k(\theta)\vert$, hypotheses {\bf(A2)} and {\bf(A4)}\textnormal{(b)} will be used repeatedly without being quoted. 

Recall that for the simple L\'evy process in \cite{KNT14}, we used a large deviation principle by conditioning on the number of jumps inside and outside the conditional expectation in order to obtain the large deviation type estimates (see \cite[Lemma 2.4]{KNT14}). For the non-linear model \eqref{c3eq1}, we need to obtain an analogue of \cite[Lemma 2.4]{KNT14}. For this, we use the fact that the study of the asymptotic behavior of the transition density leads us to study the behavior of the transition density under the additional condition on the number of jumps which has to be compared with another transition density with a different number of jumps. This is why one needs to use lower bounds for the transition density and upper bounds for the transition density conditioned on the jump structure in order to show the following large deviation type estimates.

\begin{lemma}\label{c3ordre} Under conditions {\bf(A1)}-{\bf(A3)}, {\bf(A4)}\textnormal{(a)}, \textnormal{(b)}, {\bf(A6)} and {\bf(A8)}, for any $\theta\in\Theta$ and $n$ large enough, there exist constants $C, C_0, C_1>0$ and $q_1>1$ such that for all $\alpha\in(\upsilon+3m\gamma+3\gamma,\frac{1}{2})$, $\alpha_0\in(\frac 14,\frac 12-3\gamma)$, $\varepsilon\in(0,\alpha_0-3m\gamma)$, $q>1$, and $k \in\{0,...,n-1\}$,
\begin{align} \label{c3m1}
M_0^{\theta} & \leq C\left(1+\vert X_{t_k}\vert^{q_1}\right)\left(\Delta_n^{-2m\gamma}e^{-C_0\frac{\Delta_n^{2\alpha-1}}{(1+\vert X_{t_k}\vert^{2})^3}}+\lambda^{\frac{2}{q}}\Delta_n^{1+\frac{1}{q}}+\Delta_n\left(\int_{\{\vert z\vert \leq\rho_1\Delta_n^{\upsilon}\}}\nu(dz)\right)^{\frac{1}{q}}\right),\\ \label{c3m2}
M_1^{\theta} & \leq C\left(1+\vert X_{t_k}\vert^{q_1}\right)\bigg(\Delta_n^{-2m\gamma}e^{-C_0\frac{\Delta_n^{2\alpha-1}}{(1+\vert X_{t_k}\vert^2)^3}}+\lambda^{\frac{2}{q}}\Delta_n^{1+\frac{1}{q}}+\Delta_n^{1+2\varepsilon}\notag\\
&\qquad\qquad+\Delta_n\left(\int_{\{\vert z\vert \leq\rho_1\Delta_n^{\upsilon}\}}\nu(dz)\right)^{\frac{1}{q}}+\Delta_n^{-2m\gamma-\frac{d}{2}}e^{-C_1\frac{\Delta_n^{2\alpha_0+6\gamma-1}}{(1+\vert X_{t_k}\vert^2)^3}}\bigg),\\ \label{c3m3}
M_2^{\theta} & \leq C\lambda^{\frac{2}{q}}\Delta_n^{1+\frac{1}{q}}(1+\vert X_{t_k}\vert^{q_1}).
\end{align}
In particular, \textnormal{(\ref{c3m3})} holds for all $n \geq 1$.
\end{lemma}
\begin{proof}
We start showing (\ref{c3m1}). Multiplying the random variable inside the conditional expectation of $M_0^{\theta}$ by ${\bf{1}}_{\widetilde{A}_{k,n}}+{\bf{1}}_{\widetilde{A}_{k,n}^c}$ and ${\bf 1}_{\widetilde{J}_{0,k}} +{\bf 1}_{\widetilde{J}_{1,k}}+ {\bf 1}_{\widetilde{J}_{2,k}}$, we get that $M_0^{\theta}\leq 3(M_{0,1}^{\theta}+M_{0,2}^{\theta}+M_{0,3}^{\theta})$, where for $i\in\{1,2\}$, setting $\Delta M_k:=M_{t_{k+1}}-M_{t_{k}}$, 
\begin{align*}
M_{0,i}^{\theta}&=\E_{\widehat{Q}_k^{\theta,\theta_0}}\left[{\bf 1}_{\widehat{J}_{0,k}}\left(e_k(\theta)\widetilde{\E}_{X_{t_k}}^{\theta}\left[{\bf 1}_{\widetilde{J}_{i,k}}{\bf{1}}_{\widetilde{A}_{k,n}}\int_{t_k}^{t_{k+1}}
\int_{\R_0^d}c(Y_{t_k}^{\theta},z)M(ds,dz)\bigg\vert Y_{t_{k+1}}^{\theta}=X_{t_{k+1}}\right]\right)^2\Big\vert X_{t_k}\right],\\
M_{0,3}^{\theta}&=\E_{\widehat{Q}_k^{\theta,\theta_0}}\left[{\bf 1}_{\widehat{J}_{0,k}}\left(e_k(\theta)\widetilde{\E}_{X_{t_k}}^{\theta}\left[{\bf 1}_{\{\Delta M_k>0\}}{\bf{1}}_{\widetilde{A}_{k,n}^c}\int_{t_k}^{t_{k+1}}
\int_{\R_0^d}c(Y_{t_k}^{\theta},z)M(ds,dz)\bigg\vert Y_{t_{k+1}}^{\theta}=X_{t_{k+1}}\right]\right)^2\Big\vert X_{t_k}\right].
\end{align*}

By \eqref{c31}, we have that
\begin{equation*} \begin{split}
M_{0,1}^{\theta}=\int_{\R^d}\left(\dfrac{e_k(\theta)\int_{I}q_{(1)}^{\theta}(\Delta_n,X_{t_k},y;a)c\left(X_{t_k},a\right)\mu(da)e^{-\lambda \Delta_n}\lambda \Delta_n}{p^{\theta}(\Delta_n,X_{t_k},y)}\right)^2q_{(0)}^{\theta}(\Delta_n,X_{t_k},y)e^{-\lambda \Delta_n}dy.
\end{split}
\end{equation*}
We next divide the $dy$ integral in $M_{0,1}^{\theta}$ into the subdomains $J_1:=\{y\in\R^d:\vert f(y)-f(X_{t_k})\vert>\frac{\Delta_n^{\alpha}}{(1+\vert X_{t_k}\vert^2)^{\frac{3}{2}}}\}$ and $J_2:=\{y\in\R^d:\vert f(y)-f(X_{t_k})\vert\leq\frac{\Delta_n^{\alpha}}{(1+\vert X_{t_k}\vert^2)^{\frac{3}{2}}}\}$, where $\alpha\in(\upsilon+3m\gamma+3\gamma,\frac{1}{2})$, and call each integral $M_{0,1,1}^{\theta}$ and $M_{0,1,2}^{\theta}$. Therefore, the estimation of $M_{0,1}^{\theta}$ is divided into two parts. The first one uses a large deviation type principle for the continuous process. The other uses the fact that the jump term is significantly bigger than the continuous term. This fact is obtained under condition {\bf(A3)}. We start bounding $M_{0,1,1}^{\theta}$. By \eqref{c3density},
\begin{equation}\label{c3eq25}
p^{\theta}(\Delta_n,X_{t_k},y)\geq \int_{I}q_{(1)}^{\theta}(\Delta_n,X_{t_k},y;a)\mu(da)e^{-\lambda \Delta_n}\lambda \Delta_n.
\end{equation}
Then, using {\bf(A1)}, on $I$, $\vert c(X_{t_k},a)\vert\leq C\Delta_n^{-m\gamma}(1+\vert X_{t_k}\vert)$ for some constant $C>0$, and using \eqref{q0}, together with the equality $e^{-\vert x\vert^2}=e^{-\frac{\vert x\vert^2}{2}}e^{-\frac{\vert x\vert^2}{2}}$, valid for all $x\in\R^d$, we get that 
\begin{equation}\label{esm011}\begin{split}
M_{0,1,1}^{\theta}&\leq C\Delta_n^{-2m\gamma}\left(1+
\vert X_{t_k}\vert^{q_1}\right)\int_{J_1}q_{(0)}^{\theta}(\Delta_n,X_{t_k},y)dy\\
&\leq C\Delta_n^{-2m\gamma}
\left(1+\vert X_{t_k}\vert^{q_1}\right)\int_{J_1}\frac{1}{\Delta_n^{d/2}}e^{-\frac{\left\vert f(y)-f(X_{t_k})\right\vert^2}{c\Delta_n}}\det\nabla f(y)dy\\
&\leq C\Delta_n^{-2m\gamma}\left(1+\vert X_{t_k}\vert^{q_1}\right)e^{-\frac{\Delta_n^{2\alpha-1}}{2c(1+\vert X_{t_k}\vert^2)^3}}\int_{J_1}\frac{1}{\Delta_n^{d/2}}e^{-\frac{\left\vert f(y)-f(X_{t_k})\right\vert^2}{2c\Delta_n}}\det\nabla f(y)dy\\
&\leq C\Delta_n^{-2m\gamma}
\left(1+\vert X_{t_k}\vert^{q_1}\right)e^{-\frac{\Delta_n^{2\alpha-1}}{2c(1+\vert X_{t_k}\vert^2)^3}},
\end{split}
\end{equation}
for some constants $C>0$, $c>1$, $q_1>1$, since the $dy$ integral is Gaussian and thus finite. We next treat $M_{0,1,2}^{\theta}$. Observe that \eqref{c3density} yields 
\begin{equation}\label{c3eq28}\begin{split}
\left(p^{\theta}(\Delta_n,X_{t_k},y)\right)^2\geq q_{(0)}^{\theta}(\Delta_n,X_{t_k},y)e^{-\lambda \Delta_n}\int_{I}q_{(1)}^{\theta}(\Delta_n,X_{t_k},y;a)\mu(da)e^{-\lambda \Delta_n}\lambda \Delta_n.
\end{split}
\end{equation}
Then, using hypothesis {\bf(A1)}, Fubini's theorem and \eqref{q1}, we get that 
\begin{align*}
M_{0,1,2}^{\theta}&\leq C\Delta_n^{-2m\gamma}
\left(1+\vert X_{t_k}\vert^{q_1}\right)e^{-\lambda \Delta_n}\lambda\Delta_n\int_{J_2}\int_{I}q_{(1)}^{\theta}(\Delta_n,X_{t_k},y;a)\mu(da)dy\\
&\leq C\Delta_n^{-2m\gamma}
\left(1+\vert X_{t_k}\vert^{q_1}\right)\int_{I}\int_{J_2}\dfrac{\beta^d(a)}{\eta(a)\Delta_n^{d/2}}e^{-\frac{\left\vert f(y)-f(X_{t_k})-\widetilde{c}(f(X_{t_k}),a)\right\vert^2}{C_2\beta^2(a)\Delta_n}}\det\nabla f(y)dy\nu(da),
\end{align*}
for some constants $C, C_2>0$ and $q_1>1$, since $e^{-\lambda \Delta_n}\Delta_n\leq 1$. 

Then, using the mean value theorem for vector-valued functions, we get that
\begin{align}
\left\vert\widetilde{c}(f(X_{t_k}),a)\right\vert^2&=\left\vert f(X_{t_k}+c(X_{t_k},a))-f(X_{t_k})\right\vert^2 \notag\\
&=\left\vert\left(\int_0^{1}\nabla f(X_{t_k}+\eta c(X_{t_k},a))d\eta\right) c(X_{t_k},a)\right\vert^2 \notag\\
&=\left\vert\left(\int_0^{1}\dfrac{1}{\left(1+\vert X_{t_k}+\eta c(X_{t_k},a)\vert^2\right)^{\frac{3}{2}}}\left(\textup{I}_d+A_\eta\right)d\eta\right)c(X_{t_k},a)\right\vert^2 \notag\\
&\geq\dfrac{C\Delta_n^{6m\gamma}}{(1+\vert X_{t_k}\vert^2)^{3}}\left\vert c(X_{t_k},a)+\left(\int_0^{1}A_\eta d\eta\right)c(X_{t_k},a)\right\vert^2 \label{minoration}\\
&=\dfrac{C\Delta_n^{6m\gamma}}{(1+\vert X_{t_k}\vert^2)^{3}}\bigg\{\vert c(X_{t_k},a)\vert^2+2\int_0^{1}(c(X_{t_k},a))^{\ast}A_\eta c(X_{t_k},a)d\eta \notag\\
&\qquad\qquad+\left\vert \left(\int_0^{1}A_\eta d\eta\right)c(X_{t_k},a)\right\vert^2\bigg\} \notag\\
&\geq \dfrac{C\Delta_n^{6m\gamma}}{(1+\vert X_{t_k}\vert^2)^{3}}\vert c(X_{t_k},a)\vert^2 \notag\\
&\geq \dfrac{C_3\Delta_n^{2(\upsilon+3m\gamma)}}{(1+\vert X_{t_k}\vert^2)^{3}},\notag
\end{align}
for some constant $C_3>0$, since the matrix $A_\eta\equiv A(X_{t_k}+\eta c(X_{t_k},a))$ is non-negative definite. Here, we have used the fact that on $I$, from {\bf(A1)}, $\vert c(X_{t_k},a)\vert\leq C\Delta_n^{-m\gamma}(1+\vert X_{t_k}\vert)$, and from {\bf(A3)}, $\vert c(X_{t_k},a)\vert\geq C\vert a \vert \geq C\rho_1\Delta_n^{\upsilon}$ for some constant $C>0$.

On the other hand, on $J_2$ we have that $\vert f(y)-f(X_{t_k})\vert\leq\frac{\Delta_n^{\alpha}}{(1+\vert X_{t_k}\vert^2)^{3/2}}$. Thus, using the inequality $\vert u+v\vert^2\geq \frac{\vert u\vert^2}{2}-\vert v\vert^2$, valid for all $u, v\in\R^d$, together with the fact that $\alpha>\upsilon+3m\gamma$, we deduce that for $n$ large enough,
\begin{align*}
\left\vert f(y)-f(X_{t_k})-\widetilde{c}(f(X_{t_k}),a)\right\vert^2&\geq\dfrac{\left\vert\widetilde{c}(f(X_{t_k}),a)\right\vert^2}{2}-\vert f(y)-f(X_{t_k})\vert^2\\
&\geq\dfrac{C_3\Delta_n^{2(\upsilon+3m\gamma)}}{2(1+\vert X_{t_k}\vert^2)^{3}}-\dfrac{\Delta_n^{2\alpha}}{(1+\vert X_{t_k}\vert^2)^{3}}\\
&\geq \dfrac{C_4\Delta_n^{2(\upsilon+3m\gamma)}}{(1+\vert X_{t_k}\vert^2)^{3}},
\end{align*}
for some constant $C_4>0$. Therefore, using {\bf(A6)}, we obtain that for $n$ large enough,
\begin{align*}
M_{0,1,2}^{\theta}&\leq C\Delta_n^{-2m\gamma}
\left(1+\vert X_{t_k}\vert^{q_1}\right)e^{-C_5\frac{\Delta_n^{2(\upsilon+3m\gamma+3\gamma)-1}}{(1+\vert X_{t_k}\vert^2)^3}}\\
&\qquad\times\int_{I}\int_{J_2}\dfrac{\beta^{2d}(a)}{\eta(a)(2C_2\beta^2(a)\Delta_n)^{d/2}}e^{-\frac{\left\vert f(y)-f(X_{t_k})-\widetilde{c}(f(X_{t_k}),a)\right\vert^2}{2C_2\beta^2(a)\Delta_n}}\det\nabla f(y)dy\nu(da)\\
&\leq C\Delta_n^{-2m\gamma}
\left(1+\vert X_{t_k}\vert^{q_1}\right)e^{-C_5\frac{\Delta_n^{2(\upsilon+3m\gamma+3\gamma)-1}}{(1+\vert X_{t_k}\vert^2)^3}}\int_{I}\dfrac{\beta^{2d}(a)}{\eta(a)}\nu(da)\\
&\leq C\Delta_n^{-2m\gamma}
\left(1+\vert X_{t_k}\vert^{q_1}\right)e^{-C_5\frac{\Delta_n^{2(\upsilon+3m\gamma+3\gamma)-1}}{(1+\vert X_{t_k}\vert^2)^3}}\left(\int_{\{\vert a\vert \leq 1\}}\nu(da)+\int_{\{\vert a\vert > 1\}}\vert a\vert^{6d+3}\nu(da)\right)\\
&\leq C\Delta_n^{-2m\gamma}
\left(1+\vert X_{t_k}\vert^{q_1}\right)e^{-C_5\frac{\Delta_n^{2(\upsilon+3m\gamma+3\gamma)-1}}{(1+\vert X_{t_k}\vert^2)^3}},
\end{align*}
for some constant $C_5>0$, since the $dy$ integral is Gaussian and thus finite. This shows that for $n$ large enough and $\alpha\in(\upsilon+3m\gamma+3\gamma,\frac{1}{2})$, for some constants $C, C_0>0$,
\begin{align}\label{m01}
M_{0,1}^{\theta} & \leq C\left(1+\vert X_{t_k}\vert^{q_1}\right)\Delta_n^{-2m\gamma}e^{-C_0\frac{\Delta_n^{2\alpha-1}}{(1+\vert X_{t_k}\vert^2)^3}}.
\end{align}

Next, using Jensen's and H\"older's inequalities with $p,q >1$ and $\frac{1}{p}+\frac{1}{q}=1$, and hypotheses {\bf(A1)} and {\bf(A6)}, it holds that
\begin{align}
M_{0,2}^{\theta}&\leq \E\left[{\bf 1}_{\widetilde{J}_{2,k}}\left(e_k(\theta)\int_{t_k}^{t_{k+1}}\int_{\R_0^d}c(Y_{t_k}^{\theta},z)M(ds,dz)\right)^2 \Big\vert Y_{t_k}^{\theta}=X_{t_k}\right] \notag\\
&\leq \left(\P\left(\widetilde{J}_{2,k}\Big\vert Y_{t_k}^{\theta}=X_{t_k}\right)\right)^{\frac{1}{q}}\left(\E\left[\left(e_k(\theta)\int_{t_k}^{t_{k+1}}\int_{\R_0^d}c(Y_{t_k}^{\theta},z)M(ds,dz)\right)^{2p} \Big\vert Y_{t_k}^{\theta}=X_{t_k}\right]\right)^{\frac{1}{p}} \notag\\
&\leq C\lambda^{\frac{2}{q}}\Delta_n^{1+\frac{1}{q}}\left(1+\vert X_{t_k}\vert^{q_1}\right).\label{m02}
\end{align}

Next, using again Jensen's and H\"older's inequalities, and {\bf(A1)} and {\bf(A6)}, we get 
\begin{align*}
M_{0,3}^{\theta}&\leq\E\left[{\bf 1}_{\{\Delta M_k>0\}}{\bf{1}}_{\widetilde{A}_{k,n}^c}\left(e_k(\theta)\int_{t_k}^{t_{k+1}}
\int_{\R_0^d}c(Y_{t_k}^{\theta},z)M(ds,dz)\right)^2\bigg\vert Y_{t_{k}}^{\theta}=X_{t_{k}}\right]\\
&\leq\left(\E\left[\left(e_k(\theta)\int_{t_k}^{t_{k+1}}\int_{\R_0^d}c(Y_{t_k}^{\theta},z)M(ds,dz)\right)^{2p}\Big\vert Y_{t_k}^{\theta}=X_{t_k}\right]\right)^{\frac{1}{p}}\\
&\qquad\times\left(\P\left(\widetilde{A}_{k,n}^c,\Delta M_k>0\big\vert Y_{t_k}^{\theta}=X_{t_k}\right)\right)^{\frac{1}{q}}\\
&\leq C\Delta_n^{\frac{1}{p}}\left(1+\vert X_{t_k}\vert^{q_1}\right)\left(\P\left(\widetilde{A}_{k,n}^c,\Delta M_k>0\big\vert Y_{t_k}^{\theta}=X_{t_k}\right)\right)^{\frac{1}{q}},
\end{align*}
where $p,q >1$ and $\frac{1}{p}+\frac{1}{q}=1$. On the other hand, using Chebyshev's inequality and {\bf (A6)}, we have that for any $\kappa\geq 1$,
\begin{equation}\label{prob}\begin{split}
&\P\left(\widetilde{A}_{k,n}^c,\Delta M_k>0\big\vert Y_{t_k}^{\theta}=X_{t_k}\right)=\P\left(\left\vert\widetilde{Z}_{t_{k+1}}-\widetilde{Z}_{t_{k}}\right\vert<\rho_1\Delta_n^{\upsilon},\widetilde{J}_{1,k}\right)\\
&\qquad+\P\left(\left\vert\widetilde{Z}_{t_{k+1}}-\widetilde{Z}_{t_{k}}\right\vert<\rho_1\Delta_n^{\upsilon},\widetilde{J}_{2,k}\right)+\P\left(\left\vert\widetilde{Z}_{t_{k+1}}-\widetilde{Z}_{t_{k}}\right\vert>\rho_2\Delta_n^{-\gamma},\Delta M_k>0\right)\\
&\leq \P\left(\left\vert \widetilde{\Lambda}_{[t_k,t_{k+1}]}\right\vert<\rho_1\Delta_n^{\upsilon}\big\vert \widetilde{J}_{1,k}\right)\P(\widetilde{J}_{1,k})+\left(\lambda\Delta_n\right)^2+\left(\rho_2^{-1}\Delta_n^{\gamma}\right)^{\kappa}\E\left[\left\vert\widetilde{Z}_{t_{k+1}}-\widetilde{Z}_{t_{k}}\right\vert^{\kappa}\right]\\
&\leq e^{-\lambda\Delta_n}\lambda\Delta_n\int_{\{\vert z\vert \leq\rho_1\Delta_n^{\upsilon}\}}\mu(dz)+\left(\lambda\Delta_n\right)^2+\left(\rho_2^{-1}\Delta_n^{\gamma}\right)^{\kappa}\int_{t_k}^{t_{k+1}}\int_{\R_0^d}\vert z\vert^{\kappa}\nu(dz)ds\\
&\leq e^{-\lambda\Delta_n}\Delta_n\int_{\{\vert z\vert \leq\rho_1\Delta_n^{\upsilon}\}}\nu(dz)+\left(\lambda\Delta_n\right)^2+C\Delta_n^{\gamma\kappa+1}\\
&\leq e^{-\lambda\Delta_n}\Delta_n\int_{\{\vert z\vert \leq\rho_1\Delta_n^{\upsilon}\}}\nu(dz)+C\left(\lambda\Delta_n\right)^2,
\end{split}
\end{equation}
for some constant $C>0$, where $\kappa$ is chosen in order that $\gamma\kappa+1>2$. Therefore,
\begin{align*}
M_{0,3}^{\theta}&\leq C\Delta_n\left(\left(\int_{\{\vert z\vert \leq\rho_1\Delta_n^{\upsilon}\}}\nu(dz)\right)^{\frac{1}{q}}+\lambda^{\frac{2}{q}}\Delta_n^{\frac{1}{q}}\right)\left(1+\vert X_{t_k}\vert^{q_1}\right),
\end{align*}
which, together with \eqref{m01} and \eqref{m02}, shows (\ref{c3m1}). 

We next show (\ref{c3m2}). As for the term $M_0^{\theta}$, multiplying the random variable inside the conditional expectation of $M_1^{\theta}$ by ${\bf{1}}_{\widetilde{A}_{k,n}}+{\bf{1}}_{\widetilde{A}_{k,n}^c}$ and ${\bf 1}_{\widetilde{J}_{0,k}} +{\bf 1}_{\widetilde{J}_{1,k}}+ {\bf 1}_{\widetilde{J}_{2,k}}$, we have that $M_1^{\theta}\leq 4(M_{1,1}^{\theta}+M_{1,2}^{\theta}+M_{1,3}^{\theta}+M_{1,4}^{\theta})$, where
\begin{align*} 
M_{1,1}^{\theta}&=\E_{\widehat{Q}_k^{\theta,\theta_0}}\bigg[{\bf 1}_{\widehat{J}_{1,k}}\bigg(e_k(\theta)\bigg({\bf{1}}_{\widehat{A}_{k,n}}\int_{t_k}^{t_{k+1}}\int_{\R_0^d}c(X_{t_k},z)N(ds,dz)\\
&\quad\quad\quad-\widetilde{\E}_{X_{t_k}}^{\theta}\left[{\bf 1}_{\widetilde{J}_{1,k}}{\bf{1}}_{\widetilde{A}_{k,n}}\int_{t_k}^{t_{k+1}}\int_{\R_0^d}c(Y_{t_k}^{\theta},z)M(ds,dz)\bigg\vert Y^{\theta}_{t_{k+1}}=X_{t_{k+1}}\right]\bigg)\bigg)^2\Big\vert X_{t_k}\bigg],\\
M_{1,2}^{\theta}&=\E_{\widehat{Q}_k^{\theta,\theta_0}}\left[{\bf 1}_{\widehat{J}_{1,k}}\left(e_k(\theta)\widetilde{\E}_{X_{t_k}}^{\theta}\left[{\bf 1}_{\widetilde{J}_{2,k}}{\bf{1}}_{\widetilde{A}_{k,n}}\int_{t_k}^{t_{k+1}}\int_{\R_0^d}c(Y_{t_k}^{\theta},z)M(ds,dz)\bigg\vert Y^{\theta}_{t_{k+1}}=X_{t_{k+1}}\right]\right)^2\Big\vert X_{t_k}\right],\\
M_{1,3}^{\theta}&=\E_{\widehat{Q}_k^{\theta,\theta_0}}\left[{\bf 1}_{\widehat{J}_{1,k}}{\bf{1}}_{\widehat{A}_{k,n}^c}\left(e_k(\theta)\int_{t_k}^{t_{k+1}}\int_{\R_0^d}c(X_{t_k},z)N(ds,dz)\right)^2\Big\vert X_{t_k}\right],\\
M_{1,4}^{\theta}&=\E_{\widehat{Q}_k^{\theta,\theta_0}}\left[{\bf 1}_{\widehat{J}_{1,k}}\left(e_k(\theta)\widetilde{\E}_{X_{t_k}}^{\theta}\left[{\bf 1}_{\{\Delta M_k>0\}}{\bf{1}}_{\widetilde{A}_{k,n}^c}\int_{t_k}^{t_{k+1}}\int_{\R_0^d}c(Y_{t_k}^{\theta},z)M(ds,dz)\bigg\vert Y^{\theta}_{t_{k+1}}=X_{t_{k+1}}\right]\right)^2\Big\vert X_{t_k}\right].
\end{align*}

To bound $M_{1,3}^{\theta}$, using Jensen's and H\"older's inequalities, and {\bf(A1)}, {\bf(A6)}, we have that
\begin{align*}
M_{1,3}^{\theta}&\leq \left(\E_{\widehat{Q}_k^{\theta,\theta_0}}\left[\left(e_k(\theta)\int_{t_k}^{t_{k+1}}
\int_{\R_0^d}c(X_{t_k},z)N(ds,dz)\right)^{2p}\Big\vert X_{t_k}\right]\right)^{\frac{1}{p}}\left(\P\left(\widehat{A}_{k,n}^c,\widehat{J}_{1,k}\big\vert X_{t_k}\right)\right)^{\frac{1}{q}}d\ell\\
&\leq C\Delta_n^{\frac{1}{p}}\left(1+\vert X_{t_k}\vert^{q_1}\right)\left(\P\left(\widehat{A}_{k,n}^c,\widehat{J}_{1,k}\big\vert X_{t_k}\right)\right)^{\frac{1}{q}},
\end{align*}
where $p,q>1$ and $\frac{1}{p}+\frac{1}{q}=1$. On the other hand, as \eqref{prob}, using Chebyshev's inequality and hypothesis {\bf (A6)}, we have that for any $\kappa\geq 1$,
\begin{equation*}\begin{split}
\P\left(\widehat{A}_{k,n}^c,\widehat{J}_{1,k}\big\vert X_{t_k}\right)&=\P\left(\left\vert\widehat{Z}_{t_{k+1}}-\widehat{Z}_{t_{k}}\right\vert<\rho_1\Delta_n^{\upsilon},\widehat{J}_{1,k}\right)+\P\left(\left\vert\widehat{Z}_{t_{k+1}}-\widehat{Z}_{t_{k}}\right\vert>\rho_2\Delta_n^{-\gamma},\widehat{J}_{1,k}\right)\\
&\leq \P\left(\left\vert \widehat{\Lambda}_{[t_k,t_{k+1}]}\right\vert<\rho_1\Delta_n^{\upsilon}\big\vert \widehat{J}_{1,k}\right)\P(\widehat{J}_{1,k})+\left(\rho_2^{-1}\Delta_n^{\gamma}\right)^{\kappa}\E\left[\left\vert\widehat{Z}_{t_{k+1}}-\widehat{Z}_{t_{k}}\right\vert^{\kappa}\right]\\
&\leq e^{-\lambda\Delta_n}\lambda\Delta_n\int_{\{\vert z\vert \leq\rho_1\Delta_n^{\upsilon}\}}\mu(dz)+\left(\rho_2^{-1}\Delta_n^{\gamma}\right)^{\kappa}\int_{t_k}^{t_{k+1}}\int_{\R_0^d}\vert z\vert^{\kappa}\nu(dz)ds\\
&\leq e^{-\lambda\Delta_n}\Delta_n\int_{\{\vert z\vert \leq\rho_1\Delta_n^{\upsilon}\}}\nu(dz)+C\left(\lambda\Delta_n\right)^2,
\end{split}
\end{equation*}
where $\kappa$ is chosen in order that $\gamma\kappa+1>2$. Therefore, for any $q>1$,
\begin{align*}
M_{1,3}^{\theta}\leq C\Delta_n\left(\left(\int_{\{\vert z\vert \leq\rho_1\Delta_n^{\upsilon}\}}\nu(dz)\right)^{\frac{1}{q}}+\lambda^{\frac{2}{q}}\Delta_n^{\frac{1}{q}}\right)\left(1+\vert X_{t_k}\vert^{q_1}\right).
\end{align*}

Proceeding as for $M_{0,3}^{\theta}$, we also get that for any $q>1$,
\begin{align*}
M_{1,4}^{\theta}\leq C\Delta_n\left(\left(\int_{\{\vert z\vert \leq\rho_1\Delta_n^{\upsilon}\}}\nu(dz)\right)^{\frac{1}{q}}+\lambda^{\frac{2}{q}}\Delta_n^{\frac{1}{q}}\right)\left(1+\vert X_{t_k}\vert^{q_1}\right).
\end{align*}

We next bound $M_{1,1}^{\theta}$. For this, adding and subtracting the term 
$$
{\bf{1}}_{\widehat{A}_{k,n}}\int_{t_k}^{t_{k+1}}\int_{\R_0^d}c(X_{t_k},z)N(ds,dz)\widetilde{\E}_{X_{t_k}}^{\theta}\left[{\bf 1}_{\widetilde{J}_{1,k}}\vert Y^{\theta}_{t_{k+1}}=X_{t_{k+1}}\right]
$$
inside the square, we get that $M_{1,1}^{\theta}\leq 2(M_{1,1,1}^{\theta}+M_{1,1,2}^{\theta})$, where
\begin{align*} 
M_{1,1,1}^{\theta}&=\E_{\widehat{Q}_k^{\theta,\theta_0}}\bigg[{\bf 1}_{\widehat{J}_{1,k}}\bigg(e_k(\theta)\bigg({\bf{1}}_{\widehat{A}_{k,n}}\int_{t_k}^{t_{k+1}}\int_{\R_0^d}c(X_{t_k},z)N(ds,dz)\\
&\quad\quad\quad\quad-{\bf{1}}_{\widehat{A}_{k,n}}\int_{t_k}^{t_{k+1}}\int_{\R_0^d}c(X_{t_k},z)N(ds,dz)\widetilde{\E}_{X_{t_k}}^{\theta}\left[{\bf 1}_{\widetilde{J}_{1,k}}\bigg\vert Y^{\theta}_{t_{k+1}}=X_{t_{k+1}}\right]\bigg)\bigg)^2\Big\vert X_{t_k}\bigg],\\
M_{1,1,2}^{\theta}&=\E_{\widehat{Q}_k^{\theta,\theta_0}}\bigg[{\bf 1}_{\widehat{J}_{1,k}}\bigg(e_k(\theta)\bigg({\bf{1}}_{\widehat{A}_{k,n}}\int_{t_k}^{t_{k+1}}\int_{\R_0^d}c(X_{t_k},z)N(ds,dz)\widetilde{\E}_{X_{t_k}}^{\theta}\left[{\bf 1}_{\widetilde{J}_{1,k}}\bigg\vert Y^{\theta}_{t_{k+1}}=X_{t_{k+1}}\right]\\
&\quad\quad\quad\quad-\widetilde{\E}_{X_{t_k}}^{\theta}\left[{\bf 1}_{\widetilde{J}_{1,k}}{\bf{1}}_{\widetilde{A}_{k,n}}\int_{t_k}^{t_{k+1}}\int_{\R_0^d}c(Y_{t_k}^{\theta},z)M(ds,dz)\bigg\vert Y^{\theta}_{t_{k+1}}=X_{t_{k+1}}\right]\bigg)\bigg)^2\Big\vert X_{t_k}\bigg].
\end{align*}

Observe that $M_{1,1,1}^{\theta}\leq 2(M_{1,1,1,0}^{\theta}+M_{1,1,1,2}^{\theta})$, where for $i\in\{0,2\}$,
\begin{align*} 
M_{1,1,1,i}^{\theta}&=\E_{\widehat{Q}_k^{\theta,\theta_0}}\left[{\bf 1}_{\widehat{J}_{1,k}}\left(e_k(\theta){\bf{1}}_{\widehat{A}_{k,n}}\int_{t_k}^{t_{k+1}}\int_{\R_0^d}c(X_{t_k},z)N(ds,dz)\widetilde{\E}_{X_{t_k}}^{\theta}\left[{\bf 1}_{\widetilde{J}_{i,k}}\bigg\vert Y^{\theta}_{t_{k+1}}=X_{t_{k+1}}\right]\right)^2\Big\vert X_{t_k}\right].
\end{align*}

By (\ref{c32}),
\begin{align*}
M_{1,1,1,0}^{\theta}&=\int_{I}\int_{\R^d}\left(\dfrac{q_{(0)}^{\theta}(\Delta_n,X_{t_k},y)e^{-\lambda\Delta_n}}{p^{\theta}(\Delta_n,X_{t_k},y)}\right)^2q_{(1)}^{\theta}(\Delta_n,X_{t_k},y;a)e^{-\lambda \Delta_n}\lambda \Delta_n\left(e_k(\theta)c(X_{t_k},a)\right)^2dy\mu(da).
\end{align*}

Again we divide the $dy$ integral into the subdomains $J_1:=\{y\in\R^d:\vert f(y)-f(X_{t_k})\vert>\frac{\Delta_n^{\alpha}}{(1+\vert X_{t_k}\vert^2)^{\frac{3}{2}}}\}$ and $J_2:=\{y\in\R^d:\vert f(y)-f(X_{t_k})\vert\leq\frac{\Delta_n^{\alpha}}{(1+\vert X_{t_k}\vert^2)^{\frac{3}{2}}}\}$, where $\alpha\in(\upsilon+3m\gamma+3\gamma,\frac 12)$, and call the terms $M_{1,1,1,0,1}^{\theta}$ and $M_{1,1,1,0,2}^{\theta}$. In the same way the term $M_{0,1,1}^{\theta}$ was treated, using \eqref{c3eq28}, \eqref{q0} and hypothesis {\bf(A1)}, we obtain that
\begin{align*}
M_{1,1,1,0,1}^{\theta}&\leq C\Delta_n^{-2m\gamma}\left(1+
\vert X_{t_k}\vert^{q_1}\right) \int_{J_1}q_{(0)}^{\theta}(\Delta_n,X_{t_k},y)dy\\
&\leq C\Delta_n^{-2m\gamma}
\left(1+\vert X_{t_k}\vert^{q_1}\right)e^{-\frac{\Delta_n^{2\alpha-1}}{2c(1+\vert X_{t_k}\vert^2)^3}},
\end{align*}
for some constants $C>0$, $c>1$ and $q_1>1$. Next, \eqref{c3density} yields 
\begin{equation}\label{c3eq30}
p^{\theta}(\Delta_n,X_{t_k},y)\geq q_{(0)}^{\theta}(\Delta_n,X_{t_k},y)e^{-\lambda \Delta_n}.
\end{equation}
Then, as for the term $M_{0,1,2}^{\theta}$, using hypotheses {\bf(A1)}, {\bf(A3)}, {\bf(A6)}, and \eqref{q1}, we get that for $n$ large enough,
\begin{equation*} \begin{split}
M_{1,1,1,0,2}^{\theta}&\leq C\Delta_n^{-2m\gamma}
\left(1+\vert X_{t_k}\vert^{q_1}\right)e^{-\lambda \Delta_n}\lambda \Delta_n \int_{I}\int_{J_2}q_{(1)}^{\theta}(\Delta_n,X_{t_k},y;a)dy\mu(da)\\
&\leq C\Delta_n^{-2m\gamma}
\left(1+\vert X_{t_k}\vert^{q_1}\right)e^{-C_0\frac{\Delta_n^{2(\upsilon+3m\gamma+3\gamma)-1}}{(1+\vert X_{t_k}\vert^2)^3}},
\end{split}
\end{equation*}
for some constants $C, C_0>0$ and $q_1>1$. Therefore, the term $M_{1,1,1,0}^{\theta}$ satisfies (\ref{m01}). 

As for the term $M_{0,2}^{\theta}$, we have that $
M_{1,1,1,2}^{\theta}\leq C\lambda^{\frac{2}{q}}\Delta_n^{1+\frac{1}{q}}(1+\vert X_{t_k}\vert^{q_1})$, for all $q>1$ and for some constants $C>0$, $q_1>1$. 

We next treat $M_{1,1,2}^{\theta}$. Using (\ref{c33}), we have that 
\begin{equation*} \begin{split}
M_{1,1,2}^{\theta}&=\int_{I}\int_{\R^d}\left(\dfrac{e_k(\theta)\int_{I}\left(c(X_{t_k},z)-c(X_{t_k},a)\right)q_{(1)}^{\theta}(\Delta_n,X_{t_k},y;a)\mu(da)e^{-\lambda \Delta_n}\lambda \Delta_n}{p^{\theta}(\Delta_n,X_{t_k},y)}\right)^2\\
&\qquad \qquad \qquad \times q_{(1)}^{\theta}(\Delta_n,X_{t_k},y;z)e^{-\lambda \Delta_n}\lambda\Delta_ndy\mu(dz).
\end{split}
\end{equation*}

We next fix $\alpha_0\in(\frac{1}{4},\frac{1}{2}-3\gamma)$ and let $\varepsilon\in(0,\alpha_0-3m\gamma)$. Consider the set
\begin{align*}
E^k_z=\left\{a\in I: \left\vert c(X_{t_k},z)-c(X_{t_k},a)\right\vert\leq\Delta_n^{\varepsilon},\ \text{for all}\ z\in I\right\}.
\end{align*}
We next split the integral inside the square of $M_{1,1,2}^{\theta}$ over the sets ${\bf{1}}_{E^k_z}$ and ${\bf{1}}_{(E^k_z)^c}$ and call both terms $M_{1,1,2,1}^{\theta}$ and $M_{1,1,2,2}^{\theta}$. First, \eqref{c3eq25}, \eqref{q1} and hypothesis {\bf(A6)} yield that
\begin{align}
M_{1,1,2,1}^{\theta}&\leq Ce^{-\lambda \Delta_n}\lambda\Delta_n^{1+2\varepsilon}\left(1+\vert X_{t_k}\vert^{q_1}\right)\int_{I}\int_{\R^d}q_{(1)}^{\theta}(\Delta_n,X_{t_k},y;z)dy\mu(dz) \notag\\
&\leq C\Delta_n^{1+2\varepsilon}\left(1+\vert X_{t_k}\vert^{q_1}\right)\int_{I}\int_{\R^d}\dfrac{\beta^d(z)}{\eta(z)\Delta_n^{d/2}}e^{-\frac{\left\vert f(y)-f(X_{t_k})-\widetilde{c}(f(X_{t_k}),z)\right\vert^2}{C_2\beta^2(z)\Delta_n}}\det\nabla f(y)dy\nu(dz) \notag\\
&\leq C\Delta_n^{1+2\varepsilon}\left(1+\vert X_{t_k}\vert^{q_1}\right).\label{c3m1121}
\end{align}

Next, to treat $M_{1,1,2,2}^{\theta}$, we divide the domain of the $dy$ integral into two subdomains $I_1:=\{y\in\R^d:\vert f(y)-f(X_{t_k})-\widetilde{c}(f(X_{t_k}),z)\vert>\frac{\Delta_n^{\alpha_0}}{(1+\vert X_{t_k}\vert^2)^{3/2}}\}$ and $I_2:=\{y\in\R^d:\vert f(y)-f(X_{t_k})-\widetilde{c}(f(X_{t_k}),z)\vert\leq\frac{\Delta_n^{\alpha_0}}{(1+\vert X_{t_k}\vert^2)^{3/2}}\}$, and call both terms $M_{1,1,2,2,1}^{\theta}$ and $M_{1,1,2,2,2}^{\theta}$. Then, using hypotheses {\bf(A1)}, {\bf(A6)}, together with \eqref{c3eq25} and \eqref{q1}, we get
that
\begin{align*} 
&M_{1,1,2,2,1}^{\theta}\leq C\Delta_n^{-2m\gamma}
\left(1+\vert X_{t_k}\vert^{q_1}\right)e^{-\lambda \Delta_n}\lambda\Delta_n\int_{I}\int_{I_1} q_{(1)}^{\theta}(\Delta_n,X_{t_k},y;z)dy\mu(dz)\\
&\leq C\Delta_n^{-2m\gamma}\left(1+\vert X_{t_k}\vert^{q_1}\right)\int_{I}\int_{I_1} \dfrac{\beta^d(z)}{\eta(z)\Delta_n^{d/2}}e^{-\frac{\left\vert f(y)-f(X_{t_k})-\widetilde{c}(f(X_{t_k}),z)\right\vert^2}{C_2\beta^2(z)\Delta_n}}\det\nabla f(y)dy\nu(dz)\\
&\leq C\Delta_n^{-2m\gamma}\left(1+\vert X_{t_k}\vert^{q_1}\right)e^{-\frac{\Delta_n^{2\alpha_0+6\gamma-1}}{2C_3(1+\vert X_{t_k}\vert^2)^3}}\int_{I}\int_{I_1} \dfrac{\beta^d(z)}{\eta(z)\Delta_n^{d/2}}e^{-\frac{\left\vert f(y)-f(X_{t_k})-\widetilde{c}(f(X_{t_k}),z)\right\vert^2}{2C_2\beta^2(z)\Delta_n}}\det\nabla f(y)dy\nu(dz)\\
&\leq C\Delta_n^{-2m\gamma}\left(1+\vert X_{t_k}\vert^{q_1}\right)e^{-\frac{\Delta_n^{2\alpha_0+6\gamma-1}}{2C_3(1+\vert X_{t_k}\vert^2)^3}},
\end{align*}
for some constants $C, C_2, C_3>0$. Next, \eqref{c3density} yields 
\begin{equation*} \begin{split}
&\left(p^{\theta}(\Delta_n,X_{t_k},y)\right)^2\geq p^{\theta}(\Delta_n,X_{t_k},y) \int_{I}q_{(1)}^{\theta}(\Delta_n,X_{t_k},y;a)\mu(da)e^{-\lambda \Delta_n}\lambda \Delta_n.
\end{split}
\end{equation*}
Then, using hypothesis {\bf(A1)} and \eqref{q1}, we obtain that
\begin{align*}
&M_{1,1,2,2,2}^{\theta}\leq C\Delta_n^{-2m\gamma}
\left(1+\vert X_{t_k}\vert^{q_1}\right)e^{-\lambda \Delta_n}\lambda\Delta_n\\
&\qquad\times\int_{I}\int_{I_2}\int_{I}{\bf{1}}_{(E^k_z)^c}q_{(1)}^{\theta}(\Delta_n,X_{t_k},y;a)\mu(da)
 \dfrac{q_{(1)}^{\theta}(\Delta_n,X_{t_k},y;z)e^{-\lambda \Delta_n}\lambda \Delta_n}{p^{\theta}(\Delta_n,X_{t_k},y)}dy\mu(dz)\\
&\leq C\Delta_n^{-2m\gamma}
\left(1+\vert X_{t_k}\vert^{q_1}\right)\int_{I}\int_{I_2}\int_{I}{\bf{1}}_{(E^k_z)^c}\dfrac{\beta^d(a)}{\eta(a)\Delta_n^{d/2}}e^{-\frac{\left\vert f(y)-f(X_{t_k})-\widetilde{c}(f(X_{t_k}),a)\right\vert^2}{C_2\beta^2(a)\Delta_n}}\det\nabla f(y)\nu(da)\\
&\qquad\times\dfrac{q_{(1)}^{\theta}(\Delta_n,X_{t_k},y;z)e^{-\lambda \Delta_n}\lambda \Delta_n}{p^{\theta}(\Delta_n,X_{t_k},y)}dy\mu(dz)\\
&\leq C\Delta_n^{-2m\gamma-\frac{d}{2}}\left(1+\vert X_{t_k}\vert^{q_1}\right)\int_{I}\int_{\{\vert h\vert\leq\frac{\Delta_n^{\alpha_0}}{(1+\vert X_{t_k}\vert^2)^{3/2}}\}}\int_{I}{\bf{1}}_{(E^k_z)^c}\dfrac{\beta^d(a)}{\eta(a)}e^{-\frac{\left\vert h+\widetilde{c}(f(X_{t_k}),z)-\widetilde{c}(f(X_{t_k}),a)\right\vert^2}{C_3\Delta_n^{1-6\gamma}}}\nu(da)\\
&\qquad\times \dfrac{q_{(1)}^{\theta}(\Delta_n,X_{t_k},y;z)e^{-\lambda \Delta_n}\lambda \Delta_n}{p^{\theta}(\Delta_n,X_{t_k},y)}\det\nabla f(y)dy\mu(dz),
\end{align*}
for some constants $C, C_2, C_3>0$, where we set $h:=f(y)-f(X_{t_k})-\widetilde{c}(f(X_{t_k}),z)$. 

Next, using the same arguments as in \eqref{minoration}, we get that
\begin{align*}
\left\vert\widetilde{c}(f(X_{t_k}),z)-\widetilde{c}(f(X_{t_k}),a)\right\vert^2&=\left\vert f\left(X_{t_k}+c(X_{t_k},z)\right)-f\left(X_{t_k}+c(X_{t_k},a)\right)\right\vert^2\\
&\geq \dfrac{C\Delta_n^{6m\gamma}}{\left(1+\vert X_{t_k}\vert^2\right)^{3}}\left\vert c(X_{t_k},z)-c(X_{t_k},a)\right\vert^2\\
&\geq \dfrac{C\Delta_n^{2(\varepsilon+3m\gamma)}}{\left(1+\vert X_{t_k}\vert^2\right)^{3}},
\end{align*}
for some constant $C>0$, since $\vert c(X_{t_k},z)-c(X_{t_k},a)\vert>\Delta_n^{\varepsilon}$ on $(E^k_z)^c$. Here, we have used the following estimate, by {\bf(A1)}, 
\begin{align*}
1+\vert X_{t_k}+\eta (c(X_{t_k},z)-c(X_{t_k},a))\vert^2&\leq 1+2\left(\vert X_{t_k}\vert^2+\left\vert c(X_{t_k},z)-c(X_{t_k},a)\right\vert^2\right)\\
&\leq C\left(1+\vert X_{t_k}\vert^2+\left(1+\vert X_{t_k}\vert^2\right)(\vert z\vert^{2m}+\vert a\vert^{2m})\right)\\
&\leq C\Delta_n^{-2m\gamma}\left(1+\vert X_{t_k}\vert^2\right).
\end{align*}

Thus, using $\vert h\vert\leq \frac{\Delta_n^{\alpha_0}}{(1+\vert X_{t_k}\vert^2)^{\frac{3}{2}}}$ and $\varepsilon+3m\gamma<\alpha_0$, we deduce that for $n$ large enough,
\begin{align*}
\left\vert h+\widetilde{c}(f(X_{t_k}),z)-\widetilde{c}(f(X_{t_k}),a)\right\vert^2&\geq\dfrac{\left\vert\widetilde{c}(f(X_{t_k}),z)-\widetilde{c}(f(X_{t_k}),a)\right\vert^2}{2}-\vert h\vert^2\\
&\geq\dfrac{C\Delta_n^{2(\varepsilon+3m\gamma)}}{2(1+\vert X_{t_k}\vert^2)^{3}}-\dfrac{\Delta_n^{2\alpha_0}}{(1+\vert X_{t_k}\vert^2)^{3}}\\
&\geq \dfrac{C_4\Delta_n^{2(\varepsilon+3m\gamma)}}{(1+\vert X_{t_k}\vert^2)^{3}},
\end{align*}
for some constant $C_4>0$. Therefore, using \eqref{c3eq25} and $\int_{I}\frac{\beta^d(a)}{\eta(a)}\nu(da)<\infty$, for $n$ large enough,
\begin{align*}
M_{1,1,2,2,2}^{\theta}&\leq C\Delta_n^{-2m\gamma-\frac{d}{2}}e^{-\frac{C_4\Delta_n^{2\left(\varepsilon+3m\gamma+3\gamma\right)-1}}{C_3(1+\vert X_{t_k}\vert^2)^3}}\left(1+\vert X_{t_k}\vert^{q_1}\right)\\
&\qquad\times \int_{I}\int_{\{\vert h\vert\leq \Delta_n^{\alpha_0}\}}\dfrac{q_{(1)}^{\theta}(\Delta_n,X_{t_k},y;z)e^{-\lambda \Delta_n}\lambda \Delta_n}{p^{\theta}(\Delta_n,X_{t_k},y,z))}\det\nabla f(y)dy\mu(dz)\\
&\leq C\Delta_n^{-2m\gamma-\frac{d}{2}}e^{-\frac{C_4\Delta_n^{2\left(\varepsilon+3m\gamma+3\gamma\right)-1}}{C_3(1+\vert X_{t_k}\vert^2)^3}}\left(1+\vert X_{t_k}\vert^{q_1}\right)\\
&\qquad\times \int_{\{\vert u\vert \leq\Delta_n^{\alpha_0}+1\}}\dfrac{\int_{I}q_{(1)}^{\theta}(\Delta_n,X_{t_k},f^{-1}(u);z)\mu(dz)e^{-\lambda \Delta_n}\lambda \Delta_n}{p^{\theta}(\Delta_n,X_{t_k},f^{-1}(u))}du\\
&\leq C\Delta_n^{-2m\gamma-\frac{d}{2}}e^{-\frac{C_4\Delta_n^{2\left(\varepsilon+3m\gamma+3\gamma\right)-1}}{C_3(1+\vert X_{t_k}\vert^2)^3}}\left(1+\vert X_{t_k}\vert^{q_1}\right), 
\end{align*}
where we have used the change of variables $u:=f(y)$, and $f^{-1}$ is the inverse function of $f$.

Since $\alpha_0>\varepsilon+3m\gamma$, we deduce that for any $\alpha_0\in(\frac{1}{4},\frac{1}{2}-3\gamma)$ and $n$ large enough,
\begin{equation*} 
M_{1,1,2,2}^{\theta}\leq C \Delta_n^{-2m\gamma-\frac{d}{2}}
\left(1+\vert X_{t_k}\vert^{q_1}\right)e^{-C_1\frac{\Delta_n^{2\alpha_0+6\gamma-1}}{(1+\vert X_{t_k}\vert^2)^3}},
\end{equation*}
for some constants $C, C_1>0$, which together with (\ref{c3m1121}) gives
\begin{equation*} 
M_{1,1,2}^{\theta}\leq C\left(1+\vert X_{t_k}\vert^{q_1}\right) \left(\Delta_n^{1+2\varepsilon}+ \Delta_n^{-2m\gamma-\frac{d}{2}}e^{-C_1\frac{\Delta_n^{2\alpha_0+6\gamma-1}}{(1+\vert X_{t_k}\vert^2)^3}}\right),
\end{equation*}
for any $\alpha_0\in(\frac{1}{4},\frac{1}{2}-3\gamma)$, $\varepsilon\in(0,\alpha_0-3m\gamma)$ and $n$ large enough.

Finally, as for $M_{0,2}^{\theta}$, we obtain that $M_{1,2}^{\theta}+M_{2}^{\theta}\leq C\lambda^{\frac{2}{q}}\Delta_n^{1+\frac{1}{q}}(1+\vert X_{t_k}\vert^{q_1})$, for all $q>1$ and for some constants $C>0$, $q_1>1$, which concludes the proof of (\ref{c3m2}) and (\ref{c3m3}).
\end{proof}


\begin{thebibliography}{00}

\bibitem{A84} Azencott, R. (1984), Densit\'e des diffusions en temps petit: D\'eveloppements asymptotiques, In {\em Seminar on probability, {XVIII}}, volume~1059 of {\em Lecture Notes in Math.}, 402--498, Springer, Berlin.

\bibitem{AJ07} A{\"{\i}}t-Sahalia, Y. and Jacod, J. (2007), Volatility estimators for discretely sampled {L}\'evy processes, {\it Ann. Statist.}, {\bf 35}(1), 355-392.

\bibitem{B03} Bally, V. (2003), {\it An elementary introduction to Malliavin calculus}, Rapport de recherche 4718, INRIA.

\bibitem{BGJ87} Bichteler, K., Gravereaux, J.B. and Jacod, J. (1987), {\it Malliavin calculus for processes with jumps}, volume~2 of Stochastics Monographs, Gordon and Breach Science Publishers, New York.

\bibitem{CDG14} Cl\'ement, E., Delattre, S. and Gloter, A. (2014), Asymptotic lower bounds in estimating jumps, {\it Bernoulli}, {\bf 20}(3), 1059-1096.

\bibitem{CG15} Cl\'ement, E. and Gloter, A. (2015), Local Asymptotic Mixed Normality property for discretely observed stochastic differential equations driven by stable L\'evy processes, {\it Stochastic Processes and their Applications}, {\bf 125}, 2316-2352.

\bibitem{GJ93} Genon-Catalot, V. and Jacod, J. (1993), On the estimation of the diffusion coefficient for multi-dimensional diffusion processes, {\it Ann. Inst. H. Poincar\'e (Probab. Statist.)}, {\bf 29}, 119-151.

\bibitem{G01} Gobet, E. (2001), Local asymptotic mixed normality property for elliptic diffusions: a Malliavin calculus approach, {\it Bernoulli}, {\bf 7}, 899-912.

\bibitem{G02} Gobet, E. (2002), LAN property for ergodic diffusions with discrete observations, {\it Ann. I. H. Poincar\'e}, {\bf 38}, 711-737.

\bibitem{G3} Gloter, A. and Gobet, E. (2008), LAMN property for hidden processes: The case of integrated diffusions, {\it Annales de l'Institut Henri Poincar\'e - Probabilit\'es et Statistiques}, {\bf 44}, 104-128.

\bibitem{Haj72} H{\'a}jek, J. (1972), Local asymptotic minimax and admissibility in estimation, {\it Proceedings of the {S}ixth {B}erkeley {S}ymposium on {M}athematical {S}tatistics and {P}robability ({U}niv. {C}alifornia, {B}erkeley, {C}alif., 1970/1971), {V}ol. {I}: {T}heory of statistics}, 175--194.

\bibitem{J11} Jacod, J. (2011), {\it Statistics and high frequency data}. Lecture Notes in The Fourth European Summer School in Financial Mathematics.

\bibitem{JS03} Jacod, J. and Shiryaev, A.N. (2003), {\it Limit theorems for stochastic processes}, Second Edition, Springer-Verlag, Berlin.

\bibitem{JP82} Jeganathan, P. (1982), On the asymptotic theory of estimation when the limit of the log-likelihood ratios is mixed normal, {\it Sankhy\=a Ser. A}, {\bf 44}(2), 173--212.

\bibitem{K13} Kawai, R. (2013), Local Asymptotic Normality Property for Ornstein-Uhlenbeck Processes with Jumps Under Discrete Sampling, {\it J Theor Probab}, {\bf 26}, 932-967.

\bibitem{KM} Kawai, R. and Masuda, H. (2013), Local asymptotic normality for normal inverse {G}aussian {L}\'evy processes with high-frequency sampling, {\it ESAIM Probab. Stat.}, {\bf 17}, 13-32.

\bibitem{K} Kessler, M. (1997), Estimation of an ergodic diffusion from discrete observations, {\it Scandinavian J. Statist.}, {\bf 24}, 211-229.

\bibitem{KNT14} Kohatsu-Higa, A., Nualart, E. and Tran, N.K. (2014), LAN property for a simple L\'evy process, {\it C. R. Acad. Sci. Paris, Ser. I}, {\bf 352}(10), 859-864.

\bibitem{K09} Kulik, A.M. (2009), Exponential ergodicity of the solutions to SDE's with a jump noise, {\it Stochastic Processes and their Applications}, {\bf 119}, 602-632.

\bibitem{K97} Kunita, H. (1997), {\it Stochastic Flows and Stochastic Differential Equations. Cambridge Studies in Advanced Mathematics}, Cambridge: Cambridge Univ. Press.

\bibitem{KS85} Kusuoka, S. and Stroock, D. (1985), Applications of the Malliavin calculus, Part II, {\it J. Fac. Sci. Univ. Tokyo Sect. IA, Math}, {\bf 32}, 1-76.

\bibitem{LC60} Le~Cam, L. (1960), Locally asymptotically normal families of distributions, {\it Univ. California, Publ. Statist}, {\bf 3}, 37-98.

\bibitem{CY90} Le~Cam, L. and Lo Yang, G. (1990), {\it Asymptotics in statistics: Some basic concepts}, Springer Series in Statistics. Springer-Verlag, New York.

\bibitem{M14} Mai, H. (2014), Efficient maximum likelihood estimation for L\'evy-driven Ornstein-Uhlenbeck processes, {\it Bernoulli}, {\bf 20}(2), 919-957.

\bibitem{M07} Masuda, H. (2007), Ergodicity and exponential $\beta$-mixing bounds for multidimensional diffusions with jumps, {\it Stochastic Processes and their Applications}, {\bf 117}, 35-56.

\bibitem{M08} Masuda, H. (2008), On stability of diffusions with compound-Poisson jumps, {\it Bulletin of Informatics and Cybernetics}, {\bf 40}, 60-74.

\bibitem{M13}
Masuda, H. (2013), Convergence of Gaussian quasi-likelihood random fields for ergodic L\'evy driven SDE observed at high frequency, {\it Ann. Statist.}, {\bf 41}, 1593-1641. 

\bibitem{MT} Meyn, S.P. and Tweedie, R.L. (1993), Stability of Markovian Processes III: Foster-Lyapunov Criteria for Continuous-Time Processes, {\it Advances in Applied Probability}, {\bf 25}, 518-548.




\bibitem{N} Nualart, D. (2006), {\it The Malliavin Calculus and Related Topics}, Second Edition, Springer.

\bibitem{OY}
Ogihara, T.
	and Yoshida, N., (2011),
	Quasi-likelihood analysis for the stochastic differential equation with jumps,
	{\it Statistical Inference for Stochastic Processes},
	{\bf 14}, 189-229.

	


\bibitem{P08} Petrou, E. (2008), Malliavin {C}alculus in {L}\'evy spaces and {A}pplications to {F}inance, {\it Electron. J. Probab.}, {\bf 13}, 852-879.

\bibitem{Q14} Qiao, H. (2014), Exponential {E}rgodicity for {SDE}s with {J}umps and {N}on-{L}ipschitz {C}oefficients, {\it J. Theoret. Probab.}, {\bf 27}(1), 137-152.

\bibitem{SK} Sato, K. (1999), {\it L\'evy Processes and Infinitely Divisible Distributions}, Cambridge University Press, Cambridge.

\bibitem{Sh06} Shimizu, Y. (2006), $M$-Estimation for Discretely Observed Ergodic Diffusion Processes with Infinitely many Jumps, {\it Statistical Inference for Stochastic Processes}, {\bf 9}, 179-225.

\bibitem{SY06} {Shimizu, Y. and Yoshida, N. (2006), Estimation of Parameters for Diffusion Processes with Jumps from Discrete Observations, {\it Stat. Inference Stoch. Process.}, {\bf 9}(3), 227-277.}

\bibitem{trabs} {Trabs, M. (2015), Information bounds for inverse problems with application to deconvolution and L\'evy models, {\it Ann. Inst. H. Poincar\'e (Probab. Statist.)}, {\bf 51}(4), 1620-1650.}


\bibitem{W} Woerner, J.H.C. (2003), Local asymptotic normality for the scale parameter of stable processes, {\it Statist. Probab. Lett.}, {\bf 63}(1), 61-65.
\end{thebibliography}
\end{document}